\documentclass[10pt]{amsart}

\usepackage{geometry,graphicx,amssymb,amsmath,amsbsy,eucal,amsfonts,mathrsfs,amscd,bm,esint,yhmath}
\usepackage[all]{xy}
\usepackage{tikz}
\usepackage{multirow}
\usetikzlibrary{arrows,shapes}
\usetikzlibrary{arrows, decorations.markings,fit}
\usetikzlibrary{calc,3d}
\usepackage{tikz-3dplot}
\usepackage{subcaption}

\usepackage{soul}

\tikzset{
    >=stealth',
    punkt/.style={
           rectangle,
           rounded corners,
           draw=black, very thick,
           text width=6.5em,
           minimum height=2em,
           text centered},
    pil/.style={
           ->,
           thick,
           shorten <=2pt,
           shorten >=2pt,}
}

\usepackage{algorithm2e}
\usepackage{listings}

\numberwithin{equation}{section}

\allowdisplaybreaks[3]

\newtheorem{theorem}{Theorem}[section]
\newtheorem{lemma}[theorem]{Lemma}
\newtheorem{assumption}[theorem]{Assumption}

\theoremstyle{definition}
\newtheorem{definition}[theorem]{Definition}
\theoremstyle{remark}
\newtheorem{remark}[theorem]{Remark}

\theoremstyle{exercise}

\newcommand{\p}{{\partial}}

\newcommand{\nab}{\nabla}

\newcommand{\jump}[1]{\left[\hspace{-0.025in}\left[#1\right]\hspace{-0.025in}\right]}

\newcommand{\les}{\lesssim}

\usepackage{yfonts}

\newcommand{\bld}[1]{\boldsymbol{#1}}

\newcommand{\bI}{{\bm I}}

\newcommand{\bv}{\bld{v}}
\newcommand{\bw}{\bld{w}}

\newcommand{\bp}{\bld{p}}
\newcommand{\bn}{\bld{n}}
\newcommand{\bu}{\bld{u}}

\newcommand{\bV}{\bld{V}}

\newcommand{\bPi}{\bld{\Pi}}
\newcommand{\bz}{\boldsymbol{z}}

\newcommand{\bH}{\bld{H}}

\newcommand{\bx}{\bld{x}}
\newcommand{\bC}{\bld{C}}
\newcommand{\bL}{\bld{L}}

\newcommand{\bX}{{\bm X}}

\newcommand{\bPhi}{{\bm \Phi}}
\newcommand{\bvarphi}{\bm \varphi}

\newcommand{\calV}{\mathcal{V}}
\newcommand{\calE}{\mathcal{E}}

\newcommand{\bbR}{\mathbb{R}}

\newcommand{\calS}{\mathcal{S}}
\newcommand{\calT}{\mathcal{T}}

\newcommand{\bbP}{\mathbb{P}}

\newcommand{\bPsi}{{\bm \Psi}}

\newcommand{\calN}{\mathcal{N}}
\newcommand{\bnu}{{\bm \nu}}
\newcommand{\calP}{\mathcal{P}}
\newcommand{\calM}{\mathcal{M}}

\newcommand{\bImac}{{\bm I}^{\rm m}}
\newcommand{\bIct}{{\bm I}^{\rm ct}}

\def\R{{\mathbb{R}}}

\newcommand{\calO}{\mathcal{O}}
\newcommand{\bbN}{\mathbb{N}}

\newcommand{\pt}{\wideparen}

\makeatletter
\def\widebreve{\mathpalette\wide@breve}
\def\wide@breve#1#2{\sbox\z@{$#1#2$}%
     \mathop{\vbox{\m@th\ialign{##\crcr
\kern0.08em\brevefill#1{0.8\wd\z@}\crcr\noalign{\nointerlineskip}%
                    $\hss#1#2\hss$\crcr}}}\limits}
\def\brevefill#1#2{$\m@th\sbox\tw@{$#1($}%
  \hss\resizebox{#2}{\wd\tw@}{\rotatebox[origin=c]{90}{\upshape(}}\hss$}
\makeatletter

\newcommand{\ipt}{\breve}

\newcommand{\Tmac}{{T^{\rm m}}}
\newcommand{\Kmac}{{K^{\rm m}}}
\newcommand{\Tmaca}{T^{\rm m}_a}
\newcommand{\calTmac}{\mathcal{T}^{\rm m}}
\newcommand{\Gammamac}{\Gamma^{\rm m}}
\newcommand{\bPsimac}{{\bm \Psi}_{\rm m}}
\newcommand{\bPsict}{{\bm \Psi}_{\rm ct}}
\newcommand{\calVmac}{\mathcal{V}^{\rm m}}
\newcommand{\calEmac}{\mathcal{E}^{\rm m}}

\newcommand{\amac}{a^{\rm m}}
\newcommand{\bvmac}{{\bm v}^{\rm m}}
\newcommand{\bvhatmac}{\hat{\bv}^{\rm m}}
\newcommand{\overbvmac}{\bar{\bv}^{\rm m}}

\newcommand{\bVmac}{{\bm V}^{\rm m}}
\newcommand{\Qmac}{Q^{\rm m}}
\newcommand{\bVhatmac}{\widehat{\bm V}^{\rm m}}
\newcommand{\Qhatmac}{\widehat{Q}^{\rm m}}
\newcommand{\bVtildemac}{\widetilde{\bm V}^{\rm m}}
\newcommand{\Qtildemac}{\widetilde{Q}^{\rm m}}

\title[Surface Scott-Vogelius]{A Divergence-Free Scott-Vogelius Finite Element Method
for the Surface Stokes Problem}

    \author[Y. Kone ]{Yerim Kone}
\address{\textsuperscript{\textdagger} Department of Mathematics, University of Pittsburgh, Pittsburgh, PA 15260, USA}
    \email{YEK52@pitt.edu}
    \email{neilan@pitt.edu}
    \email{DAP180@pitt.edu}

	 \thanks{This work was supported in part by the NSF, grant no.~DMS-2309425}
    
	\author[M. Neilan]{Michael Neilan}

    \author[D. Poling]{David Poling}

	\thanks{}

\begin{document}

\maketitle

\begin{abstract}
We construct and analyze an 
exactly divergence-free Scott-Vogelius
finite element method for the surface Stokes problem.
The proposed scheme simultaneously enforces
the tangentiality and incompressibility constructs
exactly and has the same number of unknowns
as the two-dimensional Euclidean discretization.
Our construction extends the surface finite element framework of \cite{DemlowNeilan24,DemlowNeilan25} to 
Scott--Vogelius discretizations defined on curved Clough--Tocher triangulations.
In contrast to previous isoparametric Scott--Vogelius methods based on macro-element constructions, the present approach defines the finite element spaces directly on the refined surface triangulation, leading to a substantially simpler and more practical implementation. 
We prove inf-sup stability of the method
and derive optimal-order convergence  in the isoparametric regime.
\end{abstract}

\section{Introduction}

The stationary surface Stokes problem reads
\begin{equation}\label{eqn:SurfaceStokes}
\begin{aligned}
- \bPi {\rm div}_\gamma {\rm Def}_\gamma \bu + \nab_\gamma p + \bu & = {\bm f}\qquad &&\text{on }\gamma,\\
{\rm div}_{\Gamma} \bu & = 0\qquad &&\text{on }\gamma,
\end{aligned}
\end{equation}
along with the pressure constant $\int_\gamma p = 0$
and the velocity tangentiality constrant $\bu\cdot \bnu = 0$. Here,
we assume $\gamma\subset \bbR^3$ is a smooth, closed and orientable two-dimensional surface
with outward unit normal $\bnu$.

Due its many applications in, e.g., lipid bilayers,
thin films, and surfactent-laden interfaces \cite{EdwardsEtal91,NVW12,ReutherVoigt15},
the construction of 
finite element methods for the surface 
(Navier)-Stokes \eqref{eqn:SurfaceStokes}
has become an active and burgeoning area of research.
Similar to the Euclidean setting, numerical schemes
based on the standard velocity-pressure formulation
must satisfy a compatibility criterion of the two discrete spaces
to ensure stability of the method, namely, a discrete inf-sup condition.
However, the direct use of standard Euclidean finite element
pairs is not straightforward due to the tangential velocity
constraint $\bu\cdot \bnu=0$. In particular, this constraint is, in a sense,
incompatible with conventional $\bH^1$-conforming vector-valued
spaces, since strongly enforcing tangentiality of a continuous
vector field on a discrete (Lipschitz) geometry may lead to a locking
 phenomenon that degrades the approximation
properties of the finite element space.

As such, there are generally two approaches
to construct convergent surface finite element methods
for \eqref{eqn:SurfaceStokes} based on the velocity-pressure formulation.
In the first class of methods, discrete spaces developed for the Euclidean formulation are adapted to the surface setting, and the tangential constraint
is enforced weakly either by penalization or by introducing Lagrange multipliers.
\cite{Fries18,HansboLarsonLarsson20,OQRY18,Maxim19}.
While such schemes may lead to convergent
methods, they often require user-defined penalty parameters and involve
superfluous degrees of freedom, since the normal component of the velocity
approximation remains an unknown. In addition, achieving optimal-order convergence
may require a high-order approximation of the surface, particularly 
for low-order discretizations \cite{HP23}.
Consequently, this family of schemes
may lead to sub-optimal convergence rates on solution-dependent evolving geometries
or in situations where the surface is only partially known.

The second class of finite element methods
takes the opposite approach and strongly enforce
tangentiality and relax the $\bH^1$-conformity
of the velocity approximation. For example,
the methods given in \cite{SurfaceStokes1,SurfaceStokes2,BruersEtal26,Kilicer2025}
approximate the velocity field by the $\bH({\rm div})$-conforming
BDM space and use interior penalty or HDG stabilization techniques to compensate
the lack of conformity of the scheme.  These methods
exactly enforce both the tangentiality and divergence-free
constraint in \eqref{eqn:SurfaceStokes},
but they lead to a relatively large algebraic system compared to
$\bH^1$-conforming schemes. In addition, as far as we are aware,
a complete stability and convergence analysis is missing from the literature.

Another approach within this second class,
and the one adopted in the present article,
    is to modify classical Euclidean Stokes pairs so that they 
    exactly enforce the tangenality constraint at the expense 
    of $\bH^1$-conformity \cite{DemlowNeilan24,DemlowNeilan25},
    while still retaining $\bH({\rm div})$-conformity. 
    In particular, the velocity space is constructed
    by utilizing an auxiliary Piola transform to communicate
    function values at interelement Lagrange degrees of freedom.    
    Unlike the BDM-based discretizations, the resuling discrete spaces 
    possess
    inherent weak continuity properties and therefore
    do not require ad hoc stabilization or interelement integration
    in the discrete bilinear forms. Moreover, these methods
     have the same
    dimensions as their two-dimensional Euclidean counterpart
    and do not require high-order approximations of the outward unit
    normal to obtain optimal-order convergence across all polynomial degrees.
The approach given in \cite{DemlowNeilan24} and \cite{DemlowNeilan25}
has been applied to the Euclidean (low-order) MINI element
and the (high-order) Taylor-Hood pair, respectively. These classical
finite element spaces only enforce the divergence-free constraint
weakly and this property is inherited in the corresponding surface
discretizations.

Recently, there has been significant interest
in discretizations that exactly enforce
the divergence-free constraint
in (Navier)-Stokes models,
as such methods lead to several inherent advantages \cite{JohnEtal17,SZhang05,ChristHu18,GuzmanNeilan14,FalkNeilan13,GuzmanNeilan18,AlfeldSorokina16}. In particular, they exactly preserve
conservation laws at the discrete level for any mesh parameter
and are pressure-robust, i.e., the velocity approximation is independent
of the pressure variable with greater robustness with respect
to the model parameters \cite{JohnEtal17}. While the construction 
of exactly divergence--free finite element spaces
is now a flourishing research area for Euclidean fluid models,
substantially less progress has been made in the surface setting.
Aside from the BDM-based methods
\cite{SurfaceStokes1,SurfaceStokes2},
all of the aforementioned surface methods impose the incompressibility 
constraint only weakly and consequently do not produce exactly divergence--free velocity
fields. This motivates the development of surface finite element methods
that simultaneously enforce both the tangentility and 
incompressibility constraints exactly, while still retainining 
the approximation and stability properties of classical Euclidean Stokes pairs.

The goal of this paper is to extend
the constructions of \cite{DemlowNeilan24,DemlowNeilan25}
to divergence-free pairs, with a focus on the Scott-Vogelius
pair \cite{SV85,GuzmanScott19}. In the Euclidean setting
the velocity space consists of continuous, piecewise polynomials
of degree $r$, whereas the pressure space is given by discontinuous
piecewise polynomials of degree $(r-1)$ (in the absence of singular vertices).
In the planar case, this pair is stable provided the polynomial degree
satisfies $r\ge 4$ and the triangulation does not contain
``nearly singular'' vertices. To reduce these restrictive
polynomial-degree and mesh constraints, a common strategy,
and the one adopted here, 
is to implement
the scheme on Clough-Tocher triangulations, 
 i.e., triangulations obtained by connecting the barycenter of
each triangle to its vertices \cite{GuzmanNeilan18,QinThesis}. 
On such triangulations, the polynomial degree restriction is relaxed
to $r\ge 2$ and the presence of (nearly) singular vertices is eliminated.
We extend these results to the surface cases to construct
divergence-free and exactly tangential methods, using 
the construction in \cite{DemlowNeilan24,DemlowNeilan25}.

While the implementation of Scott-Vogelius pairs
on split (Clough-Tocher) triangulations is 
now relatively standard, an alternative and more recent perspective
is to view the spaces as macro-element finite elements
 \cite{NeilanOtus21,DurstNeilan24}.
In this approach, the global finite element space is built on
a generic shape-regular simplicial mesh, whereas the local spaces
are defined with respect to a subdivision of each triangle
into subsimplices. Thus, the reference basis functions
used in the implementation are piecewise polynomials
relative to a locally split reference simplex.
On affine meshes, this macro-element viewpoint is equivalent
to the refined-triangulation construction, 
in the sense that both yield the same global finite element space.
However, the two approaches differ on meshes with curved edges
or faces.

Scott-Vogelius pairs in the Euclidean isoparametric setting,
based on the macro-element definition, were constructed
and analyzed in \cite{NeilanOtus21,DurstNeilan24}. This 
viewpoint naturally aligns with the underlying stability
theory, which itself is based on a macro-element argument.
However, despite its advantages for the stability analysis,
the macro-element construction leads to relatively 
complicated finite element spaces that
are arduous to implement in practice.
        In particular, the spaces
    are defined by mapping {\em piecewise} polynomial
    basis functions on the reference triangle. As a result,
    most standard finite element software does not support this construction,
    as it requires specialized quadrature rules and data structures 
(but see \cite{BrubeckKirby26} for recent progress in this direction).

Unlike the Euclidean isoparametric methods  \cite{NeilanOtus21,DurstNeilan24},
in this paper we define the Scott-Vogelius pair directly 
on a (curved) surface Clough-Tocher triangulation
rather than through a macro-element construction. 
This viewpoint leads to a simpler and more practical scheme
than its macro-element counterpart. In particular,
the proposed velocity space is the same as the Taylor-Hood
space given in \cite{DemlowNeilan25} (but on a Clough-Tocher
triangulation). The tradeoff, however, is a significantly 
more delicate stability analysis. 
To establish stability of the proposed method,
we quantify the discrepancy between the macro-element and non-macro
element constructions and then employ a perturbation
argument to transfer stability from the former to the latter.
As a byproduct of the analysis,
    we also prove the stability of the corresponding macro Scott-Vogelius pair
    on surfaces,
    as well as the stability of the surface $\bbP_r-\bbP_0$ pair.

Lastly we comment on the issue of pressure robustness.
A common feature of divergence-free methods is
that the velocity error is independent of the
pressure approximation. In the Euclidean setting,
this property is one of the main advantages of exactly
divergence-free discretizations. For the surface Stokes problem,
however, the situation is more subtle, and pressure robustness is not 
automatically guaranteed for exactly (surface) divergence-free methods.
The primary obstruction arises from geometric approximation errors
and the required extension of the source function ${\bm f}$
to the computational surface. We argue and numerically
show that commonly used extensions destroy pressure robustness.
We also provide a framework for constructing
extensions of ${\bm f}$ that preserves pressure robustness; 
see Remark~\ref{rem:PressureR}.

The rest of the paper is organized as follows.
After setting the notation, we define the triangulations,
and polynomial diffeomorphisms in Section \ref{sec-prelim}.
Here, we also define pullbacks between the continuous
and discrete surfaces, as well as the Piola transform.
In Section \ref{sec-SV} we define the surface Scott-Vogelius
spaces and prove it is a divergence-free pair.
Section \ref{sec-aux} provides an inf-sup stability
proof of the surface pair, and as a byproduct proves
stability of the associated macro pair.
The finite element method and its convergence analysis
is given in Section \ref{sec:Converge}, and numerical examples
are given in Section \ref{sec-numerics}.
Finally, the proofs of some technical results are presented 
in the appendix.


\subsection{Notation Summary}\
\begin{table}[h!]
\centering
\small
\begin{tabular}{p{0.37\textwidth} p{0.47\textwidth}}

\begin{minipage}[t]{\linewidth}
\vspace{0pt}
\begin{tabular}{@{}ll@{}}
\multicolumn{2}{l}{\emph{Exact surface}} \\[3pt]
$\gamma$            & exact, smooth surface\\
$\bnu$              & outward unit normal of $\gamma$\\
$d$                 & signed distance function\\
$\bp$               & closest point projection\\
${\bf H}$           & Weingarten map\\
$\bPi$              & tangential projection w.r.t.\ $\gamma$\\
${\bm H}_T^1$       & $\bH^1$ tangential vector fields\\
\end{tabular}
\end{minipage}
&
\begin{minipage}[t]{\linewidth}
\vspace{0pt}
\begin{tabular}{@{}ll@{}}
\multicolumn{2}{l}{\emph{Polyhedral surface}} \\[3pt]
$\bar{\Gamma}_h$    & $\calO(h^2)$ polyhedral approximation to $\gamma$\\
$\bar{\bnu}_h$      & outward unit normal of $\bar{\Gamma}_h$\\
$\bar{\mathcal{T}}_h$ & set of faces of $\bar{\Gamma}_h$\\
$\bar\calT_h^{\rm ct}$ & Clough--Tocher refinement of $\bar \calT_h$\\
$\hat{T}$           & reference triangle\\
$F_{\bar{K}}:\hat T\to \bar K$
& affine diffeomorphism for $\bar K\in \bar \calT_h^{\rm ct}$\\
$F_{\bar T}:\hat T\to \bar T$
& affine diffeomorphism for $\bar T\in \bar \calT_h$\\
\end{tabular}
\end{minipage}
\\[8pt]
\begin{minipage}[t]{\linewidth}
\begin{tabular}{@{}ll@{}}
\multicolumn{2}{l}{\emph{High-order surface}} \\[3pt]
$\Gamma_{h}$
& $\calO(h^{k+1})$ approximation to $\gamma$\\
$\bPsict$
& p.w.~polynomial diffeomorphism $\bar{\Gamma}_h \to \Gamma_{h}$\\
$\mathcal{T}_{h}^{\rm ct}
 = \{\bPsict(\bar K):\ \bar K\in \bar \calT_h^{\rm ct}\}$
& triangulation of $\Gamma_{h}$\\
$F_K = \bPsict\circ F_{\bar K}:\hat T\to K$
& polynomial diffeomorphism onto
$K\in \calT_h^{\rm ct}$\\
 $\mathcal{N}^{\rm ct}_{h}$
 & $r$th degree Lagrange nodes of $\mathcal{T}^{\rm ct}_{h}$
 \end{tabular}
 \end{minipage}
\end{tabular}
\end{table}

\begin{table}[h!]
\centering
\small
\begin{tabular}{p{0.97\textwidth}} 
%
 \begin{minipage}[t]{\linewidth}
\begin{tabular}{@{}ll@{}}
 \multicolumn{2}{l}{\emph{Auxiliary high-order surface}} \\[3pt]
 $\bPsimac =  \bI^{\rm m}_h \bPsict$
 & Lagrange interpolant of $\bPsict$\\
 $\calTmac_h
  = \{\bPsimac(\bar T):\ \bar T\in \bar \calT_h\}$
 & auxiliary high-order triangulation\\
 $\calT^{\rm m,ct}_h
  = \{\bPsimac(\bar K):\ \bar K\in \bar \calT^{\rm ct}_h\}$
 & auxiliary high-order CT-like triangulation\\
 $F_{\Tmac} = \bPsimac\circ F_{\bar T}:\hat T\to \Tmac$
 & polynomial diffeomorphism onto
 $\Tmac\in \calTmac_h$\\
  $F_{\Kmac} = \bPsimac\circ F_{\bar K}:\hat T\to \Kmac$
 & polynomial diffeomorphism onto
 $\Kmac\in \calT^{\rm m,ct}_h$\\[8pt]
 \end{tabular}
 \end{minipage}
 \\[8pt]
 \begin{minipage}[t]{\linewidth}
\begin{tabular}{@{}ll@{}}
 \multicolumn{2}{l}{\emph{Pullbacks}} \\[3pt]
 $\bv^e$             & extension of $\bv$\\
 $\bv^\ell$          & lift of $\bv$\\
 $\mathcal{P}_{\bPhi}$ & Piola transform w.r.t.\ $\bPhi$\\
 $\pt\bv = \mathcal{P}_{\bp}\bv$
 & Piola transform w.r.t.\ $\bp$\\
 $\ipt\bv = \mathcal{P}_{\bp^{-1}}\bv$
 & Piola transform w.r.t.\ $\bp|_{\Gamma_{h}}^{-1}$\\
\end{tabular}
 \end{minipage}
\end{tabular}
\end{table}


\section{Preliminaries}\label{sec-prelim}
Let $d$ be the signed-distance function of $\gamma$, defined
in a tubular region $U_\delta$ of $\gamma$.
Set $\bnu = \nab d$, and note that $\bnu|_\gamma$ is the outward unit normal of $\gamma$.
The shape operator
is given by ${\bf H} = \nab \bnu = D^2 d$, and note that ${\bf H}\bnu = 0$ since $|\bnu|=1$.
The other two eigenvalues of ${\bf H}$ are denoted by $\kappa_1$ and $\kappa_2$, the principal
curvatures of $\gamma$. The closest point projection 
$\bp:U_\delta \to \gamma$ is given by $p(x) = x- d(x)\bn(x)$,
and we define the tangential projection $\bPi = {\bf I}_3- \bnu\otimes \bnu$, 
where ${\bf I}_3$ is the $3\times 3$ identity matrix.

The surface gradient of a scalar function $p:\gamma \to \bbR$
is given by $\nab_\gamma p = \bPi \nab p^e$, where $p^e = p\circ \bp$
denotes its extension. The surface gradient 
of a vector-valued function $\bv:\gamma\to \bbR^3$ 
is $\nab_\gamma \bv = \bPi \nab \bv^e \bPi$, whereas
the surface divergence and deformation tensor is
${\rm div}_\gamma \bv = {\rm tr}(\nab_\gamma \bv)$
and ${\rm Def_\gamma}\bv = \frac12 (\nab_\gamma \bv+(\nab_\gamma \bv)^\intercal)$,
respectively.

We denote by $\bH^1(\gamma)$ the Sobolev space
of vector fields whose components lie in $H^1(\gamma)$,
and let $\bH^1_T(\gamma) = \{\bv\in \bH^1(\gamma):\ \bv\cdot \bnu = 0\}$.
Let $\bH({\rm div}_\gamma;\gamma)$ denote the space of square integrable
vector fields with (weak) divergence in $L^2(\gamma)$, and 
let $L^2_0(\gamma)$ be the space
of square integrable functions on $\gamma$ with vanishing mean.
For a flat domain $D$ and $k\in \mathbb{N}_0 = \mathbb{N}\cup \{0\}$, 
we denote by $\mathbb{P}_k(D)$ the space of polynomials defined
on $D$ with degree $\le k$.

\subsection{Affine triangulation}
We adopt the construction and notations of the triangulations found in \cite{DemlowNeilan25}.

Let $\bar \Gamma_h$ be a closed, polyhedral surface,
the union of (flat) triangular faces $\bar \calT_h$,
and set $h_{\bar T} = {\rm diam}(\bar T)$ for all $\bar T\in \bar \calT_h$.
We assume that the distance function $d$ is defined on $\bar \Gamma_h$
and satisfies $d(x) = \calO(h^2)$ for every $x\in \bar \Gamma_h$,
where $h:=\max_{\bar T\in \bar \calT_h} h_{\bar T}$.
For each $\bar T\in \bar \calT_h$, let $\bar \bnu_{\bar T}\in \bbR^3$ be
its outward unit normal. Let $\bar \calV_h$ and $\bar \calE_h$
denote the sets of vertices and edges of $\bar \calT_h$, respectively.

Let $\hat T$ be the reference triangle with vertices $(1,0),(0,1),(0,0)$,
and let $F_{\bar T}:\hat T\to \bar T$ be an affine diffeomorphism, assumed to satisfy
\begin{align}\label{eqn:FbarTJacobian}
\|\nab F_{\bar T}\|_{L^\infty(\hat T)}\lesssim h_{\bar T},\qquad \det(\nab F_{\bar T}^\intercal \nab F_{\bar T}) \approx h_{\bar T}^4.
\end{align}
Here, $\nab F_{\bar T}\in \bbR^{3\times 2}$ is the 
(constant) Jacobian of $F_{\bar T}$.
We use the notation $a\lesssim b$
to indicate the existence of a constant $C>0$
independent of the mesh size
such that $a\le C b$. Similarly $a\approx b$
means $a\lesssim b$ and $b\lesssim a$ hold.

\subsection{Clough-Tocher splits}
For each affine triangle $\bar T\in \bar \calT_h$, 
let $\bar T^{ct} = \{\bar K_1,\bar K_2,\bar K_3\}$
be the local Clough-Tocher split, obtained by connecting
the vertices of $\bar T$ with its barycenter. 
Note that, because $\bar T$ is affine, its barycenter
lies in its interior. In particular, even if the vertices
of $\bar T$ lie on the exact surface, its barycenter
generally does not, and the refinement does not lead
to a better geometric approximation, although it is still
an $\mathcal{O}(h^2)$ approximation to $\gamma$.

The analogous
split on the reference triangle is denoted by 
$\hat T^{\rm ct} = \{\hat K_1,\hat K_2,\hat K_3\}$,
and the collection of all split
triangles yields the affine Clough-Tocher triangulation:
\[
\bar \calT_h^{\rm ct}:=\bigcup_{\bar T\in \bar \calT^{\rm ct}_h} \bar T^{\rm ct},
\]
Note that $\bar \Gamma_h = \cup_{\bar K\in \bar \calT^{\rm ct}_h} \bar K$.
Letting $F_{\bar K}:\hat T\to {\bar K}$ denote an affine
diffeomorphism from $\hat T$ to $\bar K\in \bar \calT_h^{\rm ct}$,
we easily see from shape regularity of $\bar \calT_h$ that the Jacobian of $F_{\bar K}$ satisfies the same bounds
as \eqref{eqn:FbarTJacobian}, i.e.,
\begin{align}\label{eqn:FbarKJacobian}
\|\nab F_{\bar K}\|_{L^\infty(\hat T)}\lesssim h_{\bar K},\qquad \det(\nab F_{\bar K}^\intercal \nab F_{\bar K}) \approx h_{\bar K}^4,
\end{align}
with $h_{\bar K} = {\rm diam}(\bar K) \approx h_{\bar T}$ ($\bar K\in \bar T^{\rm ct})$.
We denote by $\bar \calE_h^{\rm ct}$ the set of edges in
$\bar \calT_h^{\rm ct}$.

\subsection{High-order triangulation}
Let $\Gamma_h$ be a high-order approximation to $\gamma$, defined
by a continuous, $k$-degree piecewise polynomial mapping
{\em with respect to $\bar \calT_h^{\rm ct}$},
$\bPsict:\bar \Gamma_h\to \Gamma_h$.
In particular, 
\[
\Gamma_h = \bigcup_{K\in \calT^{\rm ct}_h} {\rm cl}(K),\qquad \calT^{\rm ct}_h = \{\bPsict(\bar K):\ \bar K\in \bar \calT^{\rm ct}_h\},
\]
where $\bPsict|_{\bar K}\in [\bbP_k(\bar K)]^3$ for all $\bar K\in \bar \calT_h^{\rm ct}$. We assume $\Gamma_h$ is an $\calO(h^{k+1})$ approximation to $\gamma$ for some $k\in \bbN$ in the sense that
\begin{align}\label{eqn:bPsiBounds}
|\bp(x)- \bPsict(x)| = \calO(h^{k+1}),\qquad |\nab \bp(x) - \nab \bPsict(x)| = \mathcal{O}(h^k),\qquad \text{on }\bar\Gamma_h.
\end{align}

We set $\calE_h^{\rm ct} = \{\bPsict(\bar e):\ \bar e\in \bar \calE^{\rm ct}_h\}$ to 
denote the set of curved edges in $\calT_h^{\rm ct}$.
The outward unit normal of $\Gamma^{\rm ct}_h$ is denoted
by $\bnu_h$, and we set $\bnu_K = \bnu_h|_K$ for $K\in \calT_h^{\rm ct}$.
Note that the second inequality
in \eqref{eqn:bPsiBounds} implies
\begin{equation}\label{eqn:NormalGood}
|\bnu-\bnu_h|= \calO(h^k).
\end{equation}
Further note that high-order piecewise derivatives of
$\bPsict$ are bounded due to \eqref{eqn:bPsiBounds},
 inverse estimates, and the smoothness of $\bp$.
 It then follows that piecewise derivatives
 of $\bnu_h$ on $\Gamma^{\rm ct}_h$ are uniformly 
 bounded.

For each $K\in \calT^{\rm ct}_h$, we let $F_K:\hat T\to K$ 
be the $k$th-degree polynomial diffeomorphism defined by
$F_K = \bPsict \circ F_{\bar K}$, where $\bar K = \bPsict^{-1}(K)\in \bar \calT^{\rm ct}_h$ (cf.~Figure~\ref{fig:MappingsABC}).
We assume $(1\le m\le (k+1))$
\begin{align}\label{eqn:FKBounds}
|F_K|_{W^{m,\infty}(\hat T)}\lesssim h^m_K,\qquad |F_K^{-1}|_{W^{m,\infty}(K)}\lesssim h^{-1}_K,\qquad
\det(\nab F_K^\intercal \nab F_K) \approx h^4_K,
\end{align}
where $h_K = h_{\bar K}$ with $\bar K = \bPsict^{-1}(K)$.
These estimates hold if, e.g., $\bPsict$ is the $k$-th degree Lagrange interpolant of $\bp$ with respect to the mesh $\bar \calT_h^{\rm ct}$ (cf.~\cite{Demlow09,Lenoir86}),
but we do not necessarily make this assumption. 
We also see that \eqref{eqn:FbarKJacobian} and \eqref{eqn:FKBounds}
yield the following estimates for the mapping $\bPsict$:
\begin{align}\label{eqn:bPsictBounds}
|\bPsict|_{W^{m,\infty}(\bar K)}\lesssim 1,\qquad |\bPsict^{-1}|_{W^{m,\infty}(K)}\lesssim 1.
\end{align}



For a piecewise smooth function $v$
with respect to a partition $\mathcal{D}_h$,
we set the piecewise $H^m$ norm as
\[
\|v\|_{H^m_h(D_h)}^2 = \sum_{S\in \mathcal{D}_h} \|v\|_{H^m(S)}^2, 
\]
where $D_h = \cup_{S\in \mathcal{D}_h} S$.

\begin{figure}
\caption{\label{fig:MappingsABC}A depiction of the mappings $\bPsict|_{\bar K}$, $F_{\bar K}$, and $F_K$.
\medskip\\}
\begin{tikzpicture}[scale=1.25, line cap=round, line join=round]

\tikzset{
  surf/.style={draw, thick},
  edge/.style={draw, thick},
  map/.style={->, thick, >=stealth},
  dashededge/.style={draw, dashed, gray},
  lab/.style={font=\small}
}

\coordinate (A) at (0,3.15);
\coordinate (B) at (4.05,4.95);
\coordinate (C) at (5.25,3.10);
\coordinate (D) at (2.95,3.80);

\draw[surf] (A) .. controls (1.25,4.55) and (2.85,4.85) .. (B)
           .. controls (5.10,4.55) and (5.55,3.75) .. (C)
           .. controls (3.55,2.95) and (1.35,2.85) .. (A);

\draw[edge] (A) .. controls (1.25,3.25) and (2.20,3.45) .. (D);
\draw[edge] (D) .. controls (3.65,3.55) and (4.50,3.25) .. (C);
\draw[edge] (D) .. controls (3.35,4.20) and (3.75,4.55) .. (B);

\node[lab] at (1.95,4.15) {$K_3$};
\node[lab] at (4.05,4) {$K_2$};
\node[lab] at (2.85,3.25) {$K_1$};

\coordinate (a) at (1.45,0.95);
\coordinate (b) at (4.75,1.45);
\coordinate (c) at (4.0,-0.5);

\coordinate (d) at (3.4,0.6333);

\draw[surf] (a) -- (b) -- (c) -- cycle;
\draw[edge] (a) -- (d);
\draw[edge] (b) -- (d);
\draw[edge] (c) -- (d);

\node[lab] at (2.95,0.3611) {$\overline K_1$};
\node[lab] at (4.05,0.52778) {$\overline K_2$};
\node[lab] at (3.2,1.011) {$\overline K_3$};
\node[lab] at (3.25,-0.5) {$\overline T$};

\draw[map] (2.6,0.4) .. controls (0.75,1.65) and (0.35,2.70) .. (1.45,3.15);
\node[lab] at (0.65,1.9) {$\bPsi_{\rm ct}|_{\overline K_1}$};

\draw[map] (4.25,0.55) .. controls (5.35,1.75) and (5.25,2.85) .. (4.05,3.65);
\node[lab] at (5.42,2.25) {$\bPsi_{\rm ct}|_{\overline K_2}$};

\draw[map] (3.45,1.10) .. controls (2.80,2.20) and (2.45,3.25) .. (2.35,4.15);
\node[lab] at (2.40,2.55) {$\bPsi_{\rm ct}|_{\overline K_3}$};

\coordinate (x) at (1,-1.5);
\coordinate (y) at (-1,0.5);
\coordinate (z) at (-1,-1.5);
\draw[surf] (x) -- (y) -- (z) -- cycle;
\node[lab] at (-1/3,-2.5/3) {$\hat T$};



\draw[map]
  (0.25,-0.55)
  .. controls (0.75,-0.10) and (1.55,0.10) ..
  (3.25,0.05);
\node[lab] at (1.50,-0.25) {$F_{\overline K_1}$};

\draw[map]
  (-0.15,-0.25)
  .. controls (-0.45,0.40) and (-0.40,1.60) ..
  (0.95,3.15);
 \node[lab] at (-0.55,1.20) {$F_{K_1}$};




\end{tikzpicture}
\end{figure}
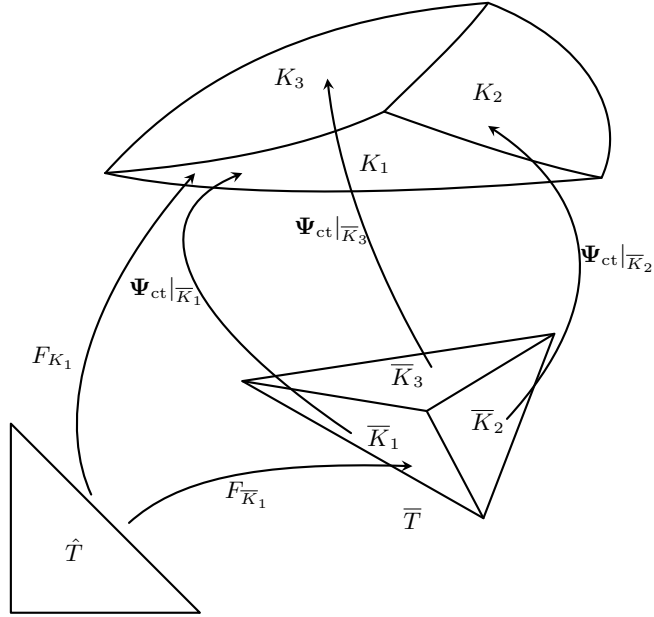

\subsection{Mappings between $\bar \Gamma_h$, $\Gamma_h$, and $\gamma$}
For a (scalar or vector-valued) function $w$ defined on $\gamma$,
we let $w^e = w\circ \bp$ denotes its extension 
to the neighborhood $U_\delta$. Likewise, for a function $v$ defined
on a surface $S_h\subset U_\delta$ (e.g., $S_h = \bar \Gamma_h$ or $S_h = \Gamma_h$), 
we set $\tilde v = v\circ (\bp|_{S_h}^{-1})$ and define
the lift of $v$ as $v^\ell = \tilde v \circ \bp$ on $U_\delta$.
%
We have \cite{CD16}
\begin{align*}
\int_{\bar \Gamma_h} \bar \mu_h w^e = \int_{\Gamma_h} \mu_h w^e = \int_{\Gamma} w,
\end{align*}
where
\begin{equation}\label{eqn:muDef}
\mu_h:= \bnu\cdot \bnu_h(1-d \kappa_1)(1-d\kappa_2),\quad 
\bar \mu_h = \bnu\cdot \bar \bnu_h(1-d \kappa_1)(1-d\kappa_2),
\end{equation}
and we recall that $\{\kappa_1,\kappa_2\}$ are the principal curvatures of $\gamma$.

Next, we introduce the Piola transform 
between two surfaces as described in \cite{CD16,DemlowNeilan24,DemlowNeilan25}.
Let $\mathcal{S}_0$ and $\mathcal{S}_1$ be two surfaces,
and suppose $\bPhi:\mathcal{S}_0\to \mathcal{S}_1$ is a 
bi-Lipschitz mapping. The Piola transform of a vector field
$\bv_0:\mathcal{S}_0\to \mathbb{R}^3$ with respect
to $\bPhi$ is the function $(\calP_{\bPhi} \bv_0):\mathcal{S}_1\to \mathbb{R}^3$ given
by
\[
(\calP_{\bPhi} \bv_0) := (\mu^{-1} \nab \bPhi \bv_0)\circ \bPhi^{-1},
\]
where $\mu:\mathcal{S}_0\to \mathbb{R}$ satisfies
$\mu d\sigma_0 = d \sigma_1$, and $d\sigma_i$ is the surface measure
of $\mathcal{S}_i$. Likewise, the Piola transform of 
a vector field $\bv_1:\mathcal{S}_1\to \mathbb{R}^3$ 
with respect to the inverse mapping $\bPhi^{-1}$
is given by 
\[
(\calP_{\bPhi^{-1}} \bv_1) = \mu (\nab \bPhi^{-1} \bv_1)\circ \bPhi:\mathcal{S}_0\to \mathbb{R}^3.
\]
These transforms satisfy 
${\rm div}_{\calS_0} \bv_0 = (\mu^{-1} {\rm div}_{\calS_1} \calP_{\bPhi} \bv_0)\circ \bPhi^{-1}$
and 
${\rm div}_{\calS_1} \bv_1 = \mu ({\rm div}_{\calS_0} \calP_{\bPhi^{-1}} \bv_1)\circ \bPhi$
for all $\bv_0\in \bH({\rm div}_{\calS_0};\calS_0)$ and 
$\bv_1\in \bH({\rm div}_{\calS_1};\calS_1)$.
In particular, $\calP_{\bPhi}:\bH({\rm div}_{\calS_0};\calS_0)\to\bH({\rm div}_{\calS_1};\calS_1)$
and
 $\calP_{\bPhi^{-1}}:\bH({\rm div}_{\calS_1};\calS_1)\to\bH({\rm div}_{\calS_0};\calS_0)$
are bounded operators.

In the case $\calS_0 = \Gamma_h$, $\calS_1 = \gamma$, 
and $\bPhi = \bp|_{\Gamma_h}$, we introduce the notation
\[
\pt \bv_0:=\calP_{\bp} \bv_0 = (\mu_h^{-1} \nab \bp \bv_0)\circ \bp^{-1} = (\mu_h^{-1}\big[\bPi- d{\bf H}\big]\bv_0)\circ \bp^{-1},
\]
where $\mu_h$ is given by \eqref{eqn:muDef}.
Likewise, we set
\[
\ipt \bv_1:= \calP_{\bp^{-1}}\bv_1 = \mu_h \left[{\bf I}_3 - \frac{\bnu\otimes \bnu_h}{\bnu\cdot \bnu_h}\right]\left[{\bf I}_3 - d {\bf H}\right]^{-1} \bv_1\circ \bp,
\]
and note that
\begin{alignat*}{2}
\int_{\Gamma_h} ({\rm div}_{\Gamma_h} \bv_0) q &= \int_\gamma ({\rm div}_{\gamma}\pt \bv_0)q^\ell\qquad &&\forall q\in L^2(\Gamma_h),\\
\int_\gamma ({\rm div}_\gamma \bv_1)q & = \int_{\Gamma_h} ({\rm div}_{\Gamma_h} \ipt \bv_1) q^e\qquad&& \forall q\in L^2(\gamma),
\end{alignat*}
and (cf.~\cite[(2.11)]{DemlowNeilan25})
\[
\|\bv_0\|_{H^\ell(K)}\approx \|\pt \bv_0\|_{H^\ell(\bp(K))}\qquad \forall \bv_0 \in \bH^\ell(K)\ \forall K\in \calT_h^{\rm ct},\ \ell\in \mathbb{N}_0.
\]
Similarly, if  
$\bv\in \bH({\rm div}_{\Gamma_h};\Gamma_h)$
    and $\bar \bv = \calP_{\bPsict^{-1}} \bv \in \bH({\rm div}_{\bar \Gamma_h};\bar \Gamma_h)$, then we have
    \begin{align*}
        \int_{\Gamma_h} ({\rm div}_{\Gamma_h} \bv) q = \int_{\bar \Gamma_h} ({\rm div}_{\bar \Gamma_h}  \bar \bv) (q\circ \bPsict)\qquad \forall q\in L^2(\Gamma_h),
    \end{align*}
and \cite[(2.16)]{DemlowNeilan25}
\[
\|\bv\|_{H^\ell(K)} \approx \|\bar \bv\|_{H^\ell(\bar K)}\qquad \forall \bv\in \bH^\ell(K),\ \ell\in \mathbb{N}_0,\ K = \bPsict(\bar K).
\]

\section{The surface Scott-Vogelius finite element pair}\label{sec-SV}
For $r\in \bbN$ with $r\ge 2$ and
for a curved surface triangle
in the Clough-Tocher triangulation $K\in \calT_h^{\rm ct}$, we define the local spaces
\[
\bV(K) :=\{\calP_{F_K} \hat \bv:\ \hat \bv\in \bbP_r(\hat T)^2\},\qquad
Q(K) := \{\hat q\circ F_K^{-1}:\ \hat q\in \bbP_{r-1}(\hat T)\},
\]
where we recall that $F_K:\hat T\to K$ is a polynomial
diffeomorphism of degree $k$, and $\calP_{F_K}$ is the Piola
transform with respect to this mapping.

Let $\bar \calN_h^{\rm ct}$ denote the set
of points of Lagrange degrees of freedom (DOFs) of degree $r$
with respect to the affine Clough-Tocher triangulation $\bar \calT_h^{\rm ct}$,
and map these points to the computational triangulation $\calT_h^{\rm ct}$ via $\bPsict$:
\[
\calN_h^{\rm ct}:=\{\bPsict(\bar a):\ \bar a\in \bar \calN_h^{\rm ct}\}.
\]
As in \cite{DemlowNeilan25}, to control the consistency errors
of the resulting finite element scheme in Section \ref{sec:Converge} below,
we make the following assumption
about the placement of the edge DOFs.
\begin{assumption}\label{assume:GL}
We assume that the edge DOFs in $\bar \calN_h^{\rm ct}$ correspond to the $(r+1)$-point
Gauss-Lobatto rule.
\end{assumption}

For a DOF $a$, we denote by $\calT^{\rm ct}_a\subset \calT_h^{\rm ct}$
the set of elements abutting $a$.
We then introduce the following Piola transform
used to enforce weak continuity properties
in the (velocity) finite element space (cf.~\cite{DemlowNeilan24,DemlowNeilan25}).
\begin{definition}
 For each $a\in \calN^{\rm ct}_h$, we arbitrarily choose a single fixed $K_a\in \calT^{\rm ct}_a$.
For $K\in \calT^{\rm ct}_a$, we define $\mathcal{M}_a^K:\bbR^3\to \bbR^3$ by 
\begin{align}\label{eqn:calMaK}
\calM_a^K {\bm x} = \left((\bnu_{K_a}(a)\cdot \bnu_K(a)){\bf I}_3 - \bnu_{K_a}(a) \otimes \bnu_K(a)\right) {\bm x}.
\end{align}
Namely, $\calM_a^K {\bm x}$ is the Piola transform
of the vector ${\bm x}$ with respect to the inverse of the closest point projection
onto the plane tangent to $K_a$ at $a$ (cf.~\cite{DemlowNeilan24,DemlowNeilan25}).
\end{definition}

The global Scott-Vogelius-type pair
 is defined by concatenating the local spaces,
gluing the velocity space at the Lagrange DOFs via the operator $\calM^K_a$:
\begin{align*}
\bV_h & = \{\bv\in \bL^2(\Gamma_h):\ \bv|_K\in \bV(K)\ \forall K\in \calT_h^{\rm ct};\ 
 \bv|_K(a) = \calM_a^K \bv|_{K_a}(a)\ \forall a\in \calN_h^{\rm ct}\ \forall K\in \calT_a^{\rm ct}\},\\
Q_h & = \{q\in L^2_0(\Gamma_h):\ q|_K\in Q(K)\ \forall K\in \calT_h^{\rm ct}\}.
\end{align*}
The discrete velocity space $\bV_h$ is the same
as the one given in \cite{DemlowNeilan25}, but defined with respect to a surface Clough-Tocher
triangulation. However, the discrete pressure space $Q_h$ consists of discontinuous
(mapped) polynomials, whereas continuous (mapped) piecewise polyomials are used
in \cite{DemlowNeilan25}.

\begin{remark}\label{rem:HdivRemark}
It is shown in \cite{DemlowNeilan25}
that $\bV_h\subset \bH({\rm div}_{\Gamma_h};\Gamma_h)$, i.e.,
functions in $\bV_h$ have co-normal continuity across edges
in the mesh. In addition, the space has ``weak continuity'' properties,
as described in Section \ref{sec:Consistency} below.
\end{remark}

The next lemma shows that 
$\bV_h\times Q_h$ constitutes a divergence-free pair
for the Stokes problem.
%
\begin{lemma}\label{lem:DivFreeProp}
    If $\bv \in \bV_h$ satisfies 
    \[\int_{\Gamma_h}({\rm div}_{\Gamma_h} \bv)q=0 \qquad \forall q\in Q_h,
    \]
    then ${\rm div}_{\Gamma_h} \bv \equiv 0$ on $\Gamma_h$.
\end{lemma}
\begin{proof}
For each $K\in \calT_h^{\rm ct}$,
 write $\bv|_K = \calP_{F_K} \hat \bv_K$ for some $\hat \bv_K\in \bbP_r(\hat T)^2$.
 Define $q$ such that
 \[
 q|_K=(\hat{\nab}\cdot \hat{\bv}_K)\circ F_K^{-1}\qquad \forall K\in \calT^{\rm ct}_h,
 \]
and let $c = |\Gamma_h|^{-1}\int_{\Gamma_h} q\in \bbR$.
Because $\hat \nab \cdot \hat \bv_K\in \hat \bbP_{r-1}(\hat T )$, we have
$q|_K\in Q(K)$ for all $K\in \calT^{\rm ct}_h$, and so $(q-c)\in Q_h$.

By the divergence theorem, and by divergence-preserving 
    properties of the Piola transform, we have
\begin{align*}
0 &=     \int_{\Gamma_h} ({\rm div}_{\Gamma_h} \bv)(q-c)
= \int_{\Gamma_h} ({\rm div}_{\Gamma_h} \bv)q\\
&= \sum_{K\in \calT^{\rm ct}_h} \int_K  ({\rm div}_{\Gamma_h} \bv) (\hat \nab \cdot \hat \bv_K)\circ F_K^{-1}
=  \sum_{K\in \calT^{\rm ct}_h} \int_K  \sqrt{\det(\nab F_K^\intercal \nab F_K)\circ F_K^{-1}} |{\rm div}_{\Gamma_h} \bv|^2.
\end{align*}
From this identity, we conclude ${\rm div}_{\Gamma_h} \bv \equiv 0$.
\end{proof}

One of the main goals in this paper is to prove the following inf-sup
stability result. The next section is devoted to its proof.

\begin{theorem}\label{thm:MainStab}
There exists $\beta>0$ independent of the mesh parameter $h$ such that
\begin{equation}\label{equ:MainInfSup}
\beta \|q\|_{L^2(\Gamma_h)}\le \sup_{\bv\in \bV_h\backslash \{0\}} \frac{\int_{\Gamma_h} ({\rm div}_{\Gamma_h}\bv)q}{\|\bv\|_{H^1_h(\Gamma_h)}}\qquad \forall q\in Q_h.
\end{equation}
\end{theorem}

\section{An auxiliary macro surface Scott-Vogelius finite element pair and
proof of Theorem \ref{thm:MainStab}}\label{sec-aux}
To prove the stability result in Theorem \ref{thm:MainStab}, 
we first establish an analogous stability result for 
auxiliary macro finite element spaces. Once this is achieved, we apply a perturbation argument to derive the desired estimate \eqref{equ:MainInfSup}.
These results also show that the auxiliary macro finite element pair provides a stable and divergence-free discretization of the surface Stokes problem and is therefore of independent interest. However, its practical implementation is relatively involved, as the reference basis functions of the finite element spaces are piecewise polynomials.

\subsection{Auxiliary high-order triangulation}
The triangulation $\calT_h^{\rm ct}$ is the high-order computational mesh
used in our proposed scheme, but we also introduce an auxiliary
high-order triangulation that will be utilized in the proof of 
Theorem \ref{thm:MainStab}.
Let $\bImac_h:{\bm C}(\bar \Gamma_h)\to \bbR^3$ denote the $k$th degree
Lagrange interpolant {\em with respect to the unrefined affine
triangulation} $\bar \calT_h$, and set 
\[
\bPsimac = \bImac_h \bPsict.
\]
Thus, $\bPsimac$ is a continuous, piecewise polynomial
mapping of degree $k$ with respect to $\bar \calT_h$ such that $\bPsimac(\bar a) = \bPsict(\bar a)$
at all $k$th degree Lagrange (nodal) degrees of freedom $\bar a$ in  $\bar \calT_h$.
We then let $\calTmac_h$ denote the {\em macro} mesh,
which consists of elements mapped through $\bPsimac$:
\[
\calTmac_h = \{\bPsimac(\bar T):\ \bar T\in \bar \calT_h\}.
\]
For $\bar T\in \bar \calT_h$ and $\Tmac=\bPsimac(\bar T)$, we define
$F_{\Tmac}:\hat T\to \Tmac$ such that $F_{\Tmac} = \bPsimac\circ F_{\bar T}$.
We also set
\[
\Gammamac_h = \bigcup_{\Tmac\in \calTmac_h} {\rm cl}(\Tmac).
\]
We denote by $\calVmac_h =\{\bPsimac(\bar a):\ \bar a\in \bar \calV_h\}$
and $\calEmac_h = \{\bPsimac(\bar e):\ \bar e\in \bar \calE_h\}$
the sets of vertices and edges in $\calTmac_h$, respectively. 
We denote by $\bnu_h^{\rm m}$ the outward unit
normal of $\Gammamac_h$, and for
for each $\Tmac\in \calTmac_h$, we set $\bnu_{\Tmac} = \bnu_h^{\rm m}|_{\Tmac}$.

Associated with the macro mesh $\calTmac_h$ is
the partition
\[
\calT^{\rm m,ct}_h:=\{\bPsimac(\bar K):\ \bar K\in \bar \calT_h^{\rm ct}\}.
\]
Similar to the notation above,
we set $\bnu_{K^{\rm m}} = \bnu_h^{\rm m}|_{K^{\rm m}}$.
Note that for each $K^{\rm m}$,
there exists a unique $\Tmac\in \calTmac_h$ such
that $K^{\rm m}\subset \Tmac$. Likewise, 
for each $\Tmac\in \calTmac_h$, there exists
$\{K^{\rm m}_1,K^{\rm m}_2,K^{\rm m}_3\}\subset \calT^{\rm m,ct}_h$ such
that ${\rm cl}(\Tmac) = {\rm cl}(K^{\rm m}_1)\cup {\rm cl}(K^{\rm m}_2)\cup {\rm cl}(K^{\rm m}_3)$.
We set, for each $\Kmac\in \calT^{\rm m,ct}_h$, $F_{\Kmac}:=\bPsimac\circ F_{\bar K}$ with $\Kmac = \bPsimac(\bar K)$.

If $\bPsict$ is the Lagrange interpolant
of $\bp$ (with respect to $\bar \calT_h^{\rm ct}$), then $\bPsimac$
is the Lagrange interpolant of $\bp$ with respect to $\bar \calT_h$,
and therefore satisfies bounds similar to \eqref{eqn:bPsiBounds}.
The next result shows that these bounds hold even if $\bPsict$
is not the Lagrange interpolant. Its proof is given in the appendix.
\begin{lemma}\label{lem:PsiMac}
There holds
\begin{align}\label{eqn:PsiMacApprox}
|\bp(x) - \bPsimac(x)| = \calO(h^{k+1}),\qquad |\nab \bp(x) - \nab \bPsimac(x)| = \calO(h^k),\qquad \text{on }\bar \Gamma_h,
\end{align}
and therefore
\begin{equation}
    \label{eqn:macNuGood}
    |\bnu-\bnu_h^{\rm m}| = \mathcal{O}(h^k).
\end{equation}
Consequently, by \eqref{eqn:bPsiBounds} we have
\begin{align}\label{eqn:PsiMacInterpProp}
|\bPsict - \bPsimac(x)| = \calO(h^{k+1}),\qquad |\nab \bPsict - \nab \bPsimac(x)| = \calO(h^k),\qquad \text{on }\bar \Gamma_h.
\end{align}
Moreover, for all $\Tmac\in \calTmac_h$ and $\calT^{\rm m,ct}_h\ni \Kmac\subset \Tmac$,
\begin{equation}\label{eqn:FTMacBounds}
\begin{aligned}
&|F_{\Tmac}|_{W^{\ell,\infty}(\hat T)}\lesssim h^\ell_T,\qquad |F_{\Tmac}^{-1}|_{W^{\ell,\infty}(\Tmac)}\lesssim h^{-1}_T,\qquad
&&\det(\nab F_{\Tmac}^\intercal \nab F_{\Tmac}) \approx h^4_T,\\
&|F_{\Kmac}|_{W^{\ell,\infty}(\hat T)}\lesssim h^\ell_K,\qquad |F_{\Kmac}^{-1}|_{W^{\ell,\infty}(\Kmac)}\lesssim h^{-1}_K,\qquad
&&\det(\nab F_{\Kmac}^\intercal \nab F_{\Kmac}) \approx h^4_K,
\end{aligned}
\end{equation}
where $h_T = {\rm diam}(\bPsimac^{-1}(\Tmac))$ and $h_K = {\rm diam}(\bPsimac^{-1}(\Kmac))$.
\end{lemma}

Recall that we set $K = \bPsict(\bar K)\in \calT^{\rm ct}_h$ for $\bar K\in \bar \calT^{\rm ct}_h$,
and note that $K\neq \Kmac=\bPsimac(\bar T)\in \calTmac_h$ in general (see Figure \ref{fig:FancyFig}).
Nonetheless,
the following lemmas show that $\Kmac$ is an
$\calO(h^{k+1})$ perturbation of $K$. Their proofs are given
in the appendix.
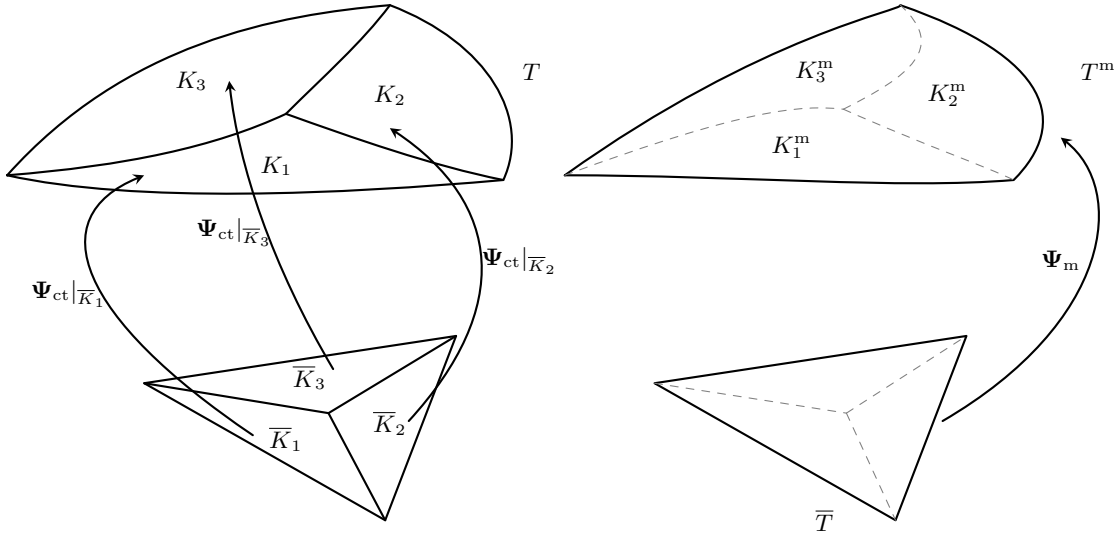
\begin{figure}[h]
\caption{\label{fig:FancyFig}A depiction of the meshes and mappings. Left: a local Clough–Tocher triangulation obtained by mapping each subtriangle $\bar K_i$ to the curved surface. Right: the macro-element triangulation, defined via a single polynomial diffeomorphism on $\bar T$.}
\begin{tikzpicture}[scale=1.25, line cap=round, line join=round]

\tikzset{
  surf/.style={draw, thick},
  edge/.style={draw, thick},
  map/.style={->, thick, >=stealth},
  dashededge/.style={draw, dashed, gray},
  lab/.style={font=\small}
}

\coordinate (A) at (0,3.15);
\coordinate (B) at (4.05,4.95);
\coordinate (C) at (5.25,3.10);
\coordinate (D) at (2.95,3.80);

\draw[surf] (A) .. controls (1.25,4.55) and (2.85,4.85) .. (B)
           .. controls (5.10,4.55) and (5.55,3.75) .. (C)
           .. controls (3.55,2.95) and (1.35,2.85) .. (A);

\draw[edge] (A) .. controls (1.25,3.25) and (2.20,3.45) .. (D);
\draw[edge] (D) .. controls (3.65,3.55) and (4.50,3.25) .. (C);
\draw[edge] (D) .. controls (3.35,4.20) and (3.75,4.55) .. (B);

\node[lab] at (1.95,4.15) {$K_3$};
\node[lab] at (4.05,4) {$K_2$};
\node[lab] at (2.85,3.25) {$K_1$};
\node[lab] at (5.55,4.25) {$T$};

\coordinate (a) at (1.45,0.95);
\coordinate (b) at (4.75,1.45);
\coordinate (c) at (4.0,-0.5);

\coordinate (d) at (3.4,0.6333);

\draw[surf] (a) -- (b) -- (c) -- cycle;
\draw[edge] (a) -- (d);
\draw[edge] (b) -- (d);
\draw[edge] (c) -- (d);

\node[lab] at (2.95,0.3611) {$\overline K_1$};
\node[lab] at (4.05,0.52778) {$\overline K_2$};
\node[lab] at (3.2,1.011) {$\overline K_3$};

\draw[map] (2.6,0.4) .. controls (0.75,1.65) and (0.35,2.70) .. (1.45,3.15);
\node[lab] at (0.65,1.9) {$\bPsi_{\rm ct}|_{\overline K_1}$};

\draw[map] (4.25,0.55) .. controls (5.35,1.75) and (5.25,2.85) .. (4.05,3.65);
\node[lab] at (5.42,2.25) {$\bPsi_{\rm ct}|_{\overline K_2}$};

\draw[map] (3.45,1.10) .. controls (2.80,2.20) and (2.45,3.25) .. (2.35,4.15);
\node[lab] at (2.40,2.55) {$\bPsi_{\rm ct}|_{\overline K_3}$};


\coordinate (E) at (5.90,3.15);
\coordinate (F) at (9.45,4.95);
\coordinate (G) at (10.65,3.10);


\draw[surf] (E) .. controls (8.00,4.65) and (9.50,4.90) .. (F)
           .. controls (11.05,4.40) and (11.25,3.70) .. (G)
           .. controls (9.45,3.00) and (7.75,3.15) .. (E);

\draw[dashededge] (E) .. controls (7.80,3.85) and (8.40,3.90) .. (8.85,3.85);
\draw[dashededge] (8.85,3.85) .. controls (9.55,4.20) and (9.95,4.50) .. (F);
\draw[dashededge] (8.85,3.85) .. controls (9.75,3.45) and (10.35,3.25) .. (G);

\node[lab] at (9.95,4.) {$K_2^{\rm m}$};
\node[lab] at (8.55,4.25) {$K_3^{\rm m}$};
\node[lab] at (8.3,3.45) {$K_1^{\rm m}$};

\node[lab] at (11.55,4.25) {$T^{\rm m}$};

\coordinate (p) at (6.85,0.95);
\coordinate (q) at (10.15,1.45);
\coordinate (r) at (9.4,-0.5);

\draw[surf] (p) -- (q) -- (r) -- cycle;

\coordinate (pd) at (8.88,0.63333);
\draw[dashededge] (p) -- (pd);
\draw[dashededge] (pd) -- (q);
\draw[dashededge] (pd) -- (r);

\node[lab] at (8.65,-0.5) {$\overline T$};

\draw[map] (9.9,0.55) .. controls (11.70,1.60) and (11.85,3.00) .. (11.15,3.55);
\node[lab] at (11.15,2.25) {$\bPsimac$};

\end{tikzpicture}
\end{figure}

\begin{lemma}\label{lem:FDiff}
Let $\bar K\in \bar \calT^{\rm ct}_h$, $K = \bPsict(\bar K)$ and $\Kmac = \bPsimac(\bar K)$.
Then there holds
\begin{align*}
|F_K - F_{\Kmac}|_{W^{\ell,\infty}(\hat T)}\lesssim h^{k+1}_K\qquad \ell=0,1.
\end{align*}
\end{lemma}

\begin{lemma}\label{lem:ATDiff}
Set $A_K = \nab F_K/J_K$ and $A_{\Kmac} = \nab F_{\Kmac}/J_{\Kmac}$, where
$J_K = \sqrt{\nab F_K^\intercal \nab F_K}$ and 
$J_{\Kmac} = \sqrt{\nab F_{\Kmac}^\intercal \nab  F_{\Kmac}}$. 
Then there holds
\[
\|A_K - A_\Kmac\|_{W^{\ell,\infty}(\hat T)}\lesssim h^{k-1}_K\qquad \ell=0,1.
\]
\end{lemma}

\subsection{Local macro spaces}
Let
 $\bbP_r(\hat T^{ct})$ be the space
of (discontinuous) piecewise polynomials of degree $\le r$
with respect to the reference Clough-Tocher split $\hat T^{ct}$,
and set $\bbP_r^c(\hat T^{ct}) = \bbP_r(\hat T^{ct})\cap H^1(\hat T)$
to be the analogous continuous piecewise polynomial space.
For $\bar T\in \bar \calT_h$ and $\Tmac\in \calTmac_h$ related via
$\Tmac = \bPsimac(\bar T)$, we define the local spaces
\begin{alignat*}{3}
\bVhatmac:&=[\bbP^c_r(\hat T^{ct})]^2,\qquad
&&\bVmac(\bar T):=\{\calP_{F_{\bar T}} \hat \bw:\ \hat \bw\in \bVhatmac\},\qquad
&&\bVmac(\Tmac):=\{\calP_{F_{\Tmac}} \hat \bw:\ \hat \bw\in \bVhatmac\},\\
\Qhatmac:&=\bbP_{r-1}(\hat T^{ct}),\qquad
&&\Qmac(\bar T):= \{\hat r\circ F_{\bar T}^{-1}:\ \hat r\in \Qhatmac\},\qquad
&&\Qmac(\Tmac):=\{\hat r\circ F_{\Tmac}^{-1}:\ \hat r\in \Qhatmac\}.
\end{alignat*}
Because $F_{\bar T}$ is affine, $\calP_{F_{\bar T}}$ can be identified
with a constant $3\times 2$ matrix, and thus $\bVmac(\bar T)$ simply
consists of (continuous) tangential piecewise polynomials
of degree $r$ with respect to the local Clough-Tocher split of $\bar T$. 
Further note that $\calP_{F_{\Tmac}} = \calP_{\bPsimac\circ F_{\bar T}} = \calP_{\bPsimac} \calP_{F_{\bar T}}$,
and so we have an equivalent definition of $\bVmac(\Tmac)$ and $\Qmac(\Tmac)$: 
\begin{align*}
\bVmac(\Tmac) &= \{\calP_{\bPsimac} \bar \bw:\ \bar \bw\in \bVmac(\bar T)\},\\
\Qmac(\Tmac) & = \{\bar r\circ \bPsimac^{-1}:\ \bar r\in \Qmac(\bar T)\}.
\end{align*}
We also define the spaces with homogeneous Dirichlet boundary conditions:
\begin{alignat*}{3}
\bVhatmac_0:&= \bVhatmac\cap \bH^1_0(\hat T),\qquad
&&\bVmac_0(\bar T):=\{\calP_{F_{\bar T}} \hat \bw:\ \hat \bw\in \bVhatmac_0\},\qquad
&&\bVmac_0(\Tmac):=\{\calP_{F_{\Tmac}} \hat \bw:\ \hat \bw\in \bVhatmac_0\},\\
\Qhatmac_0:&= \Qhatmac\cap L^2_0(\hat T),\qquad
&&\Qmac_0(\bar T):= \{\hat r\circ F_{\bar T}^{-1}:\ \hat r\in \Qhatmac_0\},\qquad
&&\Qmac_0(\Tmac):=\{\hat r\circ F_{\Tmac}^{-1}:\ \hat r\in \Qhatmac_0\}.
\end{alignat*}

The following result essentially follows from \cite[Theorem 3.9]{NeilanOtus21}.
For completeness, we provide its proof.
\begin{lemma}\label{lem:LocalMacroStable}
Given $\Tmac\in \calTmac_h$ and $r\in \Qmac_0(\Tmac)$, there exists $\bw\in \bVmac_0(\Tmac)$ such that
\[
{\rm div}_{\Gammamac_h} \bw=\frac{h^2_Tr}{\sqrt{\det((\nab F_{\Tmac}^\intercal \nab F_{\Tmac})\circ F_{ \Tmac}^{-1})}}\quad \text{in }\Tmac,\quad \text{and}\quad \|\bw\|_{H^1(\Tmac)}\lesssim \|r\|_{L^2(\Tmac)}.
\]
\end{lemma}
\begin{proof}
    Let $r\in \Qmac_0(\Tmac)$ be arbitrary. Then, there exists $\hat{r}\in \Qhatmac_0$ such that $r=\hat{r} \circ F_{\Tmac}^{-1}$. It follows that $h^2_T\hat{r}\in \Qhatmac_0$, and so by the surjectivity 
    of the divergence operator on the local Clough-Tocher split $\hat{T}^{ct}$ (cf.~\cite{GuzmanNeilan18}), there exists $\hat{\bw}\in \bVhatmac_0$ such that $\widehat{\rm div}\, \hat{\bw}=h_T^2\hat{r}$ and $\|\hat{\bw}\|_{H^1(\hat{T})}\lesssim \|h_T^2\hat{r}\|_{L^2(\hat{T})}$. Letting $\bw = \calP_{F_{\Tmac}}\hat{\bw}$, we have
    \begin{align*}
    {\rm div}_{\Gammamac_h} \bw=\frac{\widehat{\rm div} \hat{\bw}}{\sqrt{\det(\nab F_{\Tmac}^\intercal \nab F_{\Tmac})}}\circ F_{\Tmac}^{-1}
    &=\frac{h^2_T\hat{r}}{\sqrt{\det(\nab F_{\Tmac}^\intercal \nab F_{\Tmac})}}\circ F_{\Tmac}^{-1}
    =\frac{h^2_Tr}{\sqrt{\det(\nab F_{\Tmac}^\intercal \nab F_{\Tmac})\circ F_{\Tmac}^{-1}}}.
    \end{align*}
Finally, we make a change of variables to obtain
    \[
    \|\bw\|_{H^1(\Tmac)}\lesssim h^{-1}_T\|\hat{\bw}\|_{H^1(\hat{T})}\lesssim h_T\|\hat{r}\|_{L^2(\hat{T})}\lesssim \|r\|_{L^2(\Tmac)}. 
    \]
\end{proof}

\subsection{Auxiliary global macro finite element spaces}
Similar to the space $\bV_h$, defined on the Clough-Tocher triangulation,
we define  the global auxiliary Scott-Vogelius velocity
space as the space of functions which
are locally in $\bVmac(\Tmac)$ for each $\Tmac\in \calTmac_h$
and have weak continuity properties
at the degrees of freedom. 
Let 
\[
\calN_h^{\rm m,ct} = \{\bPsimac(\bar a):\ \bar a\in \bar \calN_h^{\rm ct}\}
\]
denote the locations of the $r$th-degree Lagrange degrees of freedom
on $\calT^{\rm m,ct}_h$.
For each $a\in \calN_h^{\rm m,ct}$, 
let $\calT^{\rm m,ct}_a \subset \calT^{\rm m,ct}_h$ be the set of elements
abutting $a$. We then set 
\[
K^{\rm m}_a= \bPsimac(\bPsict^{-1}(K_{\bPsict(\bPsimac^{-1}(a))}))
\]
as the selected element associated with the DOF $a$, that is,
the selected element corresponds to the selected element in the 
triangulation $\calT_h^{\rm ct}$.

Similar to \eqref{eqn:calMaK}, we define the operator ($a\in \calN_h^{\rm m,ct}$)
\begin{align}
\calM_a^{K^{\rm m}} {\bm x} = \left((\bnu_{K^{\rm m}_a}(a)\cdot \bnu_{K^{\rm m}}(a)){\bf I}_3 - \bnu_{K^{\rm m}_a}(a) \otimes \bnu_{K^{\rm m}}(a)\right) {\bm x}.
\end{align}
Note that if $a\in {\rm int}(\Tmac)$ for some $\Tmac\in \calTmac_h$,
then the outward unit normal at $a$ is single-valued
because $\bPsimac|_{\bar T}$ is a $C^\infty$ mapping (with $\Tmac = \bPsimac(\bar T)$).
Thus in this case $\bnu_{K^{\rm m}_a}(a) =\bnu_{K^{\rm m}}(a)$,
and so $\calM_a^{K^{\rm m}} = \bPi_{\Tmac}(a)$. In particular,
$\calM_a^{K^{\rm m}} {\bm x} = {\bm x}$ for all $\bx$
tangent to $\Gammamac_h$ at $a$.

The global macro spaces are given by
\begin{equation}
\label{eqn:GlobalMacSpaces}
\begin{split}
\bVmac_h & = \{\bw\in \bL^2(\Gammamac_h):\ \bw|_\Tmac\in \bVmac(\Tmac)\ \forall \Tmac\in \calTmac_h,\\ &\qquad\bw|_{K^{\rm m}}(a) = \calM_a^{K^{\rm m}}(\bw|_{K^{\rm m}_a}(a))\ \forall K^{\rm m}\in \calT^{\rm m,ct}_a,\ \forall a\in \calN^{\rm m,ct}_h\},\\
\Qmac_h & = \{r\in L^2_0(\Gammamac_h):\ r|_\Tmac \in \Qmac(\Tmac)\ \forall \Tmac\in \calTmac_h\}.
\end{split}
\end{equation}

\begin{remark}
Similar to the space $\bV_h$ (cf.~Remark \ref{rem:HdivRemark}),
there holds $\bVmac_h\subset \bH({\rm div}_{\Gammamac_h};\Gammamac_h)$
due to the arguments in 
\cite[Lemma 3.5]{DemlowNeilan25}. Furthermore, by
the preceding paragraph, there holds the following equivalent definition of
$\bVmac_h$, defined with with respect to the triangulation $\calT_h^{\rm m}$:
\begin{equation}\label{eqn:VmacAlt}
\begin{split}
\bVmac_h & = \{\bw\in \bL^2(\Gammamac_h):\ \bw|_\Tmac\in \bVmac(\Tmac)\ \forall \Tmac\in \calTmac_h,\\ &\qquad\bw|_{T^{\rm m}}(a) = \calM_a^{T^{\rm m}}(\bw|_{T^{\rm m}_a}(a))\ \forall T^{\rm m}\in \calT^{\rm m}_a,\ \forall a\in \calN^{\rm m}_h\},
\end{split}
\end{equation}
where $ \calN^{\rm m}_h = \{\bPsimac(\bar a):\ \bar a\in \bar \calN_h\}$,
$\bar \calN_h$ is the set of $r$th-degree degrees of freedom with respect to $\bar \calT_h$,
$\Tmaca\in \calT_h^{\rm m}$ is the unique element with $K^{\rm m}_a\subset \Tmaca$,
and $\calT_a^{\rm m}$ are the elements in $\calT_h^{\rm m}$ abutting $a$.
Similar to Assumption \ref{assume:GL},
we assume that the edge DOFs in $\bar \calN_h$ correspond
to the $(r+1)$-point Gauss-Lobatto integration rule.
The original definition \eqref{eqn:GlobalMacSpaces} will be used in the perturbation argument in Section \ref{sec:Correspond}, while the alternative definition \eqref{eqn:VmacAlt} will be used in the stability analysis of the pair $\bVmac_h\times Q_h^{\rm m}$. 
\end{remark}

\subsection{Inf-sup stability for auxiliary macro pair}
Similar to the Euclidean setting (cf.~\cite{GuzmanNeilan18}),
we prove stability of $\bVmac_h\times \Qmac_h$ via a macro-element
technique, which relies on a local inf-sup stability result
and the stability of the $\bbP_r-\bbP_0$ pair to ``glue'' the local
results together. The local stability result has been established
in Lemma \ref{lem:LocalMacroStable}, so we now focus
on the inf-sup stability of a surface analogue of the $\bbP_r-\bbP_0$ pair. 

We define
\begin{equation}
\label{eqn:TildeSpacesDef}
\begin{split}
\bVtildemac_h &= \{\bw\in \bL^2(\Gammamac_h):\ \bw|_{\Tmac} = \calP_{F_\Tmac} \hat \bw\ \exists \hat{\bw} \in [\bbP_r(\hat T)]^2;\\ &\qquad\bw|_\Tmac(a) = \calM^{\Tmac}_a (\bw|_{\Tmaca}(a))\ \forall T\in \calTmac_a,\ \forall a\in  \calN^{\rm m}_h\},\\
\Qtildemac_h & = \{r\in L^2_0(\Gamma_h):\ r|_{\Tmac} \in \mathbb{P}_0(\Tmac)\ \forall \Tmac\in \calTmac_h\}.
\end{split}
\end{equation}

\begin{remark}
Because $[\bbP_r(\hat T)]^2\subset \bVhatmac$, we easily see from \eqref{eqn:VmacAlt} that the inclusion
$\bVtildemac_h\subset \bVmac_h$ holds. 
\end{remark}

Continuing, we state an approximation property of $\bVtildemac_h$.
The proof of the next  lemma can be found in 
\cite[Lemma 3.9]{DemlowNeilan25}.
%
\begin{lemma}\label{lem:Interp2}
For each $\bz\in \bH^1_T(\gamma)$,
there exists $\bI_{h,1}  \bz\in \bVtildemac_h$ such that
\begin{align*}
\|\ipt\bz - \bI_{h,1} \bz\|_{H^\ell(\Tmac)}\lesssim h^{1-\ell} \|\bz\|_{H^1(\bp(\omega'_{T^{\rm m}}))}\ (\ell=0,1)\qquad \forall \Tmac\in \calTmac_h,
\end{align*}
where $\ipt \bz = \calP_{\bp^{-1}} \bz$ and $\omega'_{\Tmac} = \mathop{\cup_{S\in \calTmac_h}}_{{\rm cl}(S)\cap {\rm cl}(\Tmac)\neq \emptyset} \mathop{\cup_{S'\in \calTmac_h}}_{{\rm cl}(S')\cap {\rm cl}(S)\neq \emptyset} S'$ 
is the two-layered local patch around $\Tmac$.
\end{lemma}

\begin{lemma}\label{lem:P2P0}
There exists $\tilde \beta^{\rm m}>0$ independent of $h$ such that
\begin{align}\label{eqn:P2P0infsup}
\tilde \beta^{\rm m} \|r\|_{L^2(\Gammamac_h)}\le \sup_{\bw\in \bVtildemac_h\backslash \{0\}} \frac{\int_{\Gammamac_h} ({\rm div}_{\Gammamac_h} \bw) r}{\|\bw\|_{H^1_h(\Gammamac_h)}}\qquad \forall r\in \Qtildemac_h.
\end{align}
\end{lemma}
\begin{proof}
Fix $r\in \Qtildemac_h$, and let $\bz\in \bH^1_T(\gamma)$
satisfy ${\rm div}_\gamma \bz = (\mu_h^{-1} r)^\ell$
and $\|\bz\|_{H^1(\gamma)}\lesssim \|(\mu_h^{-1} r)^\ell\|_{L^2(\gamma)}\lesssim \|r\|_{L^2(\Gamma_h)}$.
Let $\bI_{h,1} \bz\in \bVtildemac_h$ be the approximation 
of $\ipt \bz = \calP_{\bp^{-1}} \bz$ defined in Lemma \ref{lem:Interp2}.

By the definition of $\bVtildemac_h$ \eqref{eqn:TildeSpacesDef},
any $\bw\in \bVtildemac_h$ is uniquely
determined by the values $\bw_{\Tmaca}(a)$ 
for all $a\in \calN_h^{\rm m}$, where we recall
$\calN_h^{\rm m}$ is the set of $\bPsimac$-mapped
$r$th-degree Lagrange nodes with respect to $\bar \calT_h$.
For each $e\in \calEmac_h$, we let $\{b_{e,j}\}_{j=1}^{r-1}\subset \calN_h^{\rm m}$ denote
the $r$th-degree Lagrange degrees of freedom
that lie in the interior of $e$. Likewise, for each $\Tmac\in \calT_h^{\rm m}$,
we set $\{c_{\Tmac,j}\}_{j=1}^{M_r}\subset \calN_h^{\rm m}$ to denote the $r$th degree Lagrange degrees of freedom that lie in the interior of $\Tmac$,
where $M_r = \frac12(r-1)(r-2)$. The Lagrange degrees of freedom 
for $\bVtildemac_h$ then has the decomposition
\[
\calN_h^{\rm m} = \calV_h^{\rm m}\cup \{b_{e,j}:\ e\in \calE_h^{\rm m},\ 1\le j\le r-1\}
\cup \{c_{\Tmac,j}:\ \Tmac\in \calT_h^{\rm m},\ 1\le j\le M_r\}.
\]
Next, let $\{\hat \omega_j\}_{j=0}^r$ be the weights
of the (closed) $(r+1)$-point Gauss-Lobatto integration rule on 
the interval $[0,1]$, where $\omega_0$ is the weight
for the node $0$, and $\omega_{r}$ is the weight
for the node $1$. We then set
\begin{equation}\label{eqn:AlphaDef}
\alpha:=\left(\sum_{j=1}^{r-1} \omega_j\right)^{-1} =\calO(1),
\end{equation}
which is well-defined because all of the Gauss-Lobatto weights
are positive.

We uniquely determine $\bw\in \bVtildemac_h$ such that
\begin{subequations}
\label{eqn:vAssign}
\begin{alignat}{2}
\label{eqn:vAssign1}
\bw|_{T^{\rm m}_a}(a)&=(\bI_{h,1}\bz)|_{T^{\rm m}_a}(a)\qquad &&\forall a\in \calVmac_h,\\
\label{eqn:vAssign2}
\bw|_{T^{\rm m}_{b_{e,j}}}(b_{e,j})&=(\bI_{h,1}\bz)|_{T^{\rm m}_{b_{e,j}}}(b_{e,j})\\
&\qquad \nonumber+\alpha |\bar e|^{-1}\calP_{\bPsimac|_{\bar T_{b_{e,j}}}(\bar{b}_{e,j})}\int_{\bar{e}}\calP_{\bPsimac^{-1}}\left(\ipt\bz-{\bI_{h,1}\bz}\right)|_{\bar T_{b_{e,j}}} \quad &&\forall e\in \calEmac_h,\ j=1,\ldots,r-1,\\
\bw(c_{\Tmac,j})& = (\bI_{h,1}\bz)(c_{\Tmac,j})\qquad &&\forall \Tmac\in \calT_h^{\rm m},\ j=1,\ldots,M_r.
\end{alignat}
\end{subequations}
Here, $\bar{T}_{b_{e,j}} = \bPsimac^{-1}(T^{\rm m}_{b_{e,j}})$ 
 $\bar e = \bPsimac^{-1}(e)$, and $\bar b_{e,j} = \bPsimac^{-1}(b_{e,j})$.
Note that the right-hand sides of \eqref{eqn:vAssign} are tangent
to $\Gammamac_h$ at their respective points,
and therefore $\bw$ is well-defined. Furthermore, because $\bI_{h,1} \bz\in \bVtildemac_h$, the
first assignment \eqref{eqn:vAssign1} implies, by the definition of $\bVtildemac_h$,
\begin{align*}
\bw|_{\Tmac}(a)& = \calM_a^{T^{\rm m}}(\bw|_{T^{\rm m}_a}(a))
= \calM_a^{T^{\rm m}}((\bI_{h,1}\bz)|_{T^{\rm m}_a}(a))
= (\bI_{h,1} \bz)|_{T^{\rm m}}(a)\qquad \forall \Tmac\in \calTmac_a,\ \forall a\in \calVmac_h.
\end{align*}

Set $\bar \bw = \calP_{\bPsimac^{-1}} \bw$, 
$\overline{\bI_{h,1} \bz} = \calP_{\bPsimac^{-1}} (\bI_{h,1} \bz)$, and $\bar{\ipt \bz} = \calP_{\bPsimac^{-1}} \ipt \bz$
so that
\begin{subequations}
\label{eqn:vbarAssign}
\begin{alignat}{2}
\label{eqn:vbarAssign1}
\bar\bw|_{\bar{T}}(\bar a)&=(\overline{\bI_{h,1}\bz})|_{\bar{T}}(\bar a)\qquad &&\forall \bar T\in \bar \calT_{\bar a},\ \forall \bar a\in \bar \calV_h,\\
\bar \bw|_{\bar T_{b_{e,j}}}(\bar b_{e,j})&=(\overline{\bI_{h,1}\bz})|_{\bar T_{b_{e,j}}}(\bar b_{e,j})\\
&\nonumber\qquad+\alpha|\bar e|^{-1}\int_{\bar{e}}\left(\bar{\ipt\bz}-\overline{\bI_{h,1}\bz}\right)|_{\bar T_{b_{e,j}}} \qquad &&\forall \bar e\in \bar \calE_h,\ j=1,\ldots,r-1,\\
\bar \bw(\bar c_{\Tmac,j})& = (\overline{\bI_{h,1}\bz})(\bar c_{\Tmac,j})\qquad &&\forall \bar T\in \bar \calT_h,\ j=1,\ldots,M_r.
\end{alignat}
\end{subequations}
Because $\bw,\bI_{h,1}\bz,\ipt\bz \in \bH({\rm div}_{\Gammamac_h};\Gammamac_h)$,
there holds $\bar \bw, \overline{\bI_{h,1} \bz}, \overline{\ipt \bz}\in \bH({\rm div}_{\bar \Gamma_h};\bar \Gamma_h)$ by properties of the Piola transform.
In particular, these functions have co-normal continuity across edges,
and so for any co-normal $\bar \bn_e$ at $\bar e\in \bar \calE_h$, there holds
\begin{align}\label{eqn:vbarAssign2B}
(\bar \bw|_{\bar T}\cdot \bar \bn_e)(\bar b_{e,j})&=(\overline{\bI_{h,1}\bz}|_{\bar T}\cdot \bar \bn_e)
(\bar b_{e,j})+\alpha|\bar e|^{-1}\int_{\bar{e}}\left(\bar{\ipt\bz}-\overline{\bI_{h,1}\bz}\right)|_{\bar T}\cdot \bar \bn_e \quad \forall \bar e\in \bar \calE_h,\ j=1,\ldots,r-1,
\end{align}
for any $\bar T\in \bar \calT_h$ with $\bar e\subset \p T$.
Consequently, because $\bar \bw_{\bar T} \cdot \bar \bn_e$
and $(\overline{\bI_{h,1}\bz})_{\bar T}\cdot \bar \bn_e$ 
are polynomials of degree $\le r$ on $\bar e$,
we have by the Gauss-Lobatto integration rule, \eqref{eqn:vbarAssign1}, \eqref{eqn:vbarAssign2B}, and \eqref{eqn:AlphaDef},
\begin{align*}
\int_{\bar e} \bar \bw|_{\bar T} \cdot \bar \bn_e
 = \int_{\bar e}(\overline{\bI_{h,1}\bz})|_{\bar T}\cdot \bar \bn_e+\sum_{j=1}^{r-1} \alpha \omega_j
 \int_{\bar e} \left(\bar{\ipt\bz}-\overline{\bI_{h,1}\bz}\right)|_{\bar T}\cdot \bar \bn_e
 = \int_{\bar e} \bar{\ipt\bz}|_{\bar T} \cdot \bar \bn_e.
\end{align*}
Thus,
\begin{align*}
\int_{\p \bar T} \bar \bw\cdot  \bn_{\p \bar T} = \int_{\p \bar T} 
\bar{\ipt \bz} \cdot \bn_{\p \bar T}\qquad \forall \bar T\in \bar \calT_h,
\end{align*}
where $\bn_{\p \bar T}$ is the outward unit co-normal of $\bar T$.
Using properties of the Piola transform once again, along with the divergence theorem, we find (with $\Tmac = \bPsimac(\bar T)$)
\begin{align*}
\int_{\Tmac} {\rm div}_{\Gammamac_h} \bw = \int_{\p \Tmac} \bw\cdot \bn_{\p \Tmac}
= \int_{\p \bar T} \bar \bw\cdot \bn_{\p \bar T}  = \int_{\p \bar T} 
\bar{\ipt \bz} \cdot \bn_{\p \bar T} = \int_{\p \Tmac} \ipt \bz \cdot \bn_{\p \Tmac} = \int_\Tmac {\rm div}_{\Gammamac_h} \ipt \bz.
\end{align*}
Using that $r$ is piecewise constant, we conclude
\begin{equation}\label{eqn:Divvq}
\begin{split}
\int_{\Gammamac_h} ({\rm div}_{\Gammamac_h} \bw)r
& = \sum_{\Tmac\in \calTmac_h} r|_{\Tmac} \int_{\Tmac} {\rm div}_{\Gammamac_h} \bw\\
& = \sum_{\Tmac\in \calTmac_h} r|_{\Tmac} \int_{\Tmac} {\rm div}_{\Gammamac_h} \ipt \bz\\
& = \int_{\Gammamac_h} ({\rm div}_{\Gammamac_h} \ipt \bz) r
= \int_\gamma ({\rm div}_\gamma \bz)r^\ell  = \int_\gamma \mu_h^{-1} |r^\ell|^2 = \|r\|_{L^2(\Gammamac_h)}^2.
\end{split}
\end{equation}

Next, we make a change of variables and 
a standard scaling argument to obtain
\begin{align}\label{eqn:vboundStart}
\|\bw- \bI_{h,1} \bz\|_{H^1(\Tmac)}^2
&\lesssim \|\bar \bw - \overline{\bI_{h,1} \bz} \|_{H^1(\bar T)}^2
\lesssim \sum_{\bar e \subset \p \bar T} \sum_{j=1}^{r-1} |(\bar \bw - \overline{\bI_{h,1} \bz})_{\bar T}(\bar b_{e,j})|^2,
\end{align}
where we used that $\bar \bw  = \overline{\bI_{h,1} \bz}$
at vertices and interior DOFs of the affine mesh (cf.~\eqref{eqn:vbarAssign}).

Let $e$ be a fixed (curved) edge 
of $\Tmac$, and let $\bar e = \bPsimac^{-1}(e)$ be the corresponding
affine edge. Using the definition of the finite element space 
$ \bVtildemac_h$, we have
\begin{align*}
(\bar{\bw}-\overline{\bI_{h,1}\bz})|_{\bar{T}}(\bar b_{e,j})
&= \calP_{\bPsimac^{-1}(\bar b_{{e,j}})}\left((\bw-\bI_{h,1}\bz)|_{\Tmac}(b_{e,j})\right)\\
&=\calP_{\bPsimac^{-1}(\bar b_{{e,j}})}\calM_{b_{e,j}}^{\Tmac}\left(\bw-(\bI_{h,1}\bz)|_{T^{\rm m}_{b_{e,j}}}(b_{e,j}) \right),
\end{align*}
and therefore,
\begin{align}\label{eqn:vwTe} 
\left|(\bar{\bw}-\overline{\bI_{h,1}\bz})|_{\bar {T}}(\bar b_{e,j}) \right|
&\leq \|\calP_{\bPsimac^{-1}(\bar b_{e,j})} \| \cdot \|\calM_{b_{e,j}}^{\Tmac} \| \cdot |(\bw-\bI_{h,1}\bz)|_{T^{\rm m}_{b_{e,j}}}(b_{e,j}) |.
\end{align}

From \cite[eqn.~(2.6)]{DemlowNeilan25}, we have that $\|\calP_{\bPsimac^{-1}(\bar b_{{e,j}})}\|\lesssim 1 $, and we clearly have $\|\calM_{b_{e,j}}^T\|\lesssim 1$. Inserting
these estimates into \eqref{eqn:vwTe} and using \eqref{eqn:vbarAssign}
yields
\begin{align*}
\left|(\bar{\bw}-\overline{\bI_{h,1}\bz})|_{\bar {T}}(\bar b_{e,j}) \right|
&\lesssim |(\bw-\bI_{h,1}\bz)|_{T^{\rm m}_{b_{e,j}}}(b_{e,j}) |
=\alpha |\bar{e}|^{-1}\left|\int_{\bar{e}}(\bar{\ipt\bz}-\overline{\bI_{h,1}\bz})|_{\overline{T}_{b_{e,j}}} \right|
%
\lesssim h_T^{-1/2} \|(\overline{\ipt\bz}-\overline{\bI_{h,1}\bz})|_{\overline{T}_{b_{e,j}}}\|_{L^2(\bar e)}.
\end{align*}
We then apply standard trace estimates
and Lemma \ref{lem:Interp2} to obtain
\begin{align*}
\left|(\bar{\bw}-\overline{\bI_{h,1}\bz})|_{\bar {T}}(\bar b_{e,j}) \right|
&\lesssim h_T^{-1}\big\|\bar{\ipt\bz}-\overline{\bI_{h,1}\bz}\big\|_{L^2(\bar T_{b_{e,j}})}
+\big\|\bar{\ipt\bz}-\overline{\bI_{h,1}\bz}\big\|_{H^1(\bar T_{b_{e,j}})}\\ 
 &\lesssim h_T^{-1}\big\|\ipt\bz-\bI_{h,1}\bz\big\|_{L^2(T^{\rm m}_{b_{e,j}})}+\big\|\ipt\bz-\bI_{h,1}\bz\big\|_{H^1(T^{\rm m}_{b_{e,j}})}\\
&\lesssim \|\bz\|_{H^1(\bp(\omega_{T_{b_{e,j}}^{\rm m}}')}
\lesssim \|\bz\|_{H^1(\bp(\omega_{T^{\rm m}}^{\prime\prime}))},
\end{align*}
where $\omega_{T^{\rm m}}^{\prime\prime}$ is a three-layered local patch
around $\Tmac$.

Applying this estimate to \eqref{eqn:vboundStart}
and once again using Lemma \ref{lem:Interp2}, we get
\begin{align*}
\|\bw\|_{H^1(\Tmac)}\lesssim \|\bz\|_{H^1(\bp(\omega_{T^{\rm m}}^{\prime\prime})}.
\end{align*}
Summing over $T\in \calTmac_h$, we then conclude $\|\bw\|_{H^1_h(\Gammamac_h)}\lesssim \|\bz\|_{H^1(\gamma)}\lesssim \|r\|_{L^2(\Gammamac_h)}$,
which combined with \eqref{eqn:Divvq} yields
the desired inf-sup condition \eqref{eqn:P2P0infsup}.
\end{proof}

The local stability result
given in Lemma \ref{lem:LocalMacroStable}, combined with the stability
of the $\bbP_r-\bbP_0$ pair, shows
that the auxiliary pair $\bVmac_h\times \Qmac_h$
is inf-sup stable.

\begin{theorem}\label{thm:MacStable}
There exists $\beta_{\rm m}>0$ independent
of the discretization parameter $h$ such that
\begin{align}\label{eqn:Macroinfsup}
 \beta_{\rm m} \|r\|_{L^2(\Gammamac_h)}\le \sup_{\bw\in \bVmac_h\backslash \{0\}} \frac{\int_{\Gammamac_h} ({\rm div}_{\Gammamac_h} \bw) r}{\|\bw\|_{H^1_h(\Gammamac_h)}}\qquad \forall r\in \Qmac_h.
\end{align}
\end{theorem}
\begin{proof}
The result essentially follows directly from Lemmas \ref{lem:LocalMacroStable} and \ref{lem:P2P0}, together with the arguments given in \cite[Theorem 4.4]{NeilanOtus21}, \cite[Theorem 4.9]{DurstNeilan24}, and \cite[Proposition 6.1]{GuzmanNeilan18}. For completeness, we provide the proof.

Let $r\in \Qmac_h$, and let $\tilde r_0$ be the piecewise constant function with respect to $\calT_h^{\rm m}$ that satisfies
\[
\int_{\hat T} (r-\tilde r_0)\circ F_{\Tmac} = \int_{\Tmac} J_{\Tmac}^{-1} (r-\tilde r_0)=0\qquad \forall \Tmac\in \calT_h^{\rm m} ,
\]
where $J_{\Tmac} = \sqrt{\det((\nab F_{\Tmac}^\intercal \nab F_{\Tmac})\circ F_{\Tmac}^{-1})}$.
Then $(r-\tilde r_0)|_{\Tmac}\in Q_0^{\rm m}(\Tmac)$ for each $\Tmac\in \calT_h^{\rm m}$,
and so, by Lemma \ref{lem:LocalMacroStable}, there exists $\bw_{0,\Tmac}\in \bV_0^{\rm m}(\Tmac)$ such
that ${\rm div}_{\Gamma_h^{\rm m}} \bw_{0,\Tmac} = h_T^2 (r-\tilde r_0)/J_{\Tmac}$ and $\|\bw_{0,\Tmac}\|_{H^1(\Tmac)}\lesssim \|r-\tilde r_0\|_{L^2(\Tmac)}$.
We then set $\bw_0\in \bV_h^{\rm m}$ such that $\bw_0|_{\Tmac} = \bw_{0,\Tmac}$ for all $\Tmac\in \calT_h^{\rm m}$. 
Then $\|\bw_0\|_{H^1_h(\Gammamac_h)}\lesssim \|r-\tilde r_0\|_{L^2(\Gammamac_h)}$ and
\begin{align*}
 \sup_{\bw\in \bVmac_h\backslash \{0\}} \frac{\int_{\Gammamac_h} ({\rm div}_{\Gammamac_h} \bw) r}{\|\bw\|_{H^1_h(\Gammamac_h)}}
  &\ge   \frac{\int_{\Gammamac_h} ({\rm div}_{\Gammamac_h} \bw_0) (r-\tilde r_0)}{\|\bw_0\|_{H^1_h(\Gammamac_h)}}
=   \frac{\sum_{\Tmac\in \calT_h^{\rm m}} h_T^2 \int_{\Tmac}|r-\tilde r_0|^2/J_T}{\|\bw_0\|_{H^1_h(\Gammamac_h)}} \ge \beta_{{\rm m},0} \|r-\tilde r_0\|_{L^2(\Gamma^{\rm m}_h)}
\end{align*}
for some constant $\beta_{{\rm m},0}$ independent of the mesh parameter $h$.

Next, we set
\[
\gamma:=|\Gammamac_h|^{-1}\int_{\Gammamac_h} \tilde r_0 = |\Gammamac_h|^{-1}\sum_{\Tmac\in \calT_h^{\rm m}} \int_{\Tmac} \left(J^{-1}_{\Tmac}-1\right)(r-\tilde r_0)
\]
so that $r_0:=\tilde r_0-\gamma\in \Qtildemac_h$ and $\|\gamma\|_{L^2(\Gammamac_h)}\lesssim h^2 \|r-\tilde r_0\|_{L^2(\Gamma)}$. Therefore by Lemma \ref{lem:P2P0},
\begin{align*}
\beta_{{\rm m},1} \|r_0\|_{L^2(\Gammamac_h)}
&\le  \sup_{\bw\in \bVtildemac_h\backslash \{0\}} \frac{\int_{\Gammamac_h} ({\rm div}_{\Gammamac_h} \bw) r_0}{\|\bw\|_{H^1_h(\Gammamac_h)}}
\le  \sup_{\bw\in \bVmac_h\backslash \{0\}} \frac{\int_{\Gammamac_h} ({\rm div}_{\Gammamac_h} \bw) r_0}{\|\bw\|_{H^1_h(\Gammamac_h)}}\\
&\le \sup_{\bw\in \bVmac_h\backslash \{0\}} \frac{\int_{\Gammamac_h} ({\rm div}_{\Gammamac_h} \bw) r}{\|\bw\|_{H^1_h(\Gammamac_h)}}+\|r-r_0\|_{L^2(\Gammamac_h)}\\
&\lesssim \sup_{\bw\in \bVmac_h\backslash \{0\}} \frac{\int_{\Gammamac_h} ({\rm div}_{\Gammamac_h} \bw) r}{\|\bw\|_{H^1_h(\Gammamac_h)}}+(1+h^2)\|r-\tilde r_0\|_{L^2(\Gammamac_h)}\\
&\lesssim (1+\beta_{{\rm m},0}^{-1})\sup_{\bw\in \bVmac_h\backslash \{0\}} \frac{\int_{\Gammamac_h} ({\rm div}_{\Gammamac_h} \bw) r}{\|\bw\|_{H^1_h(\Gammamac_h)}}.
\end{align*}

Therefore 
\begin{align*}
\|r\|_{L^2(\Gammamac_h)}
&\le \|r-\tilde r_0\|_{L^2(\Gammamac_h)}+\|\tilde r_0\|_{L^2(\Gammamac_h)}\\
&\le \|r-\tilde r_0\|_{L^2(\Gammamac_h)}+\|r_0\|_{L^2(\Gammamac_h)}+\|\gamma\|_{L^2(\Gammamac_h)}\\
&\lesssim \|r-\tilde r_0\|_{L^2(\Gammamac_h)}+\|r_0\|_{L^2(\Gammamac_h)}\\
&\lesssim \left(\beta_{{\rm m},0}^{-1} +\beta_{{\rm m},1}^{-1} (1+\beta_{{\rm m},0}^{-1})\right)\sup_{\bw\in \bVmac_h\backslash \{0\}} \frac{\int_{\Gammamac_h} ({\rm div}_{\Gammamac_h} \bw) r}{\|\bw\|_{H^1_h(\Gammamac_h)}}.
\end{align*}
\end{proof}

\subsection{Correspondences between $\bV_h$ and $\bVmac_h$ and $Q_h$ and $\Qmac_h$}\label{sec:Correspond}
The following lemma shows that the finite element
pairs $\bV_h\times Q_h$ and $\bVmac_h\times \Qmac_h$
are $\calO(h)$ perturbations of each other. 
This result, along with the inf-sup stability 
estimate in Theorem \ref{thm:MacStable}
is the basis of the proof of Theorem \ref{thm:MainStab}.
Its proof is given in Appendix \ref{App:Pert}.
\begin{lemma}\label{lem:PITA}
Given $\bv^m\in \bVmac_h$, there exists $\bv\in \bV_h$ such that
\begin{equation}\label{eqn:vPertEst}
\|\bv\circ \bPsict - \bv^{\rm m}\circ \bPsimac\|_{H^1_h(\bar \Gamma_h)}\lesssim  h \|\bv^{\rm m}\|_{H^1_h(\Gammamac_h)}.
\end{equation}
Likewise, given $q\in Q_h$, there exists $q^{\rm m}\in \Qmac_h$ such that
\begin{equation}\label{eqn:qPertEst}
\|q\circ \bPsict - q^{\rm m}\circ \bPsimac\|_{L^2(\bar \Gamma_h)}\lesssim  h \|q\|_{L^2(\Gamma_h)}.
\end{equation}
\end{lemma}

\subsection{Proof of Theorem \ref{thm:MainStab}}

Let $q\in Q_h$ and let $q^{\rm m}\in \Qmac_h$ 
satisfy \eqref{eqn:qPertEst}. We then find
\begin{align*}
\|q\|_{L^2(\Gamma_h)}
&\lesssim \|q \circ \bPsict\|_{L^2(\bar \Gamma_h)}\\
&\le \|q\circ \bPsict - q^{\rm m}\circ \bPsimac\|_{L^2(\bar \Gamma_h)}+\|q^{\rm m}\circ \bPsimac\|_{L^2(\bar \Gamma_h)}\\
&\lesssim h \|q\|_{L^2(\Gamma_h)}+ \|q^{\rm m}\|_{L^2(\Gammamac_h)},
\end{align*}
and so $\|q\|_{L^2(\Gamma_h)}
\lesssim \|q^{\rm m}\|_{L^2(\Gammamac_h)}$ for $h$ sufficiently small.

By the inf-sup stability result of the auxiliary pair
\eqref{eqn:Macroinfsup}, there exists $\bv^{\rm m}\in \bV_h^{\rm m}$
such that 
\[
\beta_m\|q^{\rm m}\|_{L^2(\Gammamac_h)} \le \int_{\Gammamac_h} ({\rm div}_{\Gammamac_h} \bv^{\rm m}) q^{\rm m},\quad\text{and}\quad \|\bv^{\rm m}\|_{H^1_h(\Gammamac_h)} = 1.
\]
We then let $\bv\in \bV_h$ satisfy \eqref{eqn:vPertEst}. 
Therefore for $h$ sufficiently small,
\begin{align}\label{eqn:MainStep1}
\|q\|_{L^2(\Gamma_h)}
&\lesssim \|q^{\rm m}\|_{L^2(\Gammamac_h)}\le \beta_{\rm m}^{-1}\int_{\Gammamac_h} ({\rm div}_{\Gammamac_h} \bv^{\rm m}) q^{\rm m}.
\end{align}
Set $\bar \bv^m = \bv^m\circ \bPsimac:\bar \Gamma_h\to \bbR^3$
and $\bar \bv = \bv\circ \bPsict:\Gamma_h\to \bbR^3$, so that
by \eqref{eqn:vPertEst}, $\|\bar \bv - \bar \bv^{\rm m}\|_{H^1_h(\bar \Gamma_h)}\lesssim h \|\bvmac\|_{H^1_h(\Gammamac_h)}$.
We then compute via a change of variables and the chain rule,
\begin{align*}
&\int_{\Gammamac_h} ({\rm div}_{\Gammamac_h} \bv^{\rm m}) q^{\rm m}
- \int_{\Gamma_h} ({\rm div}_{\Gamma_h} \bv) q\\
&\qquad = \int_{\bar \Gamma_h} {\rm tr}(\bPi_h^{\rm m} \bar \nab \bar \bv^{\rm m}\nab\bPsimac^{-1}\bPi^{\rm m}_h) q^{\rm m}\circ \bPsimac J_{\bPsimac}
- \int_{\bar \Gamma_h} {\rm tr}(\bPi_h \bar \nab \bar \bv \nab\bPsict^{-1}\bPi_h) q \circ \bPsict J_{\bPsict}.
\end{align*}
Here $\bar \bv^{\rm m}$ and $\bar \bv$ are extensions of $\bv^{\rm m}$ and $\bv$, respectively,
and $J_{\bPsict}$ is the product of non-zero singular values of $\nab\bPsict$.
We then add and subtract terms, use \eqref{eqn:bPsictBounds}, \eqref{eqn:PsiMacInterpProp}, 
and \eqref{eqn:vPertEst}--\eqref{eqn:qPertEst} to obtain
\begin{align*}
&\left|\int_{\Gammamac_h} ({\rm div}_{\Gammamac_h} \bv^{\rm m}) q^{\rm m}
- \int_{\Gamma_h} ({\rm div}_{\Gamma_h} \bv) q\right|\\
&\qquad \lesssim 
 \int_{\bar \Gamma_h} {\rm tr}(\bPi_h^{\rm m} \bar \nab \bar \bv^{\rm m}\nab\bPsimac^{-1}\bPi^{\rm m}_h) (q^{\rm m}\circ \bPsimac J_{\bPsimac}- q\circ \bPsict J_{\bPsict})\\
 &\qquad\qquad 
+ \int_{\bar \Gamma_h} {\rm tr}(\bPi_h^{\rm m} \bar \nab \bar \bv^{\rm m}\nab\bPsimac^{-1}\bPi^{\rm m}_h-\bPi_h \bar \nab \bar \bv \nab\bPsict^{-1}\bPi_h)q\circ \bPsict J_{\bPsict}\\
&\lesssim h \|\bvmac\|_{H^1_h(\Gammamac_h)} \|q\|_{L^2(\Gamma_h)} = h\|q\|_{L^2(\Gamma_h)}.
\end{align*}

Applying this estimate in \eqref{eqn:MainStep1} we conclude that, for $h$ sufficiently small,
\begin{align}\label{eqn:MainStep2}
\|q\|_{L^2(\Gamma_h)}
&\lesssim \int_{\Gamma_h} ({\rm div}_{\Gamma_h} \bv) q.
\end{align}
We also have by \eqref{eqn:vPertEst},
\begin{align*}
    \|\bv\|_{H^1_h(\Gamma_h)}\lesssim \|\bv\circ \bPsict\|_{H^1_h(\bar \Gamma_h)}\le  \|\bv^{\rm m}\circ \bPsimac\|_{H^1_h(\bar \Gamma_h)} + h \|\bv^{\rm m}\|_{H^1_h(\Gammamac_h)}\lesssim (1+h)\|\bv^{\rm m}\|_{H^1_h(\Gammamac_h)} = (1+h).
\end{align*}
We combine this estimate with \eqref{eqn:MainStep2} to conclude the proof.

\section{Finite Element Method and Convergence Analysis}\label{sec:Converge}
The finite element method seeks $(\bu_h,p_h)\in \bV_h\times Q_h$
such that
\begin{equation}
\label{eqn:FEM}
\begin{aligned}
     a_h(\bu_h,\bv)
    -\int_{\Gamma_h} ({\rm div}_{\Gamma_h} \bv)p_h & = \int_{\Gamma_h} {\bm f}_h\cdot \bv\qquad &&\forall \bv\in \bV_h,\\
    \int_{\Gamma_h} ({\rm div}_{\Gamma_h} \bu_h)q & =0\qquad &&\forall q\in Q_h,
    \end{aligned}
    \end{equation}
where 
\[
a_h(\bw,\bv) =  \int_{\Gamma_h}\left( {\rm Def}_{\Gamma_h} \bw:{\rm Def}_{\Gamma_h} \bv+ \bw\cdot \bv\right),
\]
and ${\bm f}_h:\Gamma_h\to \bbR^3$ is an approximation of ${\bm f}$.
It follows from Theorem \ref{thm:MainStab} and a discrete Korn-like inequality
(cf.~\cite[Lemma 4.6]{DemlowNeilan25}) that there exists a unique
solution to \eqref{eqn:FEM}. Further note that the approximated
velocity is exactly divergence free, i.e., ${\rm div}_{\Gamma_h} \bu_h =0$; 
see Lemma \ref{lem:DivFreeProp}.

\subsection{Consistency Estimates}\label{sec:Consistency}

Because the velocity space of the Scott-Vogelius
and Taylor-Hood pairs coincide, we
have the following result in \cite[Lemma 4.5]{DemlowNeilan25},
which provides estimates quantifying 
the weak continuity properties of $\bV_h$.
\begin{lemma}
Suppose that Assumption \ref{assume:GL} is satisfied.
For $\bv\in \bV_h$,
let $\pt \bv = \calP_{\bp} \bv$ denote its Piola transform
with respect to $\bp$. Then for all $\bw\in \bH^{s}(\gamma)$
with $2\le s\le r$, there holds for $m=0,1,\ldots,r$,
\begin{align}\label{eqn:KrisKross}
\sum_{e\in \calE_h^{\rm ct}} \int_{\bp(e)} {\rm Def}_\gamma \bw: \jump{\pt \bv}\lesssim h^{s+m-1} \|\bw\|_{H^{s}(\gamma)} \|\pt \bv\|_{H^m_h(\gamma)} + h^{2k-1} \|\bw\|_{H^2(\gamma)} \|\pt \bv\|_{L^2(\gamma)}.
\end{align}
Here, the jump of $\pt \bv$ is given by
\[
\jump{\pt \bv}|_{\bp(e)} = \pt \bv_+\otimes \bn_+^\gamma+ \pt\bv_-\otimes \bn_-^\gamma,
\]
where $e = \p K_+\cap \p K_-\ (K_{\pm}\in \calT_h^{\rm ct})$, $\pt \bv|_{\pm} = \pt \bv|_{\bp(K_{\pm})}$, and $\bn_{\pm}^\gamma$ is the outward unit co-normal
with respect to $\bp (\p K_{\pm})$.
\end{lemma}

Next, we state the geometric inconsistency estimates given in
\cite[Lemmas 5.1--5.2]{DemlowNeilan25}.
\begin{lemma}
Define 
\[
G_h(\bv,\bw) = a(\pt \bv,\pt \bw) - a_h(\bv,\bw),
\]
where $a(\pt \bv,\pt \bw) = \int_\gamma \big({\rm Def}_\gamma \pt \bv:{\rm Def}_\gamma \pt \bw+\pt\bv\cdot \pt \bw\big)$.
There holds
\begin{alignat}{2}\label{eqn:GeomEst1}
|G_h(\bv,\bw)|&\lesssim h^k \|\pt \bv\|_{H^1_h(\gamma)} \|\pt \bw\|_{H^1_h(\gamma)} \qquad &&\forall \bv,\bw\in \bH^1(\calT^{\rm ct}_h),\\
\label{eqn:GeomEst2}
|G_h(\ipt \bv,\ipt \bw)|&\lesssim h^{k+1} \| \bv\|_{H^2(\gamma)}  \|\bw\|_{H^2(\gamma)} \qquad &&\forall \bv,\bw\in \bH^2(\gamma)\cap \bH^1_T(\gamma).
\end{alignat}
\end{lemma}

\subsection{Convergence analysis}

Let $\bX_h = \{\bv\in \bV_h:\ {\rm div}_{\Gamma_h} \bv\}\subset \bV_h$ be the divergence-free subspace of $\bV_h$, and set
\[
\|{\bm f}\|_{X_h^*} :=\sup_{\bv\in \bX_h} \frac{\int_\gamma {\bm f}\cdot \pt \bv}{\|\pt \bv\|_{H^1_h(\gamma)}}.
\]
Note that $\bu_h\in \bX_h$ and $\int_{\Gamma_h} {\bm f}_h\cdot \bv = \int_\gamma {\bm F}_h^\ell \cdot \pt \bv$ 
for all $\bv\in \bV_h$, where \cite[(5.7)]{DemlowNeilan25}
\begin{equation}\label{eqn:FhDeff}
{\bm F}_h = \big[{\bf I}- d {\bf H}]^{-1} \left[{\bf I}- \frac{\bnu_h\otimes \bnu}{\bnu_h\cdot \bnu}\right] {\bm f}_h.
\end{equation}
Thus, by setting $\bv = \bu_h$ in \eqref{eqn:FEM},
we conclude from a discrete Korn inequality
that $\|\pt \bu_h\|_{H^1_h(\gamma)}\lesssim \|{\bm F}_h^\ell\|_{X_h^*}$.

\begin{theorem}\label{thm:EnergyEst}
Let Assumption \ref{assume:GL} hold.
Suppose that $(\bu,p)\in \bH^s(\gamma)\times H^{s-1}(\gamma)$ for some $2\le s\le r+1$. Then there holds
\begin{align}\label{eqn:H1error}
\|\bu-\pt \bu_h\|_{H^1_h(\gamma)}&\lesssim  (h^{s-1}+h^{2k-1})\|\bu\|_{H^{s}(\gamma)}
+\|{\bm f}-{\bm F}_h^\ell\|_{X_h^*} + h^k \|{\bm F}_h^\ell\|_{X_h^*},\\
\label{eqn:PL2error}
\|p-p_h^\ell\|_{L^2(\gamma)} & \lesssim (h^{s-1}+h^{2k-1})(\|\bu\|_{H^s(\gamma)}+\|p\|_{H^{s-1}(\gamma)})  + \|{\bm f}-{\bm F}_h^\ell\|_{L^2(\gamma)}+h^k \|{\bm F}_h^\ell\|_{X_h^*}.
\end{align}
Consequently, if $\|{\bm f}-{\bm F}_h^\ell\|_{L^2(\gamma)}=\mathcal{O}(h^k)$, there holds
\[
\|\bu-\pt \bu_h\|_{H^1_h(\gamma)}+\|p-p_h\|_{L^2(\gamma)} = \mathcal{O}(h^{s-1}+h^k).
\]
\end{theorem}
\begin{proof}
Given $\bv\in \bX_h$, we compute
\begin{align*}
a(\bu-\pt \bu_h,\pt \bv)
&= a(\bu,\pt \bv) - \int_\gamma {\bm f}\cdot \pt \bv + \int_\gamma ({\bm f} - {\bm F}_h^\ell)\cdot \pt \bv + G_h(\bu_h,\bv)\\
& = \sum_{e\in \calE_h^{\rm ct}} \int_{\bp(e) } ({\rm Def}_\gamma \bu):\jump{\pt \bv}+ \int_\gamma ({\bm f} - {\bm F}_h^\ell)\cdot \pt \bv + G_h(\bu_h,\bv),
\end{align*}
where we applied integration-by-parts and used the pointwise divergence-free property of $\pt \bv$ in the last equality. We 
then conclude from \eqref{eqn:GeomEst1} (with $m=1$) and \eqref{eqn:KrisKross} that
\begin{align*}
a(\bu-\pt \bu_h,\pt \bv) &\lesssim \left(h^{s-1}\|\bu\|_{H^{s}(\gamma)} + h^{2k-1} \|\bu\|_{H^2(\gamma)}
+\|{\bm f}-{\bm F}_h^\ell\|_{X_h^*} + h^k \|{\bm F}_h^\ell\|_{X_h^*}\right)\|\pt \bv\|_{H^1_h(\gamma)}.
\end{align*}
We can then apply 
Strang's Second lemma, approximation properties
of $\bV_h$, and standard arguments to conclude
\eqref{eqn:H1error}.

Next, for arbitrary $q\in Q_h,$ we have for all $\bv\in\bV_h$,
\begin{align}\label{eqn:PressureStart}
\int_{\Gamma_h} ({\rm div}_{\Gamma_h} \bv)(q-p_h) = \int_\gamma {\bm F}_h^\ell\cdot \pt \bv - a(\pt \bu_h,\pt \bv) + \int_\gamma ({\rm div}_\gamma \pt \bv) q^\ell+G_h(\bu_h,\bv).
\end{align}
We then write, using \eqref{eqn:KrisKross} and \eqref{eqn:H1error},
\begin{align*}
&\int_\gamma {\bm F}_h^\ell\cdot \pt\bv - a(\pt \bu_h,\pt \bv) + \int_\gamma ({\rm div}_\gamma \pt \bv) q^\ell\\
&\qquad
=\int_\gamma {\bm f}\cdot \pt \bv - a(\bu,\pt \bv)+\int_{\Gamma_h} ({\rm div}_\gamma \pt \bv)p+\int_\gamma ({\bm F}_h^\ell - {\bm f})\cdot \pt \bv + a(\bu-\pt \bu_h,\pt \bv)+\int_{\gamma} ({\rm div}_\gamma \pt \bv)(q^\ell-p)\\
&\qquad
=-\sum_{e\in \calE^{\rm ct}_h} \int_{\bp(e)} ({\rm Def}_\gamma \bu):\jump{\pt \bv}
+\int_\gamma ({\bm F}_h^\ell - {\bm f})\cdot \pt \bv + a(\bu-\pt \bu_h,\pt \bv)+\int_{\gamma} ({\rm div}_\gamma \pt \bv)(q^\ell-p)\\
&\qquad\lesssim ((h^{s-1}+h^{2k-1}) \|\bu\|_{H^s(\gamma)} + \|{\bm f}-{\bm F}_h^\ell\|_{L^2(\gamma)}+\|\bu-\pt\bu_h\|_{H^1_h(\gamma)}+\|p-q^\ell\|_{L^2(\gamma)})\|\pt\bv\|_{H^1_h(\gamma)}\\
&\qquad \lesssim ((h^{s-1}+h^{2k-1}) \|\bu\|_{H^s(\gamma)} + \|{\bm f}-{\bm F}_h^\ell\|_{L^2(\gamma)}+h^k \|{\bm F}_h^\ell\|_{X_h^*}+\|p-q^\ell\|_{L^2(\gamma)})\|\pt\bv\|_{H^1_h(\gamma)}.
\end{align*}
Applying this estimate and the geometric estimate \eqref{eqn:GeomEst1} to \eqref{eqn:PressureStart}
yields
\begin{align*}
\int_{\Gamma_h} ({\rm div}_{\Gamma_h} \bv)(q-p_h) 
\lesssim ((h^{s-1}+h^{2k-1}) \|\bu\|_{H^s(\gamma)} + \|{\bm f}-{\bm F}_h^\ell\|_{L^2(\gamma)}+h^k \|{\bm F}_h^\ell\|_{X_h^*}+\|p-q^\ell\|_{L^2(\gamma)})\|\pt\bv\|_{H^1_h(\gamma)}.
\end{align*}
We then apply the inf-sup condition \eqref{equ:MainInfSup} and the approximation
properties of $Q_h$ to conclude \eqref{eqn:PL2error}.
\end{proof}

\begin{remark}
The use of Gauss-Lobatto nodes in the
definition of the finite element space stated
in Theorem \ref{thm:EnergyEst} is the same condition given in \cite{DemlowNeilan25}.
Numerical experiments given in \cite{DemlowNeilan25}
indicate the necessity of using Gauss-Lobatto nodes
to achieve optimal-order convergence, e.g.,
the use of equidistant DOFs generally lead to sub-optimal
convergence for $r\ge 3$.
\end{remark}

\begin{theorem}\label{thm:EnergyEst2}
Under the same assumptions as Theorem \ref{thm:EnergyEst},
there holds
\begin{align*}
\|\bu-\pt\bu_h\|_{L^2(\gamma)}
&\lesssim (h^s  +h^{k+1}+h^{2k-1})\|\bu\|_{H^s(\gamma)} + \|{\bm f}-{\bm F}_h^\ell\|_{X_h^*}+(h^{k+1}+h^{2k-1}) \|{\bm F}_h^\ell\|_{X_h^*}\\
\end{align*}
\end{theorem}
\begin{proof}
The proof follows mostly from the arguments in \cite[Theorem 5.5]{DemlowNeilan25}, so we only
sketch the main points.
We set $(\bvarphi,s)\in \bH^1_T(\gamma)\times {L}_0^2(\gamma)$ satisfy the dual problem
\[
a(\bv,\bvarphi) +b(\bv,s)+b(\bvarphi,q) = \int_\gamma (\bu-\pt \bu_h)\cdot \bv\qquad \forall (\bv,q)\in \bH^1_T(\gamma)\times {L}_0^2(\gamma),
\]
and let $(\bvarphi_h,s_h)\in \bV_h\times Q_h$ denote 
the corresponding finite element solution. We then write
\begin{equation}
\label{eqn:L2PStart}
\begin{split}
    \|\bu-\pt\bu_h\|_{L^2(\gamma)}^2 
    &= a(\bu-\pt \bu_h,\bvarphi-\pt \bvarphi_h)+a(\bu-\pt \bu_h,\pt \bvarphi_h)+[a(\pt \bu_h,\bvarphi) - a_h(\bu_h,\bvarphi_h)]\\
    &\lesssim h \|\bu-\pt \bu_h\|_{H^1_h(\gamma)} \|\bu-\pt \bu_h\|_{L^2(\gamma)}+a(\bu-\pt \bu_h,\pt \bvarphi_h)+[a(\pt \bu_h,\bvarphi) - a_h(\bu_h,\bvarphi_h)].
\end{split}
\end{equation}
We have
\begin{align*}
a(\bu-\pt \bu_h,\bvarphi_h)  
&= \big[a(\bu,\pt \bvarphi_h) - \int_\gamma {\bm f}\cdot \pt \bvarphi_h\big]
+\int_\gamma ({\bm f}-{\bm F}_h^\ell)\cdot \pt \bvarphi_h - G_h(\bu_h,\bvarphi_h)\\
& = \sum_{e\in \calE_h^{\rm ct}} \int_{e^\gamma} {\rm Def}_\gamma \bu:\jump{\bvarphi_h}+\int_\gamma ({\bm f}-{\bm F}_h^\ell)\cdot \pt \bvarphi_h \\
&\qquad- \left(G(\ipt \bu,\ipt \bvarphi) + G(\ipt \bu,\bvarphi_h-\ipt \bvarphi) + G(\bu_h-\ipt \bu,\bvarphi_h)\right)
\end{align*}
where we used that $\pt\bvarphi_h$ is divergence free in the last equality.
Thus, applying \eqref{eqn:KrisKross} (with $m=2$) and the geometric consistency results \eqref{eqn:GeomEst1}--\eqref{eqn:GeomEst2} we have
\begin{equation}
\label{eqn:Cons1}
\begin{split}
a(\bu-\pt \bu_h,\bvarphi_h)  
&\lesssim \left((h^{s}+h^{2k-1}) \|\bu\|_{H^s(\gamma)}
+ \|{\bm f}-{\bm F}_h^\ell\|_{X_h^*}\right) \|\pt \bvarphi_h\|_{H^2_h(\gamma)}\\
&\hspace{-0.5cm} + h^{k+1} \|\bu\|_{H^2(\gamma)} \|\bvarphi\|_{H^2(\gamma)} + h^k \|\bu\|_{H^1(\gamma)} \|\bvarphi-\pt \bvarphi_h\|_{H^1_h(\gamma)} + h^k \|\bu-\pt \bu_h\|_{H^1_h(\gamma)} \|\pt \bvarphi\|_{H^1_h(\gamma)}\\
&\lesssim \left(\big(h^{s} +h^{k+1}+h^{2k-1}\big)\|\bu\|_{H^s(\gamma)} +\|{\bm f}-{\bm F}_h^\ell\|_{X_h^*} +h^k \|\bu-\pt\bu_h\|_{H^1_h(\gamma)}\right) \|\bu-\pt \bu_h\|_{L^2(\gamma)}.
\end{split}
\end{equation}
Next, by the exact same arguments as in \cite[(5.23)]{DemlowNeilan25}, 
and by applying \eqref{eqn:KrisKross} we obtain
\begin{align*}
a(\pt \bu_h,\bvarphi) - a_h(\bu_h,\bvarphi_h) 
&= \sum_{e\in \calE_h^{\rm ct}} \int_{e^\gamma} {\rm Def}_\gamma \bvarphi:\jump{\pt \bu_h}\\
&\lesssim (h^{s}\|\pt \bu_h\|_{H^{s-1}_h(\gamma)}+h^{2k-1}\|\bu_h\|_{L^2(\gamma)})\|\bvarphi\|_{H^2(\gamma)}.
\end{align*}
Standard arguments show $\|\pt \bu_h\|_{H^{s-1}_h(\gamma)}\lesssim \|\bu\|_{H^{s-1}(\gamma)}+h^{2-s}\|\bu-\pt\bu_h\|_{H^1_h(\gamma)}$,
and therefore
\begin{align}\label{eqn:Cons2}
a(\pt \bu_h,\bvarphi) - a_h(\bu_h,\bvarphi_h) 
&\lesssim (h^s\|\bu\|_{H^{s-1}(\gamma)} +h^2 \|\bu-\pt\bu_h\|_{H^1_h(\gamma)} +h^{2k-1}\|{\bm F}_h^\ell\|_{X_h^*})
\|\bu-\pt \bu_h\|_{L^2(\gamma)}.
\end{align}
Applying \eqref{eqn:Cons1} and \eqref{eqn:Cons2} to \eqref{eqn:L2PStart} gets
\begin{align*}
\|\bu-\pt \bu_h\|_{L^2(\gamma)}
&\lesssim h\|\bu-\pt\bu_h\|_{H^1_h(\gamma)}
 + (h^s+h^{k+1}+h^{2k-1})\|\bu\|_{H^s(\gamma)}+\|{\bm f}-{\bm F}_h^\ell\|_{X_h^*}
+h^{2k-1}\|{\bm F}_h^\ell\|_{X_h^*}\\
%
 &\lesssim h\left((h^{s-1}+h^{2k-1}) \|\bu\|_{H^s(\gamma)} + \|{\bm f}-{\bm F}_h^\ell\|_{X_h^*} +h^k \|{\bm F}_h^\ell\|_{X_h^*}\right)\\
 &\qquad + (h^s+h^{k+1}+h^{2k-1})\|\bu\|_{H^s(\Omega)}+\|{\bm f}-{\bm F}_h^\ell\|_{X_h^*}+h^{2k-1}\|{\bm F}_h^\ell\|_{X_h^*}\\
 &\lesssim (h^s  +h^{k+1}+h^{2k-1})\|\bu\|_{H^s(\gamma)} + \|{\bm f}-{\bm F}_h^\ell\|_{X_h^*}+(h^{k+1}+h^{2k-1}) \|{\bm F}_h^\ell\|_{X_h^*}\\
\end{align*}
\end{proof}

\begin{remark}\label{rem:PressureR}[Pressure Robustness]
A well-known property of divergence-free methods
for the Euclidean Stokes problem is their pressure-robustness, i.e.,
the discrete velocity solution is independent of the pressure variable \cite{JohnEtal17}.
However, for the surface Stokes problem, pressure-robustness is not guaranteed
due to geometric and source approximation as we now explain.

Assuming the solution $(\bu,p)$ is sufficiently
smooth, and assuming ${\bm f}_h = {\bm P}_h{\bm f}^e$ for some linear operator ${\bm P}_h$,  have for all $\bv\in \bX_h$,
\begin{equation}\label{eqn:PRRemark}
 \begin{split}
 \int_\gamma ({\bm f}-{\bm F}_h^\ell)\cdot \pt\bv 
 & = \int_\gamma (-(\bPi{\rm div}_\gamma {\rm Def}_\gamma \bu+\bu)\cdot \pt \bv-\int_{\Gamma_h} {\bm f}_h\cdot \bv\\
 & = \int_\gamma (-(\bPi{\rm div}_\gamma {\rm Def}_\gamma \bu+\bu)\cdot \pt \bv+\int_{\Gamma_h} {\bm P}_h
 (\bPi{\rm div}_\gamma {\rm Def}_\gamma \bu+\bu)^e \cdot \bv+\int_{\Gamma_h} {\bm P}_h(\nab_\gamma p)^e \cdot \bv,
 \end{split}
 \end{equation}
where we used the fact $\pt \bv$ is divergence-free (and so, orthogonal to gradients).
Consequently, due to the term $\|{\bm f}-{\bm F}_h^\ell\|_{X_h^*}$ in the error estimate
\eqref{eqn:H1error} we see that pressure robustness of method \eqref{eqn:FEM} is guaranteed if
the operator ${\bm P}_h$ is chosen such that
\[
\int_{\Gamma_h} {\bm P}_h(\nab_\gamma p)^e \cdot \bv=0\qquad \forall \bv\in \bX_h.
\]
This property holds if, e.g., ${\bm P}_h$ maps surface gradients (with respect to $\gamma$) to 
surface gradients (with respect to $\Gamma_h$). 
\end{remark}

\section{Numerical Experiments}\label{sec-numerics}
This section contains three numerical experiments: approximations of a divergence-free function and a no-flow function on the unit sphere, as well as a divergence-free function on a torus. For these examples we take the isoparametric case with $k=r$. We use a MATLAB code built on top of the iFEM library \cite{Ch09PP} as in \cite{DemlowNeilan24} with modifications to employ the Scott-Vogelius pair, notably changing DOFs to make a discontinuous pressure space on the Clough-Tocher split. We run tests with Scott-Vogelius quadratic ($\bbP_2-\bbP_1$), cubic ($\bbP_3-\bbP_2)$, and quartic ($\bbP_4-\bbP_3$) elements, with degrees of freedom at standard Lagrange nodes with Gauss-Lobatto nodes as edge degrees of freedom. Note that this is an implementation of $\bV_h$ and $Q_h$, not the macroelement spaces $\bVmac_h$ and $Q_h^{\rm m}$.


To compute the error, ideally one can map the true solution $\bu$ to the discrete surface $\Gamma_h$ using a Piola transform with respect to the inverse of the closest point projection. However, an explicit formula for the closest point projection is not available in general. 
Assuming $\gamma$ is given as the level set of a smooth function $\bPhi$, i.e.,
$\gamma=\{x\in\R^3:\bPhi(x)=0\}$, we make the following 
approximations (cf.~\cite{DemlowDziuk07}):
\[ 
\tilde{\bnu}(x)=\frac{\nabla \bPhi(x)}{|\nabla \bPhi(x)|},\quad \tilde{d}(x)=\frac{\bPhi(x)}{|\nabla \bPhi(x)|},\quad \tilde{\bp}(x)=x-\tilde{d}(x)\tilde{\bnu}(x).
\]
Then there holds
\[|d(x)-\tilde{d}(x)|=\calO(h^{k+1}), \]
and 
\[|\bnu(x)-\tilde{\bnu}(x)|=\left|\frac{\nabla \bPhi(\bp(x))}{|\nabla \bPhi(\bp(x))|}-\frac{\nabla \bPhi(x)}{|\nabla \bPhi(x)} \right|\les |\bp(x)-x|=\calO(h^{k+1}). \]
Note that $\tilde \bp$ may not lie on the true surface, but
we have $|\bp(x)-\tilde{\bp}(x)|=\calO(h^{k+1})$. 
A similar calculation shows
\begin{equation}\label{bnuTilde}
|\bnu(x)-\tilde{\bnu}(\tilde \bp(x))| = \calO(h^{k+1}).
\end{equation}

In the numerical experiments
below the expression for the velocity $\bu:\gamma \to \mathbb{R}^3$
is well defined in $\mathbb{R}^3$,
and we denote this expression by  $\bu^{ce}:\mathbb{R}^3\to \mathbb{R}^3$.
With the approximations above,
we substitute the true Piola transform $\calP_{p^{-1}}\bu$ with the computable approximation
\[\tilde{\calP}_{\bp^{-1}}\bu:=\left({\bf I}_3-\frac{(\tilde{\bnu}\circ \tilde{\bp})\otimes \bnu_h}{(\tilde{\bnu}\circ \tilde{\bp})\cdot \bnu_h} \right)(\bu^{ce}\circ \tilde{\bp}),
\]
and note that $|\calP_{\bp^{-1}} \bu - \tilde\calP_{\bp^{-1}} \bu| = \calO(h^{k+1}).$

In the set of numerical experiments, we set 
 $\bm f_h(x)=\tilde{\calP}_{\bp^{-1}}\bm f(x)$
 , that is, quadrature points on the discrete surface $\Gamma_h$ are 
mapped to an approximation to $\gamma$ with $\tilde{\bp}$, $\bm f^{ce}$ is evaluated, and the result is mapped back to $\Gamma_h$ with an approximate Piola transform.
\begin{lemma}\label{lem:ComputableErrors}
Let $\tilde{\calP}_{\bp^{-1}}\bu$
and $\tilde{\calP}_{\bp^{-1}}{\bm f}$ be the approximate Piola transforms
of $\bu$ and ${\bm f}$, respectively. Then in the isoparametric case $k=r$, 
there holds for $(\bu,p)\in \bH^{r+1}(\gamma)\times H^r(\gamma)$,
\begin{equation}
\label{eqn:ComputableErrors}
\begin{split}
\|\tilde{\calP}_{\bp^{-1}}\bu- \bu_h\|_{H^1_h(\Gamma_h)}& \lesssim h^r\|\bu\|_{H^{r+1}(\gamma)}+h^{r+1}\|p\|_{H^1(\gamma)},\\
\|p\circ \tilde \bp - p_h\|_{L^2(\Gamma_h)}& \lesssim h^r(\|\bu\|_{H^{r+1}(\gamma)}+\|p\|_{H^r(\gamma)}),\\
\|\tilde{\calP}_{\bp^{-1}}\bu- \bu_h\|_{L^2(\Gamma_h)}&\lesssim h^{r+1}(\|\bu\|_{H^{r+1}(\gamma)}+\|p\|_{H^r(\gamma)}).
\end{split}
\end{equation}
\end{lemma}
\begin{proof}
By a change of variables and Theorems \ref{thm:EnergyEst} and \ref{thm:EnergyEst2},
it suffices to estimate $\|{\bm f}-{\bm F}_h^\ell\|_{X_h^*}$, $\|{\bm F}_h^\ell\|_{X_h^*}$,
and $\|{\bm f}-{\bm F}_h^\ell\|_{L^2(\gamma)}$.

Using $d = \calO(h^{k+1})$ and $|1-\bnu_h\cdot \bnu| = \calO(h^{2k}) = \calO(h^{k+1})$, 
 $|1-(\tilde \bnu\circ \tilde \bp)\cdot \bnu_h| = \calO(h^{k+1})$, and 
 $|1-(\tilde \bnu\circ \tilde \bp)\cdot \bnu| = \calO(h^{k+1})$
we have (cf.~\eqref{eqn:FhDeff})
\begin{align*}
{\bm F}_h 
&= \big[{\bf I}_3 - {\bnu_h\otimes \bnu}\big] {\bm f}_h + \calO(h^{k+1} \|{\bm f}_h\|_{L^2(\Gamma_h)})\\
&= \big[{\bf I}_3 - {\bnu_h\otimes \bnu}\big]
\big[{\bf I}_3
- {(\tilde \bnu\circ \tilde \bp) \otimes \bnu_h}\big]({\bm f}^{ce}\circ \tilde \bp)+\calO(h^{k+1}\|{\bm f}\|_{L^2(\gamma)})\\
& = \big[{\bf I}_3 - \bnu_h\otimes \bnu- (\tilde \bnu\circ \tilde \bp)\otimes \bnu_h+\bnu_h\otimes \bnu_h\big]({\bm f}^{ce}\circ \tilde \bp)+\calO(h^{k+1}\|{\bm f}\|_{L^2(\gamma)} )\\
& = \big[{\bf I}_3 - \bnu_h\otimes (\bnu-\bnu_h)- (\tilde \bnu\circ \tilde \bp)\otimes (\bnu_h-\bnu)\big]({\bm f}^{ce}\circ \tilde \bp)+\calO(h^{k+1}\|{\bm f}\|_{L^2(\gamma)}),
\end{align*}
where we used $|\bnu\cdot {\bm f}^{ce}\circ \tilde \bp| = |\bnu\cdot ({\bm f}^{ce}\circ \tilde \bp - {\bm f}\circ \bp)| = \calO(h^{k+1})$
in the last equality. It then follows from \eqref{bnuTilde} 
that $\|{\bm f}^{ce}\circ \tilde \bp - {\bm F}_h\|_{L^2(\Gamma_h)} \lesssim h^{k+1} \|{\bm f}\|_{L^2(\gamma)}$, and therefore
\begin{equation}\label{eqn:fFHL2}
    \|{\bm f}-{\bm F}_h^\ell\|_{L^2(\gamma)}\lesssim h^{k+1}\|{\bm f}\|_{L^2(\gamma)}\lesssim h^{k+1}(\|\bu\|_{H^2(\gamma)}+\|p\|_{H^1(\gamma)}).
\end{equation}

We also find that, using a slight variant of \eqref{eqn:PRRemark},
\begin{align*}
    \|{\bm F}_h^\ell\|_{X_h^*}
    &\le \|{\bm f}\|_{X_h^*}+ \|{\bm f}-{\bm F}_h^\ell\|_{X_h^*}\lesssim  \|{\bm f}\|_{X_h^*}+h^{k+1} \|\bu\|_{H^2(\gamma)}+ 
    \sup_{\bv\in \bX_h} \frac{\int_{\Gamma_h} (\tilde \calP_{\bp^{-1}} \nab_\gamma p)\cdot \bv}{\|\bv\|_{H^1_h(\Gamma_h)}}
\end{align*}
A simple calculation shows that, for divergence-free $\bv\in \bX_h$,
there holds 
\begin{align*}
    \left|\int_{\Gamma_h} (\tilde \calP_{\bp^{-1}} \nab_\gamma p)\cdot \bv\right| = 
\left|\int_{\Gamma_h} (\tilde \calP_{\bp^{-1}} \nab_\gamma p- \nab_{\Gamma_h} p^e)\cdot \bv\right|\lesssim h^{k+1}\|p\|_{H^1(\gamma)}\|\bv\|_{L^2(\Gamma_h)},
\end{align*}
and therefore
\begin{align}\label{eqn:fFhDual}
\|{\bm f}-{\bm F}_h^\ell\|_{X_h^*}+
\|{\bm F}_h^\ell\|_{X_h^*}
\lesssim  \|{\bm f}\|_{X_h^*}+h^{k+1}(\|\bu\|_{H^2(\gamma)}+ \|p\|_{H^{1}(\gamma)}).
\end{align}
Finally, we apply the estimates \eqref{eqn:fFHL2}--\eqref{eqn:fFhDual}
to Theorems  \ref{thm:EnergyEst} and \ref{thm:EnergyEst2} 
with $k=r$ and use $r\ge 2$ to obtain 
\eqref{eqn:ComputableErrors}.
\end{proof}

%


%

\subsection*{Test 1}
For the first example, we take $\gamma$
to be the unit sphere, and
\[\bu(x) =\begin{pmatrix}-x_1x_3e^{x_2}+4x_2^2x_3-2x_3^3\\ x_3(-4x_1x_2+e^{x_2})\\(x_1^2-x_2)e^{x_2}+2x_1x_3^2 \end{pmatrix},\qquad p(x)=x_1x_2^3+x_3,\]
which satisfy the divergence-free and mean-zero constraint, respectively. In Table \ref{tab:convergence} we provide the errors of the velocity, velocity gradient, pressure, and divergence
for polynomial degrees $r\in \{2,3,4\}$. We observe optimal convergence of $\calO(h^{r+1})$ for the velocity and $\calO(h^r)$ for the velocity gradient and pressure in the isoparametric case $k=r$, aligning with the theoretical results in Lemma \ref{lem:ComputableErrors}.
Note also that the approximate solution is  divergence-free up to machine precision.

\subsection*{Test 2}
In the second example, we consider a no-flow solution ($\bu=0$) on the unit sphere with $p(x)=x_1e^{3x_2}-4x_3e^{x_2}$, which has mean zero. In Table \ref{tab:convergence_noflow} we provide the errors of the velocity, velocity gradient, pressure, and divergence for a no-flow example on the unit sphere.  Although the true velocity solution was identically zero, the approximation is not, nor is it near machine epsilon. This can be explained by the lack of pressure robustness due to the geometric error of the scheme, as explained in Remark \ref{rem:PressureR}.

In this experiment, we observe  optimal convergence 
$\calO(h^r)$ for the pressure approximation.
 In addition, we also observe 
 for the velocity gradient approximation
$\|\nab_{\Gamma_h} (\tilde\calP_{\bp^{-1}} \bu-\bu_h)\|_{\Gamma_h} = \calO(h^{r+1})$.
Both of these rates are consistent with the estimates given in Lemma \ref{lem:ComputableErrors}.
However, we also observe superconvergence for the $L^2$ error 
of the velocity with 
$\|\tilde \calP_{\bp^{-1}} \bu-\bu_h\|_{\Gamma_h} = \calO(h^{r+2})$,
which is one order higher than the estimate given in \eqref{eqn:ComputableErrors}.
This experiment indicates that the third estimate in Lemma \ref{lem:ComputableErrors}
may not be sharp.




\subsection*{Test 3}
For the third example, we generate a divergence-free function $\bu$ by taking the surface curl with respect to the torus $\gamma=\{x\in \bbR^3: (\sqrt{x_1^2+x_2^2}-4)^2+x_3^2=4 \}$ of the stream function $\phi(x)=\sin(x_1)+e^{x_2}$ and take $p(x)=x_1$.

In Table \ref{tab:convergence_torus}, we provide the relative errors of the velocity, velocity gradient, and pressure for a divergence-free function on the torus. Relative errors are used to account for the large magnitude of the original function. The divergence error is not relative, as the true divergence is identically zero. As in the previous two experiments the convergence rates appear optimal with $\calO(h^{r+1})$ for velocity and $\calO(h^r)$ for the velocity gradient and pressure, although there is significantly more degradation in the error. This may be due to round-off error from a larger system, as well as poor conditioning of the system. The approximate velocity and errors are graphed in Figure \ref{fig:graphs} in the case of three mesh refinements and $\bbP_2-\bbP_1$ spaces.

 {\small 
 \begin{table}[h!]
\centering
\caption{Test 1. Errors and rates of convergence on the unit sphere for the non-trivial divergence-free solution.
Here, $\|\cdot\|_{\Gamma_h}$ denotes the $L^2$ norm over $\Gamma_h$,
 $\tilde \bu = \tilde{\calP}_{\bp^{-1}}\bu$, and $\tilde p = p\circ \tilde \bp$.
}
\label{tab:convergence}
\begin{tabular}{ccllllllll}
\hline
$r$ & $h$ 
  & $\|\tilde\bu-\bu_h\|_{\Gamma_h}$ & rate 
  & $\|\nabla_{\Gamma_h} (\tilde\bu-\bu_h)\|_{\Gamma_h}$ & rate 
  & $\|\tilde p-p_h\|_{\Gamma_h}$ & rate 
  & $\| {\rm div }_{\Gamma_h}  \bu_h\|_{\Gamma_h}$ \\
\hline
\multirow{4}{*}{2}
 & 4.45e-01 & 2.2337e-01 & --     & 1.7845e+00 & --     & 1.8394e+00 & --     & 1.39e-13 \\
 & 2.05e-01 & 2.7751e-02 & 2.69 & 4.9670e-01 & 1.65 & 6.7901e-01 & 1.28 & 4.99e-14 \\
 & 1.03e-01 & 4.0127e-03 & 2.81 & 1.3420e-01 & 1.90 & 2.4056e-01 & 1.51 & 9.96e-14 \\
 & 5.11e-02 & 5.3511e-04 & 2.87 & 3.4645e-02 & 1.93 & 6.8550e-02 & 1.79 & 1.61e-13 \\
\hline
\multirow{4}{*}{3}
 & 4.57e-01 & 2.7469e-01 & --     & 2.3703e+00 & --     & 3.1448e+00 & --     & 1.16e-11 \\
 & 2.06e-01 & 1.0573e-02 & 4.09 & 2.6498e-01 & 2.75 & 3.6753e-01 & 2.69 & 8.01e-14 \\
 & 1.03e-01 & 5.8573e-04 & 4.18 & 3.2890e-02 & 3.02 & 4.7743e-02 & 2.95 & 1.53e-13 \\
 & 5.11e-02 & 3.5077e-05 & 4.01 & 4.0988e-03 & 2.97 & 6.0540e-03 & 2.94 & 2.80e-13 \\
\hline
\multirow{4}{*}{4}
 & 4.57e-01 & 2.8892e-02 & --     & 3.5492e-01 & --     & 5.1232e-01 & --     & 3.68e-13 \\
 & 2.06e-01 & 1.1148e-03 & 4.08 & 2.5515e-02 & 3.30 & 3.6641e-02 & 3.31 & 1.39e-13 \\
 & 1.03e-01 & 3.8319e-05 & 4.87 & 1.6922e-03 & 3.92 & 2.6104e-03 & 3.82 & 2.27e-13 \\
 & 5.11e-02 & 1.2368e-06 & 4.89 & 1.0795e-04 & 3.92 & 1.7095e-04 & 3.88 & 4.36e-13 \\
\hline
\end{tabular}
\end{table}
}
{\small 
\begin{table}[h!]
\centering
\caption{Test 2. Errors and rates of convergence on the unit sphere for the no-flow solution.}
\label{tab:convergence_noflow}
\begin{tabular}{ccllllllll}
\hline
$k$ & $h$ 
  & $\|\tilde\bu-\bu_h\|_{\Gamma_h}$ & rate 
  & $\|\nabla_{\Gamma_h} (\tilde{\bu}-\bu_h)\|_{\Gamma_h}$ & rate 
  & $\|\tilde p-p_h\|_{\Gamma_h}$ & rate 
  & $\| {\rm div }_{\Gamma_h}  \bu_h\|_{\Gamma_h}$ \\
\hline
\multirow{4}{*}{2}
 & 4.45e-01 & 1.1105e-03 & --     & 1.8851e-03 & --     & 7.4351e-01 & --     & 5.67e-18 \\
 & 2.05e-01 & 8.0977e-05 & 3.37 & 1.2856e-04 & 3.46 & 2.0874e-01 & 1.64 & 7.54e-19 \\
 & 1.03e-01 & 5.6581e-06 & 3.87 & 8.9333e-06 & 3.88 & 5.3350e-02 & 1.98 & 9.25e-20 \\
 & 5.11e-02 & 3.6368e-07 & 3.91 & 5.7375e-07 & 3.91 & 1.3416e-02 & 1.97 & 1.16e-20 \\
\hline
\multirow{4}{*}{3}
 & 4.57e-01 & 3.8649e-03 & --     & 6.8489e-03 & --     & 2.1878e-01 & --     & 2.56e-17 \\
 & 2.06e-01 & 7.3649e-05 & 4.97 & 1.3492e-04 & 4.93 & 2.7103e-02 & 2.62 & 2.04e-18 \\
 & 1.03e-01 & 4.5563e-06 & 4.02 & 7.3539e-06 & 4.20 & 3.4391e-03 & 2.98 & 1.22e-19 \\
 & 5.11e-02 & 2.8364e-07 & 3.95 & 4.4944e-07 & 3.98 & 4.3062e-04 & 2.96 & 1.49e-20 \\
\hline
\multirow{4}{*}{4}
 & 4.57e-01 & 8.3306e-04 & --     & 7.7039e-03 & --     & 2.5277e-02 & --     & 7.95e-18 \\
 & 2.06e-01 & 7.6050e-06 & 5.89 & 2.2547e-04 & 4.43 & 7.8496e-04 & 4.36 & 1.71e-19 \\
 & 1.03e-01 & 1.1710e-07 & 6.03 & 7.5069e-06 & 4.92 & 2.8031e-05 & 4.82 & 6.00e-21 \\
 & 5.11e-02 & 1.8171e-09 & 5.93 & 2.3808e-07 & 4.91 & 1.1416e-06 & 4.56 & 1.82e-22 \\
\hline
\end{tabular}
\end{table}
}
{\small 
\begin{table}[h!]
\centering
\caption{Test 3. Relative errors and rates of convergence on the torus for a non-trivial divergence-free solution.}
\label{tab:convergence_torus}
\begin{tabular}{ccllllllll}
\hline
$k$ & $h$ 
  & $\|\tilde\bu-\bu_h\|_{\Gamma_h}$ & rate 
  & $\|\nabla (\tilde \bu-\bu_h)\|_{\Gamma_h}$ & rate 
  & $\|\tilde p-p_h\|_{\Gamma_h}$ & rate 
  & $\|{\rm div}_{\Gamma_h}\  \bu_h\|_{\Gamma_h}$ \\
\hline
\multirow{4}{*}{2}
 & 1.69 & 1.5562e-01 & --     & 4.2444e-01 & --     & 6.9905e+00 & --     & 1.31e-11 \\
 & 7.91e-01 & 2.2517e-02 & 2.54 & 1.2114e-01 & 1.65 & 2.1854e+00 & 1.53 & 7.41e-12 \\
 & 3.82e-01 & 2.9603e-03 & 2.78 & 3.5016e-02 & 1.70 & 7.0296e-01 & 1.55 & 6.90e-12 \\
 & 1.89e-01 & 4.3856e-04 & 2.71 & 1.0075e-02 & 1.77 & 2.4189e-01 & 1.52 & 1.44e-11 \\
\hline
\multirow{4}{*}{3}
 & 1.78 & 1.9145e-01 & --     & 5.5096e-01 & --     & 8.2243e+00 & --     & 1.38e-10 \\
 & 8.03e-01 & 1.8453e-02 & 2.94 & 1.2460e-01 & 1.87 & 1.4255e+00 & 2.21 & 1.07e-08 \\
 & 3.83e-01 & 9.0754e-04 & 4.07 & 1.4272e-02 & 2.93 & 2.0147e-01 & 2.64 & 6.80e-09 \\
 & 1.89e-01 & 5.0656e-05 & 4.09 & 1.7516e-03 & 2.97 & 2.6396e-02 & 2.88 & 4.46e-08 \\
\hline
\multirow{4}{*}{4}
 & 1.71 & 1.0742e-01 & --     & 3.4331e-01 & --     & 5.6570e+00 & --     & 4.65e-10 \\
 & 8.00e-01 & 2.0565e-03 & 5.23 & 1.5965e-02 & 4.05 & 2.5649e-01 & 4.09 & 3.64e-09 \\
 & 3.83e-01 & 8.9782e-05 & 4.24 & 1.2590e-03 & 3.44 & 1.7368e-02 & 3.65 & 5.48e-09 \\
 & 1.89e-01 & 3.1836e-06 & 4.73 & 8.4639e-05 & 3.82 & 1.2295e-03 & 3.75 & 2.88e-08 \\
\hline
\end{tabular}
\end{table}
}

\begin{figure}[h]

  \centering
  \begin{subfigure}{0.45\textwidth}
    \includegraphics[width=\linewidth]{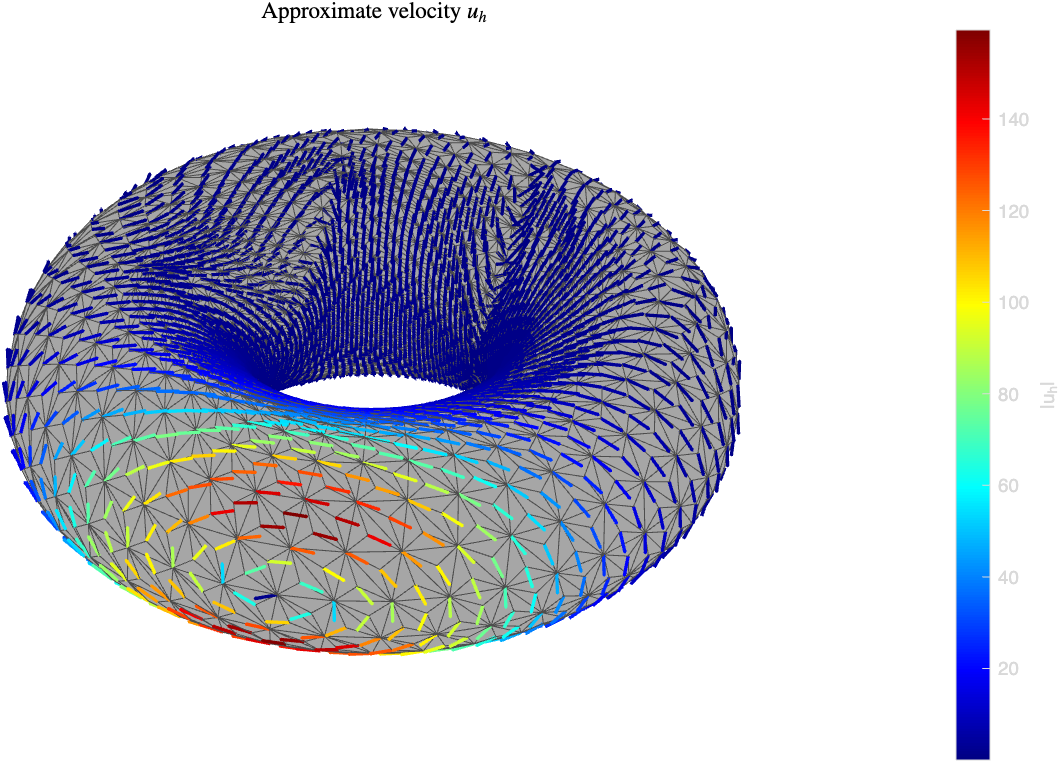}
    \caption{Approximate solution $\bu_h$ on torus
    with $r=2$ and $h=0.189$.}
  \end{subfigure}
  \hfill
  \begin{subfigure}{0.45\textwidth}
    \includegraphics[width=\linewidth]{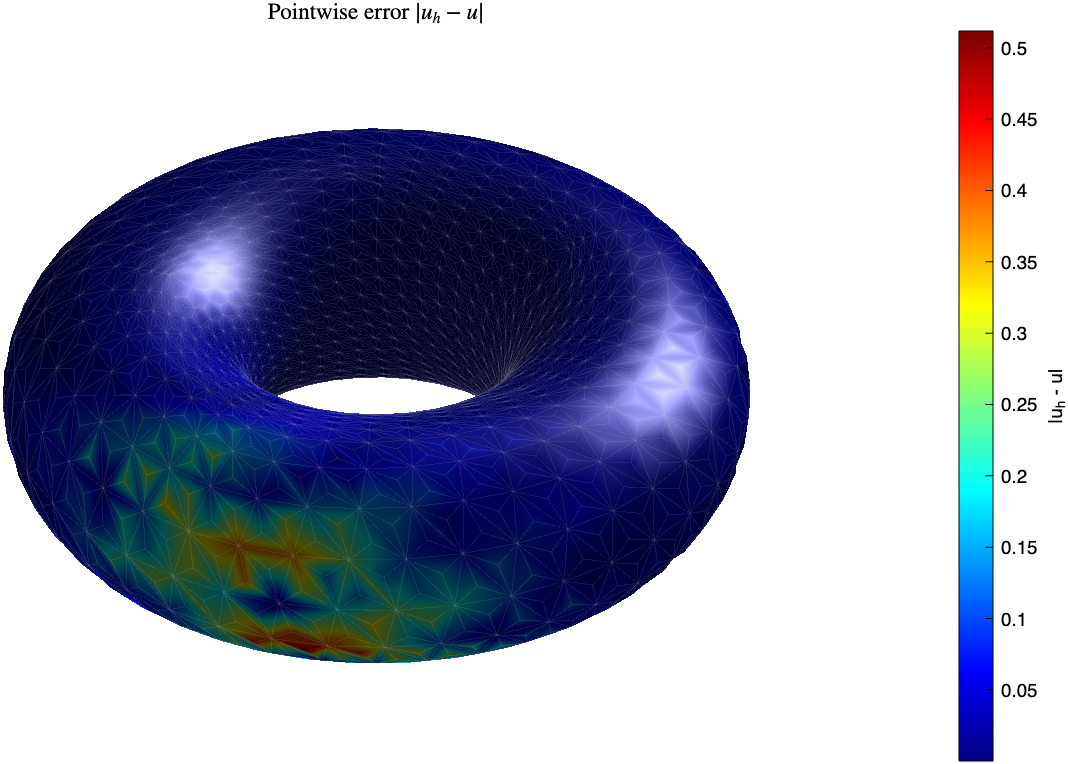}
    \caption{Pointwise errors $|\tilde \bu-\bu_h|$ on torus
    with $r=2$ and $h=0.189$.}
  \end{subfigure}
  \caption{Approximation of a divergence-free function on the torus}
  \label{fig:graphs}
\end{figure}

\bibliographystyle{siam}
\bibliography{literatur}

@article {BrubeckKirby26,
    AUTHOR = {Brubeck, Pablo D. and Kirby, Robert C.},
     TITLE = {F{IAT}: enabling classical and modern macroelements},
   JOURNAL = {ACM Trans. Math. Software},
  FJOURNAL = {Association for Computing Machinery. Transactions on
              Mathematical Software},
    VOLUME = {52},
      YEAR = {2026},
    NUMBER = {1},
     PAGES = {Art. 6, 28},
      ISSN = {0098-3500,1557-7295},
   MRCLASS = {65N30 (65M60 65Y15)},
  MRNUMBER = {5061631},
       DOI = {10.1145/3797879},
       URL = {https://doi.org/10.1145/3797879},
}

@article {GuzmanScott19,
    AUTHOR = {Guzm\'{a}n, Johnny and Scott, L. Ridgway},
     TITLE = {The {S}cott-{V}ogelius finite elements revisited},
   JOURNAL = {Math. Comp.},
  FJOURNAL = {Mathematics of Computation},
    VOLUME = {88},
      YEAR = {2019},
    NUMBER = {316},
     PAGES = {515--529},
      ISSN = {0025-5718,1088-6842},
   MRCLASS = {65N30 (65N12 65N85 76D07)},
  MRNUMBER = {3882274},
MRREVIEWER = {Huipo\ Liu},
       DOI = {10.1090/mcom/3346},
       URL = {https://doi.org/10.1090/mcom/3346},
}

@book {QinThesis,
    AUTHOR = {Qin, Jinshui},
     TITLE = {On the convergence of some low order mixed finite elements for
              incompressible fluids},
      NOTE = {Thesis (Ph.D.)--The Pennsylvania State University},
 PUBLISHER = {ProQuest LLC, Ann Arbor, MI},
      YEAR = {1994},
     PAGES = {158},
   MRCLASS = {Thesis},
  MRNUMBER = {2691498},
       URL =
              {http://gateway.proquest.com/openurl?url_ver=Z39.88-2004&rft_val_fmt=info:ofi/fmt:kev:mtx:dissertation&res_dat=xri:pqdiss&rft_dat=xri:pqdiss:9504277},
}

@article {SZhang05,
    AUTHOR = {Zhang, Shangyou},
     TITLE = {A new family of stable mixed finite elements for the 3{D}
              {S}tokes equations},
   JOURNAL = {Math. Comp.},
  FJOURNAL = {Mathematics of Computation},
    VOLUME = {74},
      YEAR = {2005},
    NUMBER = {250},
     PAGES = {543--554},
      ISSN = {0025-5718,1088-6842},
   MRCLASS = {65N30 (65F10 65N12 76D07 76M10)},
  MRNUMBER = {2114637},
MRREVIEWER = {J.\ W.\ Jerome},
       DOI = {10.1090/S0025-5718-04-01711-9},
       URL = {https://doi.org/10.1090/S0025-5718-04-01711-9},
}

@article {ChristHu18,
    AUTHOR = {Christiansen, Snorre H. and Hu, Kaibo},
     TITLE = {Generalized finite element systems for smooth differential
              forms and {S}tokes' problem},
   JOURNAL = {Numer. Math.},
  FJOURNAL = {Numerische Mathematik},
    VOLUME = {140},
      YEAR = {2018},
    NUMBER = {2},
     PAGES = {327--371},
      ISSN = {0029-599X,0945-3245},
   MRCLASS = {65N30 (58A12 76D07)},
  MRNUMBER = {3851060},
MRREVIEWER = {J\'{a}nos\ Kar\'{a}tson},
       DOI = {10.1007/s00211-018-0970-6},
       URL = {https://doi.org/10.1007/s00211-018-0970-6},
}

@article {SV85,
    AUTHOR = {Scott, L. R. and Vogelius, M.},
     TITLE = {Norm estimates for a maximal right inverse of the divergence
              operator in spaces of piecewise polynomials},
   JOURNAL = {RAIRO Mod\'{e}l. Math. Anal. Num\'{e}r.},
  FJOURNAL = {RAIRO Mod\'{e}lisation Math\'{e}matique et Analyse
              Num\'{e}rique},
    VOLUME = {19},
      YEAR = {1985},
    NUMBER = {1},
     PAGES = {111--143},
      ISSN = {0764-583X,1290-3841},
   MRCLASS = {65N30},
  MRNUMBER = {813691},
       DOI = {10.1051/m2an/1985190101111},
       URL = {https://doi.org/10.1051/m2an/1985190101111},
}

@article {AlfeldSorokina16,
    AUTHOR = {Alfeld, Peter and Sorokina, Tatyana},
     TITLE = {Linear differential operators on bivariate spline spaces and
              spline vector fields},
   JOURNAL = {BIT},
  FJOURNAL = {BIT. Numerical Mathematics},
    VOLUME = {56},
      YEAR = {2016},
    NUMBER = {1},
     PAGES = {15--32},
      ISSN = {0006-3835,1572-9125},
   MRCLASS = {41A15 (41A10 65D05 65D07)},
  MRNUMBER = {3486451},
MRREVIEWER = {Tian-Xiao\ He},
       DOI = {10.1007/s10543-015-0557-x},
       URL = {https://doi.org/10.1007/s10543-015-0557-x},
}

@article {FalkNeilan13,
    AUTHOR = {Falk, Richard S. and Neilan, Michael},
     TITLE = {Stokes complexes and the construction of stable finite
              elements with pointwise mass conservation},
   JOURNAL = {SIAM J. Numer. Anal.},
  FJOURNAL = {SIAM Journal on Numerical Analysis},
    VOLUME = {51},
      YEAR = {2013},
    NUMBER = {2},
     PAGES = {1308--1326},
      ISSN = {0036-1429,1095-7170},
   MRCLASS = {65N30 (65N12 76D07 76M10)},
  MRNUMBER = {3045658},
MRREVIEWER = {Carlos\ V\'{a}zquez Cend\'{o}n},
       DOI = {10.1137/120888132},
       URL = {https://doi.org/10.1137/120888132},
}

@article {GuzmanNeilan14,
    AUTHOR = {Guzm\'{a}n, Johnny and Neilan, Michael},
     TITLE = {Conforming and divergence-free {S}tokes elements on general
              triangular meshes},
   JOURNAL = {Math. Comp.},
  FJOURNAL = {Mathematics of Computation},
    VOLUME = {83},
      YEAR = {2014},
    NUMBER = {285},
     PAGES = {15--36},
      ISSN = {0025-5718,1088-6842},
   MRCLASS = {65N30 (65N12 76D07 76M10)},
  MRNUMBER = {3120580},
MRREVIEWER = {Marius\ Ghergu},
       DOI = {10.1090/S0025-5718-2013-02753-6},
       URL = {https://doi.org/10.1090/S0025-5718-2013-02753-6},
}

@phdthesis{Kilicer2025,
  author       = {Orsan Kilicer},
  title        = {Higher Order FEM for Surface Stokes Problems},
  school       = {Texas A\&M University},
  year         = {2025},
  month        = may,
  type         = {PhD thesis},
  advisor      = {Alan Demlow and Andrea Bonito},
  url          = {https://oaktrust.library.tamu.edu/bitstreams/68fe8214-3d59-42d4-b185-85f972db09cc/download}
}

@article {ReutherVoigt15,
    AUTHOR = {Reuther, S. and Voigt, A.},
     TITLE = {The interplay of curvature and vortices in flow on curved
              surfaces},
   JOURNAL = {Multiscale Model. Simul.},
  FJOURNAL = {Multiscale Modeling \& Simulation. A SIAM Interdisciplinary
              Journal},
    VOLUME = {13},
      YEAR = {2015},
    NUMBER = {2},
     PAGES = {632--643},
      ISSN = {1540-3459},
   MRCLASS = {76A20 (14Q10 35Q35 53Z05 76D17)},
  MRNUMBER = {3359666},
MRREVIEWER = {Akhtar A. Khan},
       DOI = {10.1137/140971798},
       URL = {https://doi.org/10.1137/140971798},
}

@book{EdwardsEtal91,
	AUTHOR = {Edwards, D.A and Brenner, H. and Wasan, T.},
	TITLE = {Interfacial Transport Processes and Rheology},
	PUBLISHER = {Elsevier},
	YEAR = {1991}
}

@unpublished{BruersEtal26,
      title={Releasing the pressure: High-order surface flow discretizations via discrete Helmholtz-Hodge decompositions}, 
      author={Tim Brüers and Christoph Lehrenfeld and Tim van Beeck and Max Wardetzky},
      year={2026},
      eprint={2603.27714},
      archivePrefix={arXiv},
      primaryClass={math.NA},
      url={https://arxiv.org/abs/2603.27714}, 
}

@article {Lenoir86,
    AUTHOR = {Lenoir, M.},
     TITLE = {Optimal isoparametric finite elements and error estimates for
              domains involving curved boundaries},
   JOURNAL = {SIAM J. Numer. Anal.},
  FJOURNAL = {SIAM Journal on Numerical Analysis},
    VOLUME = {23},
      YEAR = {1986},
    NUMBER = {3},
     PAGES = {562--580},
      ISSN = {0036-1429},
   MRCLASS = {65N15 (65N30)},
  MRNUMBER = {842644},
MRREVIEWER = {Stephen W. Brady},
       DOI = {10.1137/0723036},
       URL = {https://doi.org/10.1137/0723036},
}

@article {DemlowNeilan25,
    AUTHOR = {Demlow, Alan and Neilan, Michael},
     TITLE = {
A {T}aylor-{H}ood finite element method for the surface {S}tokes problem without penalization},
   JOURNAL = {https://arxiv.org/abs/2506.20419},
      YEAR = {2025},
}

@article {DemlowNeilan24,
    AUTHOR = {Demlow, Alan and Neilan, Michael},
     TITLE = {A tangential and penalty-free finite element method for the
              surface {S}tokes problem},
   JOURNAL = {SIAM J. Numer. Anal.},
  FJOURNAL = {SIAM Journal on Numerical Analysis},
    VOLUME = {62},
      YEAR = {2024},
    NUMBER = {1},
     PAGES = {248--272},
      ISSN = {0036-1429},
   MRCLASS = {65N15 (65N12 65N30)},
  MRNUMBER = {4695763},
MRREVIEWER = {Marius Ghergu},
       DOI = {10.1137/23M1583995},
       URL = {https://doi.org/10.1137/23M1583995},
}

@article {GuzmanNeilan18,
    AUTHOR = {Guzm\'{a}n, Johnny and Neilan, Michael},
     TITLE = {inf-sup stable finite elements on barycentric refinements
              producing divergence-free approximations in arbitrary
              dimensions},
   JOURNAL = {SIAM J. Numer. Anal.},
  FJOURNAL = {SIAM Journal on Numerical Analysis},
    VOLUME = {56},
      YEAR = {2018},
    NUMBER = {5},
     PAGES = {2826--2844},
      ISSN = {0036-1429},
   MRCLASS = {65N30 (65N12)},
  MRNUMBER = {3853609},
MRREVIEWER = {Marius Ghergu},
       DOI = {10.1137/17M1153467},
       URL = {https://doi.org/10.1137/17M1153467},
}

@article {NeilanOtus21,
    AUTHOR = {Neilan, Michael and Otus, Baris},
     TITLE = {Divergence-free {S}cott-{V}ogelius elements on curved domains},
   JOURNAL = {SIAM J. Numer. Anal.},
  FJOURNAL = {SIAM Journal on Numerical Analysis},
    VOLUME = {59},
      YEAR = {2021},
    NUMBER = {2},
     PAGES = {1090--1116},
      ISSN = {0036-1429},
   MRCLASS = {65N30 (35Q30 65N12 76D07 76M10)},
  MRNUMBER = {4249063},
MRREVIEWER = {Riccardo Sacco},
       DOI = {10.1137/20M1360098},
       URL = {https://doi.org/10.1137/20M1360098},
}

@unpublished{DurstNeilan24,
    AUTHOR = {Durst, Rebecca and Neilan, Michael},
     TITLE = {General degree divergence-free finite element methods for the {S}tokes problem on smooth domains},
   Note = {https://arxiv.org/pdf/2404.14226},
      YEAR = {2024},
      }

@article {HP23,
    AUTHOR = {Hardering, Hanne and Praetorius, Simon},
     TITLE = {Tangential errors of tensor surface finite elements},
   JOURNAL = {IMA J. Numer. Anal.},
  FJOURNAL = {IMA Journal of Numerical Analysis},
    VOLUME = {43},
      YEAR = {2023},
    NUMBER = {3},
     PAGES = {1543--1585},
      ISSN = {0272-4979},
   MRCLASS = {65M12 (65N30)},
  MRNUMBER = {4621834},
       DOI = {10.1093/imanum/drac015},
       URL = {https://doi-org.pitt.idm.oclc.org/10.1093/imanum/drac015},
}

@article{NVW12,
	Author = {Nitschke, I. and Voigt, A. and Wensch, J.},
	Doi = {10.1017/jfm.2012.317},
	Fjournal = {Journal of Fluid Mechanics},
	Issn = {0022-1120,1469-7645},
	Journal = {J. Fluid Mech.},
	Mrclass = {76D05 (76M10 76T99)},
	Mrnumber = {2975450},
	Pages = {418--438},
	Title = {A finite element approach to incompressible two-phase flow on manifolds},
	Url = {https://doi.org/10.1017/jfm.2012.317},
	Volume = {708},
	Year = {2012}}

@article{OQRY18,
	Author = {Olshanskii, Maxim A. and Quaini, Annalisa and Reusken, Arnold and Yushutin, Vladimir},
	Doi = {10.1137/18M1166183},
	Fjournal = {SIAM Journal on Scientific Computing},
	Issn = {1064-8275,1095-7197},
	Journal = {SIAM J. Sci. Comput.},
	Mrclass = {65N30 (65N12 76D07)},
	Mrnumber = {3841618},
	Mrreviewer = {J\'{a}nos\ Kar\'{a}tson},
	Number = {4},
	Pages = {A2492--A2518},
	Title = {A finite element method for the surface {S}tokes problem},
	Url = {https://doi.org/10.1137/18M1166183},
	Volume = {40},
	Year = {2018}}

@techreport{Ch09PP,
	Author = {Long Chen},
	Institution = {University of California-Irvine},
	Title = {{iFEM}: An innovative finite element method package in {M}atlab},
	Year = {2009}}

@article{Maxim19,
	Author = {Olshanskii, Maxim A. and Yushutin, Vladimir},
	Doi = {10.1007/s00021-019-0420-y},
	Fjournal = {Journal of Mathematical Fluid Mechanics},
	Issn = {1422-6928},
	Journal = {J. Math. Fluid Mech.},
	Mrclass = {65M60 (65M12 65M15 76D05)},
	Mrnumber = {3911729},
	Number = {1},
	Pages = {Paper No. 14, 18},
	Title = {A penalty finite element method for a fluid system posed on embedded surface},
	Url = {https://doi.org/10.1007/s00021-019-0420-y},
	Volume = {21},
	Year = {2019}}

@article{HansboLarsonLarsson20,
	Author = {Hansbo, Peter and Larson, Mats G. and Larsson, Karl},
	Doi = {10.1093/imanum/drz018},
	Fjournal = {IMA Journal of Numerical Analysis},
	Issn = {0272-4979},
	Journal = {IMA J. Numer. Anal.},
	Mrclass = {65N30 (53Z30 65N15)},
	Mrnumber = {4122487},
	Mrreviewer = {Javier A. Almonacid},
	Number = {3},
	Pages = {1652--1701},
	Title = {Analysis of finite element methods for vector {L}aplacians on surfaces},
	Url = {https://doi.org/10.1093/imanum/drz018},
	Volume = {40},
	Year = {2020}}

@article{Fries18,
	Author = {Fries, Thomas-Peter},
	Doi = {10.1002/fld.4510},
	Fjournal = {International Journal for Numerical Methods in Fluids},
	Issn = {0271-2091},
	Journal = {Internat. J. Numer. Methods Fluids},
	Mrclass = {65N30 (35Q30 58J32 65M60 76D05 76M10)},
	Mrnumber = {3846120},
	Number = {2},
	Pages = {55--78},
	Title = {Higher-order surface {FEM} for incompressible {N}avier-{S}tokes flows on manifolds},
	Url = {https://doi.org/10.1002/fld.4510},
	Volume = {88},
	Year = {2018}}

@article{JohnEtal17,
	Author = {John, Volker and Linke, Alexander and Merdon, Christian and Neilan, Michael and Rebholz, Leo G.},
	Doi = {10.1137/15M1047696},
	Fjournal = {SIAM Review},
	Issn = {0036-1445},
	Journal = {SIAM Rev.},
	Mrclass = {65N30 (76M10)},
	Mrnumber = {3683678},
	Mrreviewer = {Stephan Schmidt},
	Number = {3},
	Pages = {492--544},
	Title = {On the divergence constraint in mixed finite element methods for incompressible flows},
	Url = {https://doi.org/10.1137/15M1047696},
	Volume = {59},
	Year = {2017}}

@article{SurfaceStokes2,
	Author = {Lederer, Philip L. and Lehrenfeld, Christoph and Sch\"{o}berl, Joachim},
	Doi = {10.1002/nme.6317},
	Fjournal = {International Journal for Numerical Methods in Engineering},
	Issn = {0029-5981},
	Journal = {Internat. J. Numer. Methods Engrg.},
	Mrclass = {65N30 (35Q35 53E99 76Bxx 76Dxx)},
	Mrnumber = {4156753},
	Number = {11},
	Pages = {2503--2533},
	Title = {Divergence-free tangential finite element methods for incompressible flows on surfaces},
	Url = {https://doi.org/10.1002/nme.6317},
	Volume = {121},
	Year = {2020}}

@article{SurfaceStokes1,
	Author = {Bonito, Andrea and Demlow, Alan and Licht, Martin},
	Doi = {10.1137/19M1284592},
	Fjournal = {SIAM Journal on Numerical Analysis},
	Issn = {0036-1429},
	Journal = {SIAM J. Numer. Anal.},
	Mrclass = {65N30 (65N12 65N15 65N25 76D07)},
	Mrnumber = {4155235},
	Mrreviewer = {Barbara Verf\"{u}rth},
	Number = {5},
	Pages = {2764--2798},
	Title = {A divergence-conforming finite element method for the surface {S}tokes equation},
	Url = {https://doi.org/10.1137/19M1284592},
	Volume = {58},
	Year = {2020}}

@article{DemlowDziuk07,
	Author = {Demlow, Alan and Dziuk, Gerhard},
	Doi = {10.1137/050642873},
	Fjournal = {SIAM Journal on Numerical Analysis},
	Issn = {0036-1429},
	Journal = {SIAM J. Numer. Anal.},
	Mrclass = {65N30 (58J32 65N15)},
	Mrnumber = {2285862},
	Mrreviewer = {Etienne Emmrich},
	Number = {1},
	Pages = {421--442},
	Title = {An adaptive finite element method for the {L}aplace-{B}eltrami operator on implicitly defined surfaces},
	Url = {https://doi.org/10.1137/050642873},
	Volume = {45},
	Year = {2007}}

@article{Demlow09,
	Author = {Demlow, Alan},
	Doi = {10.1137/070708135},
	Fjournal = {SIAM Journal on Numerical Analysis},
	Issn = {0036-1429},
	Journal = {SIAM J. Numer. Anal.},
	Mrclass = {65N30 (65N15)},
	Mrnumber = {2485433},
	Mrreviewer = {Pascal Omnes},
	Number = {2},
	Pages = {805--827},
	Title = {Higher-order finite element methods and pointwise error estimates for elliptic problems on surfaces},
	Url = {https://doi.org/10.1137/070708135},
	Volume = {47},
	Year = {2009}}

@article{CD16,
	Author = {Cockburn, Bernardo and Demlow, Alan},
	Doi = {10.1090/mcom/3093},
	Fjournal = {Mathematics of Computation},
	Issn = {0025-5718},
	Journal = {Math. Comp.},
	Mrclass = {65N30 (58J32 65N15)},
	Mrnumber = {3522964},
	Mrreviewer = {Beny Neta},
	Number = {302},
	Pages = {2609--2638},
	Title = {Hybridizable discontinuous {G}alerkin and mixed finite element methods for elliptic problems on surfaces},
	Url = {https://doi.org/10.1090/mcom/3093},
	Volume = {85},
	Year = {2016}}

\appendix

\section{Proofs of Lemmas in Section \ref{sec-aux}}

\subsection{Proof of Lemma \ref{lem:PsiMac}}

\begin{proof}
Let $\bIct_h:\bC(\bar \Gamma_h)\to \bbR^3$ denote
the $k$th degree nodal interpolant with respect to $\bar \calT_h^{\rm ct}$,
and note that, because $\bar \calT_h^{\rm ct}$ is a refinement
of $\bar \calT_h$, there holds $\bImac_h \bIct_h = \bImac_h$.

We  write $\bPsict = \bIct_h \bp + {\bm r}$,
where ${\bm r}$ is a continuous, piecewise polynomial of degree $k$
with respect to $\bar \calT_h^{\rm ct}$. It then follows from \eqref{eqn:bPsiBounds}
that $\|{\bm r}\|_{L^\infty(\bar \Gamma_h)}\lesssim h^{k+1}$,
and so by inverse estimates, $\|{\bm r}\|_{W^{m,\infty}(\bar K)}\lesssim h^{k+1-m}$
for all $\bar K\in \bar \calT_h^{\rm ct}$ and $m\in \mathbb{N}_0$.

We then have for $m=0,1,$
\begin{align*}
|\bPsimac-\bp|_{W^{m,\infty}(\bar \Gamma_h)}
&= |\bImac_h \bPsict - \bp|_{W^{m,\infty}(\bar \Gamma_h)}\\
&\lesssim |\bImac \bPsict - \bPsict|_{W^{m,\infty}(\bar \Gamma_h)}+h^{k+1-m}\\
& \lesssim |\bImac_h (\bIct_h \bp + {\bm r}) -  (\bIct_h \bp + {\bm r})|_{W^{m,\infty}(\bar \Gamma_h)}+h^{k+1-m}\\
&\lesssim |\bImac_h{\bm r}-{\bm r}|_{W^{m,\infty}(\bar \Gamma_h)}+h^{k+1-m}.
\end{align*}
There holds $|\bImac_h {\bm r}-{\bm r}|_{W^{m,\infty}(\bar \Gamma_h)}\lesssim |{\bm r}|_{W^{m,\infty}(\bar \Gamma_h)}+ h^{1-m}|{\bm r}|_{W^{1,\infty}(\bar \Gamma_h)}\lesssim h^{k+1-m}$,
and the estimates \eqref{eqn:PsiMacApprox} now follow.

Next, we write $F_{\Tmac} = (\bImac_h \bp+\bImac_h{\bm r})\circ F_{\bar T}$, 
so that \[
|F_{\Tmac}|_{W^{m,\infty}(\hat T)}\lesssim |\bImac_h \bp+\bImac_h{\bm r}|_{W^{m,\infty}(\bar T)}|F_{\bar T}|_{W^{1,\infty}(\hat T)}^m,
\]
where we used the fact that $F_{\bar T}$ is affine.
By the stability of the Lagrange interpolant and the smoothness of $\bp$, we have
$|\bImac_h \bp|_{W^{m,\infty}(\bar T)}\lesssim 1$. We also have
by inverse estimates and and bounds of ${\bm r}$, for $m\ge 1$, $|\bImac_h {\bm r}|_{W^{m,\infty}(\bar T)}\lesssim h^{1-m} |\bImac_h {\bm r}|_{W^{1,\infty}(\bar T)}\lesssim 
h^{1-m} | {\bm r}|_{W^{1,\infty}(\bar T)}\lesssim h^{-m} \|{\bm r}\|_{L^\infty(\bar T)}\lesssim h^{k-m+1}.$
Using this estimate as well as $|F_{\bar T}|_{W^{1,\infty}(\hat T)}\lesssim h$, yields the first estimate in \eqref{eqn:FTMacBounds}:
$|F_{\Tmac}|_{W^{m,\infty}(\hat T)}\lesssim h^m_T$. The other two
estimates in the first line of \eqref{eqn:FTMacBounds} then follow from the first, 
and the second row of estimates follow from the exact same arguments.
\end{proof}

\subsection{Proof of Lemma \ref{lem:FDiff}}
\begin{proof}
We write $F_K - F_{\Kmac} = \bPsict\circ F_{\bar K} -  \bPsimac\circ F_{\bar K} = (\bPsict - \bImac_h \bPsict)\circ F_{\bar K}$ to conclude
\begin{align*}
|F_K - F_{\Kmac}|_{W^{m,\infty}(\hat T)}\lesssim |\bPsict - \bImac_h\bPsict|_{W^{m,\infty}(\bar K)} |F_{\bar K}|_{W^{1,\infty}(\hat T)}^m\lesssim h^{k+1}_K.
\end{align*}
\end{proof}

\subsection{Proof of Lemma \ref{lem:ATDiff}}
\begin{proof}
Let $C_K = \nab F_K^\intercal \nab F_K:\hat T\to \bbR^{2\times 2}$
and $C_{\Kmac} = \nab F_{\Kmac}^\intercal \nab F_{\Kmac}:\hat T\to \bbR^{2\times 2}$,
so that $J_K = \sqrt{\det(C_K)}$ and $J_{\Kmac} = \sqrt{\det(C_{\Kmac})}$.
We then use the algebraic inequality
\[
\left|a^{-\frac12} - b^{-\frac12}\right| \le \frac{|a-b|}{2(\min\{a,b\})^{\frac32}}\qquad a,b>0,
\]
the algebraic identity 
\begin{equation}\label{eqn:DeterIdentity}
\det(C_K)-\det(C_{\Kmac}) = \frac12 {\rm cof}(C_K+C_{\Kmac}):(C_K-C_{\Kmac}),
\end{equation}
and the linearity of $C\to {\rm cof}(C)$ for $2\times 2$ matrices
to get
\begin{align*}
|J_K^{-1} - J_{\Kmac}^{-1}| 
&= |(\det(C_K)^{-1/2} - \det(C_{\Kmac})^{-1/2}|\\
&\le \frac{|\det(C_K)-\det(C_{\Kmac})|}{2(\min\{\det(C_{K}),\det(C_{\Kmac})\})^{\frac32}}\\
&\lesssim \frac{(|C_T|+|C_{\Kmac}|)|C_T-C_{\Kmac}|}{(\min\{\det(C_{K}),\det(C_{\Kmac})\})^{\frac32}}.
\end{align*}
We then use $\det(C_K) \approx \det(C_{\Kmac}) \approx h^4_K$ and $|C_K|,|C_{\Kmac}| \lesssim h^2_K$
(cf.~\eqref{eqn:FTMacBounds})
to get 
\begin{align*}
|J_K^{-1} - J_{\Kmac}^{-1}| \lesssim h^{-4}_K |C_K-C_{\Kmac}|.
\end{align*}
We then write $C_K- C_{\Kmac} = \nab F_K^\intercal (\nab F_K - \nab F_{\Kmac})+(\nab F_K^\intercal - \nab F_{\Kmac}^\intercal)\nab F_{\Kmac}$
to conclude 
\begin{equation}\label{eqn:BDiff}    
\|C_K- C_{\Kmac}\|_{W^{\ell,\infty}(\hat T)}\lesssim h^{k+2}_K\qquad \ell \in \bbN_0
\end{equation}
by  \eqref{eqn:FTMacBounds} and Lemma \ref{lem:FDiff}; thus,
\begin{align}\label{eqn:JRecDiff}
\|J_K^{-1} - J_{\Kmac}^{-1}\|_{L^\infty(\hat T)} \lesssim h_K^{k-2},
\end{align}
and so, by Lemma \ref{lem:FDiff} once again,
\begin{align*}
\left\|A_K - A_{\Kmac}\right\|_{L^\infty(\hat T)}
&= \left\|\frac{\nab F_K}{J_K} - \frac{\nab F_{\Kmac}}{J_\Kmac}\right\|_{L^\infty(\hat T)}\\
&\le \left\|\frac{\nab F_K-\nab F_\Kmac}{J_K}\right\|_{L^\infty(\hat T)} +\left\| \frac{\nab F_{\Kmac}}{J_K}- \frac{\nab F_{\Kmac}}{J_\Kmac}\right\|_{L^\infty(\hat T)}\\
&\lesssim h^{k-1}_K+ \|\nab F_{\Kmac}\|_{L^\infty(\hat T)} \|J_K^{-1} - J_\Kmac^{-1}\|_{L^\infty(\hat T)}\lesssim h_K^{k-1}.
\end{align*}

Next, we compute
\begin{equation}
\label{eqn:I1I2Start}
\begin{split}
|A_K-A_{\Kmac}|_{W^{1,\infty}(\hat T)}
&\lesssim \left\|\frac{D^2 F_K}{J_K} - \frac{D^2 F_{\Kmac}}{J_\Kmac}\right\|_{L^\infty(\hat T)}
+ \left\|\frac{\nab F_K  \nab J_K}{J_K^2}-\frac{\nab F_{\Kmac} \nab J_\Kmac}{J_\Kmac^2}\right\|_{L^\infty(\hat T)}\\
&=:I_1+I_2.
\end{split}
\end{equation}
Applying Lemma \ref{lem:FDiff} and the estimates $J_K \approx J_{\Kmac} \approx h^2_K$,
we obtain $I_1\lesssim h^{k-1}_K$.

To bound $I_2$ we add and subtract terms to get
\begin{align*}
I_2
%
&\lesssim \left\|\frac{\nab F_K  \nab (J_K-J_{\Kmac})}{J_K^2}\right\|_{L^\infty(\hat T)}
+\left\|\left(\frac{\nab F_K-\nab F_{\Kmac}}{J_K^2}\right) \nab J_{\Kmac}\right\|_{L^\infty(\hat T)} +\left\|\left(\frac{\nab F_{\Kmac}}{J_K^2} -\frac{\nab F_{\Kmac}}{J_\Kmac^2}\right) \nab J_{\Kmac}\right\|_{L^\infty(\hat T)}\\
&\lesssim h^{-3}_K \| \nab (J_K - J_{\Kmac})\|_{L^\infty(\hat T)}+\left(h^{k-3}_K
+h_K \|J_K^{-2} - J_{\Kmac}^{-2}\|_{L^\infty(\hat T)}\right)
\| \nab J_{\Kmac}\|_{L^\infty(\hat T)}\\
&=:I_{2,1}+I_{2,2}.
\end{align*}
The arguments in \cite[Appendix C]{DemlowNeilan25} show $\| \nab J_{\Kmac}\|_{L^\infty(\hat T)}\lesssim h^3_K$, and we also have by \eqref{eqn:BDiff}
\begin{align*}
J_K^{-2} - J_{\Kmac}^{-2} 
&= (\det(C_K))^{-1} - (\det(C_{\Kmac}))^{-1} = 
\frac12 \frac{{\rm cof}(C_\Kmac+C_K):(C_{\Kmac}-C_T)}{\det(C_K) \det(C_{\Kmac})}
\lesssim 
h^{k-4}_K.
\end{align*}
We then conclude $I_{2,2}\lesssim h^k_K$.

Continuing, we bound $I_{2,1}$ by first applying
the chain rule:
\begin{align*}
I_{2,1} 
&= \left\| \frac{ \nab \det(C_K)}{J_K} - \frac{ \nab \det(C_{\Kmac})}{J_{\Kmac}}\right\|_{L^\infty(\hat T)}\\
&= \left\| \frac{ \nab (\det(C_K)-\det(C_{\Kmac}))}{J_K}\right\|_{L^\infty(\hat T)} 
+ \left\| \nab \det(C_{\Kmac})\big(J_K^{-1}-J_{\Kmac})\right\|_{L^\infty(\hat T)}\\
&\lesssim h^{-2}_K \| \nab (\det(C_K)-\det(C_{\Kmac}))\|_{L^\infty(\hat T)} + h_K^{k-2} \| \nab \det(C_{\Kmac})\|_{L^\infty(\hat T)},
\end{align*}
where we used \eqref{eqn:JRecDiff} in the last inequality.
Using \cite[Appendix C]{DemlowNeilan25} again, we have $ \| \nab \det(C_{\Kmac})\|_{L^\infty(\hat T)}\lesssim h^5_K$. 
Furthermore, by \eqref{eqn:DeterIdentity} and \eqref{eqn:BDiff},
\begin{align*}
\| \nab (\det(C_K)-\det(C_{\Kmac}))\|_{L^\infty(\hat T)}
&\lesssim \|C_K+C_{\Kmac}\|_{L^\infty(\hat T)} |C_K-C_{\Kmac}|_{W^{1,\infty}(\hat T)}\\
&\qquad+|C_K+C_{\Kmac}|_{W^{1,\infty}(\hat T)} \|C_K-C_{\Kmac}\|_{L^{\infty}(\hat T)}\\
&\lesssim h_K^{k+4}  + h_K^{k+2} |C_K+C_{\Kmac}|_{W^{1,\infty}(\hat T)}\lesssim h^{k+4}_K,
\end{align*}
and we conclude $I_{2,1}\lesssim h^{k+3}_K$, and so $I_2\lesssim h^k_K$.
Combining this estimate with the bound $I_1\lesssim h_K^{k-1}$ to 
\eqref{eqn:I1I2Start} completes the proof.
\end{proof}


\section{Proof of Lemma \ref{lem:PITA}}\label{App:Pert}
The proof of Lemma \ref{lem:PITA} requires a few preliminary 
results. First, for an element $K\in \calT^{\rm ct}_h$, we set
 $\Kmac =\big(\bPsimac\circ \bPsict^{-1}\big)(K)$.
 Let $a\in \calN_h^{\rm ct}$ be a Lagrange DOF on $K$
 and $\amac=\big(\bPsimac\circ \bPsict^{-1}\big)(a)$ 
 be the corresponding DOF on $\Kmac$. We then set
 \begin{equation}\label{eqn:BaDef}
 B_a:=\calM_a^K\big(\calM_{\amac}^{K^{\rm{m}}}\big)^{-1}-{\bf I}_3\in \mathbb{R}^{3\times 3},
 \quad \text{where}\quad \big(\calM_{\amac}^{K^{\rm{m}}}\big)^{-1} = \frac{1}{\bnu_{\Kmac}\cdot \bnu_{K^{\rm m}_a}} \bPi_{K^{\rm m}_a}(a^{\rm m}).
 \end{equation}
 The next result shows that $B_a$ is a perturbation of the normal projection onto $\Gamma_h$ at $a$.

\begin{lemma}\label{lem:B-Bound}
Let $B_a$ be defined by \eqref{eqn:BaDef}. There holds
    \begin{align*}
        \big|B_a+\bnu_K(a)\otimes \bnu_K(a)\big|\lesssim h^k.
    \end{align*}
\end{lemma}
\begin{proof}
    From \eqref{eqn:NormalGood} and \eqref{eqn:macNuGood}, for $\bnu_1,\bnu_2\in \{\bnu_T,\bnu_{K_a},\bnu_{\Kmac},\bnu_{K^{\rm m}_a}\},$ we have the following estimates for $h$ sufficiently small:
    \begin{align}\label{eqn:nu_bounds}
        |\bnu_1-\bnu_2|\lesssim h^k,\qquad |1-\bnu_1\cdot \bnu_2|\lesssim h^{2k},\qquad \frac{1}{\bnu_1\cdot \bnu_2}\lesssim 1.
    \end{align}
    Consequently, by \eqref{eqn:calMaK} and \eqref{eqn:BaDef} there holds
    \begin{align*}
        \big|\calM_a^K-\big({\bf I}_3-\bnu_{K_a}(a)\otimes \bnu_K(a)\big)\big|\lesssim h^{2k}, \text{ and }\big|\big(\calM_{\amac}^{\Kmac}\big)^{-1}-\bPi_{K^{\rm m}_a}(\amac)\big|\lesssim h^{2k}.
    \end{align*} 
    Using the algebraic identity $ab-cd=\frac{1}{2}(a-c)(b+d)+\frac{1}{2}(a+c)(b-d)$ and the fact that the matrices in these expressions are clearly uniformly bounded, we have
    \begin{align*}
        \big|\calM_a^K\big(\calM_{\amac}^{\Kmac}\big)^{-1}-\big({\bf I}_3-\bnu_{K_a}(a)\otimes \bnu_K(a)\big)\bPi_{K^{\rm m}_a}(\amac)\big|\lesssim h^{2k}.
    \end{align*}
    Next, we write
    \begin{align*}
    \big({\bf I}_3-\bnu_{K_a}(a)\otimes \bnu_K(a)\big)\bPi_{K^{\rm m}_a}(\amac)=\big({\bf I}_3-\bnu_{K_a}(a)\otimes \bnu_K(a)\big)\big(\bPi_{K^{\rm m}_a}(\amac)-\bPi_K(a) \big)+\bPi_K(a),
    \end{align*}
    and use \eqref{eqn:nu_bounds} to conclude
    \begin{align*}
        \big|B_a+\bnu_K(a)\otimes \bnu_K(a)\big|=\big|\calM_a^K\big(\calM_{\amac}^{\Kmac}\big)^{-1}-\bPi_K(a)\big|\lesssim h^k.
    \end{align*}
    
\end{proof}

\begin{lemma}\label{lem:AB_bound}
    Fix $K\in \calT_h^{\rm ct}$, and let $\hat B_K\in \big[\bbP_r(\hat{T})\big]^{3\times 3}$
be defined such that $\hat B_K(\hat a) = B_a$ for 
all $r$th-degree Lagrange DOFs $\hat{a}\in \hat{\calN}$,
where $a = F_K(\hat a)$ and $B_a$ is defined by \eqref{eqn:BaDef}.
Then there holds for $k\ge 2$,
    \begin{align}\label{eqn:ATDaggerB}
        |A_K^{\dagger}\hat B_K|_{W^{\ell,\infty}(\hat{T})}\lesssim h_K^{2+\ell}\qquad \ell=0,1,
    \end{align}
where $A_K^{\dagger}$ denotes the Penrose inverse of $A_K = \nab F_K/J_K$.
\end{lemma}
\begin{proof}
     Note that by \cite[Lemma 2.6]{DemlowNeilan25}, we have
    \begin{align}\label{eqn:A_est}
        |A_K|_{W^{m,\infty}(\hat{T})}\lesssim h^{m-1}_K,\qquad |A_K^{\dagger}|_{W^{m,p}(\hat{T})}\lesssim h^{m+1}_K,
    \end{align}
    and similarly for $A_{\Kmac}$.

{\bf Case} ${\bm \ell=0}$:
    Let $\hat I_h(A_K^{\dagger}\hat B_K)$ be the $r$-degree interpolant of 
    $A_K^{\dagger}\hat B_K$ with respect to the 
    Lagrange nodes $\hat \calN$.    
    Then by the Bramble-Hilbert lemma, \eqref{eqn:A_est}, and using the fact that $\hat B_K$ is a 
    polynomial of degree $\leq r$, we have
    \begin{equation}\label{eqn:TriangleAT}
    \begin{split}
    \|A_K^{\dagger}\hat B_K\|_{L^\infty(\hat{T})}
    &\leq \|\hat I_h(A_K^{\dagger}\hat B_K)\|_{L^\infty(\hat{T})}+\|A_K^{\dagger}\hat B_K-\hat I_h(A_K^{\dagger}\hat B_K)\|_{L^\infty(\hat{T})}\\ 
    &\lesssim \|\hat I_h(A_K^{\dagger}\hat B_K)\|_{L^\infty(\hat{T})}+|A_K^{\dagger}\hat B_K|_{W^{r+1,\infty}(\hat{T})}\\
    &\lesssim \|\hat I_h(A_K^{\dagger}\hat B_K)\|_{L^\infty(\hat{T})}+\sum_{\ell=0}^{r} |A_K^{\dagger}|_{W^{r+1-\ell,\infty}(\hat{T})}|\hat B_K|_{W^{\ell,\infty}(\hat{T})}\\
    &\lesssim \|\hat I_h(A_K^{\dagger}B)\|_{L^\infty(\hat{T})}+h^2_K.
    \end{split}
    \end{equation}
    It remains to bound $\|\hat I_h(A_K^{\dagger}\hat B_K)\|_{L^\infty(\hat{T})}$. To this end, we use scaling, \eqref{eqn:A_est}, Lemma \ref{lem:B-Bound}, and the identity $A_K^{\dagger}(\hat{a})\bnu_K(a)=0$ to conclude
    \begin{align*}
        \|\hat I_h(A_K^{\dagger}\hat B_K)\|_{L^\infty(\hat{T})}&\lesssim \max_{\hat{a}\in\hat{\calN}}\big|\hat I_h(A_K^{\dagger}\hat B_K)(\hat
        a
)\big|\\ &=\max_{\hat{a}\in\hat{\calN}}\big| (A_K^{\dagger}\hat B_K)(\hat{a})\big|\\ &=\max_{\hat{a}\in\hat{\calN}}\big|A_K^{\dagger}(\hat{a})B_a\big|\\ &=\max_{\hat{a}\in\hat{\calN}}\big|A_K^{\dagger}\big(\hat{a})(B_a+\bnu_K(a)\otimes \bnu_K(a)\big)\big|\lesssim h^k\|A_K^{\dagger}\|_{L^\infty(\hat{T})}\lesssim h^{k+1}_K.    \end{align*}
Combining this estimate with \eqref{eqn:TriangleAT} yields the desired result $\|A_K^\dagger \hat B_K\|_{L^\infty(\hat T)}\lesssim h^2$.

{\bf Case} {$\bm \ell=1$}: 
We apply a similar argument as the $\ell=0$ case and write
    \begin{equation}\label{eqn:W1infStart}
    \begin{split}
 |A_T^{\dagger}\hat B_K|_{W^{1,\infty}(\hat T)}
& \le  \|\hat I_h  \nab (A_K^{\dagger}\hat B_K)\|_{L^\infty(\hat T)}
+\|{\nabla}(A_K^{\dagger}\hat B_K)-\hat I_h({\nabla}(A_K^{\dagger}\hat B_K))\|_{L^\infty(\hat{T})}\\
&\lesssim \|\hat I_h  \nab (A_K^{\dagger}\hat B_K)\|_{L^\infty(\hat T)}+\sum_{\ell=0}^r |A_K^{\dagger}|_{W^{r+2-\ell,\infty}(\hat{T})}|\hat B_K|_{W^{\ell,\infty}(\hat{T})}\\
&\lesssim 
\|\hat I_h  \nab (A_K^{\dagger}\hat B_K)\|_{L^\infty(\hat T)}+h^3_K,
    \end{split} 
    \end{equation}
  where we used \eqref{eqn:A_est} in the last step.   
    Thus, it suffices to estimate $\|\hat I_h({\nabla}(A_K^{\dagger}\hat B_K))\|_{L^\infty(\hat{T})}$.

    Set $\hat C_K=(\bnu_K\circ F_K)\otimes (\bnu_K\circ F_K)$ and note that $A_K^{\dagger}\hat C_K=0$ on $\hat{T}$. 
    Consequently $ \nabla (A_K^{\dagger}\hat C_K)=0$, and so 
    \begin{align*}
        |{\nabla}(A_K^{\dagger}\hat B_K)(\hat{a})|
        &=\big|{\nabla}\big(A_K^{\dagger}(\hat B_K+\hat C_K)\big)(\hat{a})\big|\\ 
        &\lesssim |A_K^{\dagger}|_{W^{1,\infty}(\hat{T})}|(\hat B_K+\hat C_K)(\hat{a})|+\|A_K^{\dagger}\|_{L^\infty(\hat{T})}\big|{\nabla}\big(\hat B_K+\hat C_K\big)(\hat{a})\big|\qquad \forall \hat a\in \hat \calN.
    \end{align*}
    Using \eqref{eqn:A_est} and Lemma \ref{lem:B-Bound} yields
    \begin{align}\label{eqn:AB_bound2}
        |{\nabla}(A_K^{\dagger}\hat B_K)(\hat{a})|\lesssim h^{k+2}+h\big|{\nabla}\big(\hat B_K+\hat C_K\big)(\hat{a}
        )\big|.
    \end{align}
    We then add and subtract the $r$th degree interpolant of $\hat C_K$:
    \begin{align}\label{eqn:nabBPC}
        \big|{\nabla}\big(\hat B_K+\hat C_K\big)(\hat{a})\big|
        \leq \big| {\nabla}\big(\hat B_K+\hat I_h \hat C_{\hat K}\big)(\hat{a})\big|+\big|{\nabla}\big(\hat C_K-\hat I_h \hat C_K\big)(\hat{a})\big|. 
    \end{align}
    To bound the first term in \eqref{eqn:nabBPC}, we use an equivalence of norms argument and Lemma \ref{lem:B-Bound}:
    \begin{align}\label{eqn:J1Bound}
        \big|{\nabla}\big(\hat B_K+\hat I_h \hat C_K\big)(\hat{a})\big|
        \lesssim \|\hat B_K+\hat I_h \hat  C_K\|_{L^\infty(\hat{T})}\lesssim \max_{\hat{a}\in\hat{\calN}}\big|\big(\hat B_K+\hat I_h \hat C_K\big)(\hat{a})\big|\lesssim h^k_K.
    \end{align}
    To bound the second term in \eqref{eqn:nabBPC}, we apply the Bramble-Hilbert lemma, scaling,
    and the boundedness of Sobolev norms of the discrete outward unit normal:
    \begin{equation}\label{eqn:J2Bound}
    \begin{split}
        \big| {\nabla}\big(\hat C_K-\hat I_h \hat C_K\big)(\hat{a})\big|
        &\leq |\hat C_K-\hat I_h \hat C_K|_{W^{1,\infty}(\hat{T})}\\
        &\lesssim |\hat C_K|_{W^{r+1,\infty}(\hat{T})}\lesssim |\bnu_K\circ F_K|^2_{W^{r+1,\infty}(\hat{T})}\lesssim h^{r+1}\|\bnu_K\|_{W^{r+1}(K)}\lesssim h^{r+1}_K.
    \end{split}
    \end{equation}
    We 
    apply the estimates \eqref{eqn:J1Bound}--\eqref{eqn:J2Bound} towards \eqref{eqn:nabBPC} 
    to conclude
    \begin{align*}
        \big|{\nabla}\big(\hat B_K+\hat C_K\big)(\hat{a})\big|\lesssim h^k_K+h^{r+1}_K.
    \end{align*}
    Applying this estimate to \eqref{eqn:AB_bound2} gets
    \begin{align*}
        \big| {\nabla}\big(A_K^{\dagger}\hat B_K\big)(\hat{a})\big|\lesssim h^{k+1}_K+h^{r+2}_K\qquad \forall\hat{a}\in\hat{\calN},
    \end{align*} 
    and therefore
    \begin{align}\label{eqn:LastDO}
        \|\hat I_h({\nabla}(A_K^{\dagger}\hat B_K))\|_{L^\infty(\hat{T})}\lesssim \max_{\hat{a}\in\hat{\calN}}\big| \hat{\nabla}\big(A_K^{\dagger}\hat B_K \big)(\hat{a})\big|\lesssim h^{k+1}_K+h^{r+2}_K\lesssim h^3_K,
    \end{align}
for $k\ge 2$. Applying \eqref{eqn:LastDO} to \eqref{eqn:W1infStart} yields
the desired result \eqref{eqn:ATDaggerB} in the case $\ell=1$, which concludes
the proof.
\end{proof}

\subsection{Proof of Lemma \ref{lem:PITA}}

\begin{proof}
    Let $\bvmac\in\bVmac_h$ be arbitrary and define $\bv\in\bV_h$ such that $\bv_{K_a}(a)=\bvmac_{K^{\rm m}_{a^{\rm m}}}(\amac)$ at Lagrange nodes $a\in \calN_h^{\rm ct}$, where $\amac = (\bPsimac\circ \bPsict^{-1})(a)$.

    Define $\bar{\bv},\bar{\bv}^{\rm m}:\bar\Gamma_h\to \bbR^3$ via composition:
\[
\bar{\bv}=\bv \circ \bPsict,\quad \text{and}\quad
\overbvmac=\bvmac \circ \bPsimac.
\]
Let $\bar{K}\in\bar{\calT}^{\rm ct}_h$ be arbitrary and set $K=\bPsict(\bar{K}), \Kmac=\bPsimac(\bar{K})$.
Note  that as 
\begin{align*}
        \bv_{K_a}(a)=\bvmac_{K^{\rm m}_{a^{\rm m}}}(\amac), \qquad \bv_{K}(a)=\calM_a^K(v_{K_a}(a)),\qquad \text{and  }\quad\bvmac_{\Kmac}(\amac)=\calM_{\amac}^{\Kmac}(\bvmac_{K^{\rm m}_{a^{\rm m}}}(\amac)),
    \end{align*} 
    there holds 
    \begin{align}\label{eqn:AssignmentOnBar}
        \bar{\bv}_{\bar{K}}(\bar{a})=\calM_a^K\big(\calM_{\amac}^{\Kmac} \big)^{-1}\overbvmac_{\bar{K}}(\bar{a})
        = \overbvmac_{\bar K}(\bar a)+\hat B_K(\hat a) \overbvmac_{\bar K}(\bar a)
\end{align} 
with $\bar a = \bPsict^{-1}(a) = \bPsimac^{-1}(a^{\rm m})$
and $B_a$ is defined by \eqref{eqn:BaDef}.

By definition of $\bV_h$ and $\bVmac_h$, we can write $\bv|_K\circ F_K=A_K\hat{\bv}$ and $\bvmac\circ F_{K^{\rm m}}=A_{\Kmac}\bvhatmac$ for some $\hat{\bv},\bvhatmac\in [\bbP_r(\hat{T})]^2$. Thus, we have
  \begin{align*}
        \bar{\bv}|_{\bar K}\circ F_{\bar K}=A_K\hat{\bv},\qquad \overbvmac|_{\bar K}\circ F_{\bar K}=A_{\Kmac}\bvhatmac,
    \end{align*}
and so by \eqref{eqn:AssignmentOnBar},

        \begin{align}\label{eqn:AssignmentOnHat}
        \hat{\bv}(\hat{a})=A_K^{\dagger}(\hat{a})({\bf I}_3+\hat B_K(\hat{a}))A_{\Kmac}(\hat{a})\bvhatmac(\hat{a}).
    \end{align}

    \textbf{Part 1:} We first establish the following estimate:
    \begin{align}\label{eqn:PITAStep1}
        \big| \hat{\bv}-A_K^{\dagger}({\bf I}_3+\hat B_K)A_{\Kmac}\bvhatmac\big|_{H^\ell(\hat{T})}^2 \les h_K^4\|\bvmac \|_{H^1(\Kmac)}^2,\qquad \ell=0,1.
    \end{align}
    Note that ${\bf I}_3+\hat B_K\in [\bbP_r(\hat{T})]^{3\times 3}$ and $\|{\bf I}_3+\hat B_K\|_{L^\infty(\hat{T})}\les 1$. Also, see that for $\hat{a}\in\hat{\calN}$, by definition of 
    $\bV_h$, $\bVmac_h$ and construction of $\bv$ there holds 
    \begin{align*}
        \hat{\bv}(\hat{a})=A_K^{\dagger}(\hat{a})(\bI_3+B(\hat{a}))A_{\Kmac}(\hat{a})\bvhatmac(\hat{a}).
    \end{align*}
    Applying the Bramble-Hilbert lemma, the product rule, a change of variables, an inverse estimate, and \eqref{eqn:A_est} we have 
\begin{equation}
\label{eqn:FirstBHH}
\begin{split}
        \big| \hat{\bv}-A_K^{\dagger}({\bf I}_3+\hat B_K) A_{\Kmac}\bvhatmac\big|_{H^\ell(\hat{T})}^2
        &\lesssim \big|A_K^{\dagger}({\bf I}_3+\hat B_K)A_{\Kmac}\bvhatmac\big|_{H^{r+1}(\hat{T})}^2\\
&\lesssim \sum_{n=0}^{r+1}\big|A_K^{\dagger}({\bf I}_3+\hat B_K)\big|_{W^{n,\infty}(\hat{T})}^2\big| A_{\Kmac}\bvhatmac\big|_{H^{r+1-n}(\hat{T})}^2\\
&\lesssim h^2_K\sum_{n=0}^{r+1} h^{2n}_K \big|A_{\Kmac}\bvhatmac\big|_{H^{r+1-n}(\hat{T})}^2\\ 
&\les h^2_K\big|A_{\Kmac}\bvhatmac
\big|_{H^{r+1}(\hat{T})}^2+\sum_{n=1}^{r+1}h_K^{2r+2}\|\bvmac \|_{H^{r+1-n}(\Kmac)}^2\\ 
&\les h^2_K\big| A_{\Kmac}\bvhatmac\big|_{H^{r+1}(\hat{T})}^2+h^4_K\|\bvmac\|_{H^1(\Kmac)}^2. 
\end{split}
\end{equation}
It remains to bound $\big|A_{\Kmac}\bvhatmac\big|_{H^{r+1}(\hat{T})}^2$. Since $\bvhatmac$ is a piecewise polynomial of degree $r$, we can use the product rule, a change of variables, \eqref{eqn:A_est}, and an inverse estimate to find
\begin{equation}
\label{eqn:SecondBHH}
\begin{split}
    \big|A_{\Tmac}\bvhatmac\big|_{H^{r+1}(\hat{T})}
    &\leq \sum_{n=0}^{r+1}|A_{\Kmac}|_{W^{r+1-n,\infty}(\hat{T})}|\bvhatmac|_{H^{n}(\hat{T})}
    =\sum_{n=0}^r|A_{\Kmac}|_{W^{r+1-n,\infty}(\hat{T})}|\bvhatmac|_{H^{n}(\hat{T})}\\ 
    &\les \sum_{n=0}^r h^{r-n}_K|\bvhatmac|_{H^{n}(\hat{T})}
    =\sum_{n=0}^{r}h^{r-n}_K\big|A_{\Kmac}^{\dagger}A_{\Kmac}\bvhatmac\big|_{H^{n}(\hat{T})}\\
    &\les \sum_{n=0}^r h^{r-n}_K\sum_{j=0}^{n}|A_{\Kmac}^{\dagger}|_{W^{n-j,\infty}(\hat{T})}\big| A_{\Kmac}\bvhatmac\big|_{H^j(\hat{T})}\\
    &\les h^r_K\sum_{j=0}^r \|\bvmac\|_{H^j(\Kmac)}\les h^r\|\bvmac \|_{H^r(\Kmac)}\les h_K\|\bvmac\|_{H^1(\Kmac)}.
\end{split}
\end{equation}
Combining \eqref{eqn:FirstBHH}--\eqref{eqn:SecondBHH} yields \eqref{eqn:PITAStep1}.

\textbf{Part 2: }We will now derive the estimate for $\|\bar{\bv}-\overbvmac\|_{L^2(\bar K)}$.
    By a change of variables, we have
    \begin{align}\label{eqn:L2barStart}
        \|\bar {\bv}-\overbvmac\|_{L^2(\bar{K})}^2\approx h^2_K\int_{\hat{T}}\big| A_K\hat{\bv}-A_{\Kmac}\bvhatmac\big|^2.
    \end{align} 
    Then, using the triangle inequality, \eqref{eqn:A_est}, and Lemma \ref{lem:ATDiff} yields
    \begin{equation}\label{eqn:L2barStep2}
    \begin{split}
        h^2_K\int_{\hat{T}}\big| A_K\hat{\bv}-A_{\Kmac}\bvhatmac\big|^2
        &\lesssim h^2_K\int_{\hat{T}}\big|A_K(\hat{\bv}-\bvhatmac)\big|^2+h^2_K\int_{\hat{T}}\big|(A_{\Kmac}-A_K)\bvhatmac\big|^2\\ 
        &\les \int_{\hat{T}}|\hat{\bv}-\bvhatmac|^2+h^2_K\|A_{\Kmac}-A_K\|_{L^\infty(\hat{T})}^2\int_{\hat{T}}|\bvhatmac|^2\\ 
        &\les \int_{\hat{T}}|\hat{\bv}-\bvhatmac|^2+h^{2k}_K\int_{\hat{T}}\big| A_{\Kmac}^{\dagger}A_{\Kmac}\bvhatmac\big|^2\\
        &\les \int_{\hat{T}}|\hat{\bv}-\bvhatmac|^2+h^{2k+2}_K\int_{\hat{T}}\big|A_{\Kmac}\bvhatmac\big|^2\les \int_{\hat{T}}|\hat{\bv}-\bvhatmac|^2+h^{2k}_K\|\bvmac\|_{L^2(\Kmac)}.
    \end{split}
    \end{equation}
    Define $\hat B_K$ as in Lemma \ref{lem:AB_bound}. Then, using the triangle inequality and then Part 1, Lemma \ref{lem:AB_bound}, and Lemma \ref{lem:ATDiff},
 \begin{equation}\label{eqn:L2barStartStep3}
    \begin{split}
        \int_{\hat{T}}|\hat{\bv}-\bvhatmac|^2
        &\les \int_{\hat{T}}\big| \hat{\bv}-A_K^{\dagger}({\bf I}_3+\hat B_K)A_{\Kmac}\bvhatmac\big|^2+\int_{\hat{T}}\big|A_K^{\dagger}\hat B_K A_{\Kmac}\bvhatmac\big|^2
        +\int_{\hat{T}}\big|({\bf I}_3-A_K^{\dagger}A_{\Kmac})\bvhatmac\big|^2\\ 
        &\les h^4_K \|\bvmac\|_{H^1(\Kmac)}^2+\|A_K^{\dagger}B\|_{L^\infty(\hat{T})}^2\int_{\hat{T}}\big|A_{\Kmac}\bvhatmac\big|^2
        +\|A_K^{\dagger}(A_K-A_{\Kmac})\|_{L^\infty(\hat{T})}^2\int_{\hat{T}}|\bvhatmac|^2\\ 
        &\les h^4_K\|\bvmac\|_{H^1(\Kmac)}^2+h^2_K\|\bvmac\|_{L^2(\Kmac)}^2+h^{2k}_K\|\bvmac\|_{L^2(\Kmac)}^2
        \lesssim h^2_K \|\bvmac\|_{H^1(\Kmac)}^2.
        \end{split}
    \end{equation}
    Combining this estimate with \eqref{eqn:L2barStart}--\eqref{eqn:L2barStartStep3} yields
    \begin{align*}
        \|\bar{\bv}-\overbvmac\|_{L^2(\bar{K})}^2\les h^2\|\bvmac\|_{H^1(\Kmac)}^2.
    \end{align*}

\textbf{Part 3: }Next, we establish a similar result in the $H^1$ semi-norm.
We use the chain rule and a change of variables
to conclude
\begin{align*}
    |\overline{\bv}-\overbvmac|_{H^1(\bar{K})}^2
    &\les \int_{\hat{T}}\big|{\nabla}(A_K\hat{\bv})-{\nabla}(A_{\Kmac}\bvhatmac)\big|^2\\
    &\les \int_{\hat{T}}\big|{\nabla}
    (A_K(\hat{\bv}-\bvhatmac))\big|^2+\int_{\hat{T}}\big|{\nabla}((A_K-A_{\Kmac})\bvhatmac)\big|^2=:J_1+J_2.
\end{align*}
To bound $J_1$, we first use the product rule, \eqref{eqn:L2barStartStep3}, 
and \eqref{eqn:A_est}:
\begin{equation}\label{eqn:J1Start}
\begin{split}
    J_1
    &\les \|{\nabla}A_K\|_{L^\infty(\hat{T})}^2\int_{\hat{T}}|\hat{\bv}-\bvhatmac|^2+\| A_K\|_{L^\infty(\hat{T})}^2\int_{\hat{T}}|{\nabla}(\hat{\bv}-\bvhatmac)|^2\\
    &\les h^2_K\|\bvmac\|_{H^1(\Kmac)}^2+h^{-2}_K\int_{\hat{T}}|{\nabla}(\hat{\bv}-\bvhatmac)|^2.
\end{split}
\end{equation}
With a similar chain of inequalities as above, we have
\begin{align*}
    h^{-2}_K\int_{\hat{T}}|{\nabla}(\hat{\bv}-\bvhatmac)|^2&\les h^{-2}_K\int_{\hat{T}}\big|{\nabla}(\hat{\bv}-A_K^{\dagger}({\bf I}_3+\hat B_K)A_{\Kmac}\bvhatmac)\big|^2+h^{-2}_K\int_{\hat{T}}\big|{\nabla}(A_K^{\dagger}\hat B_K A_{\Kmac}\bvhatmac)\big|^2\\
    &\quad +h^{-2}_K\int_{\hat{T}}\big|{\nabla}(({\bf I}_3-A_K^{\dagger}A_{\Kmac})\bvhatmac)\big|^2=:L_1+L_2+L_3.
\end{align*}
The first term can be bounded from Part 1:
\begin{align*}
    L_1\les h^2_K\|\bvmac\|_{H^1(\Tmac)}^2.
\end{align*} 

Next, with the product rule, Lemma \ref{lem:AB_bound}, and a change of variables,
\begin{align*}
    L_2
    &\les h^{-2}_K\|{\nabla}(A_K^{\dagger}\hat B_K)\|_{L^\infty(\hat{T})}^2\int_{\hat{T}}|A_{\Kmac}\bvhatmac|^2+h^{-2}_K\|A_K^{\dagger}\hat B_K\|_{L^\infty(\hat
    {T})}^2\int_{\hat{T}}|{\nabla}(A_{\Kmac}\bvhatmac)|^2\\ &\les h^2_K\|\bvmac\|_{L^2(\Kmac)}^2+h^2_K\|\bvmac\|_{H^1(\Kmac)}^2.
\end{align*}

To bound $L_3$, we first apply the product rule:
\begin{align*}
    L_3&\les h^{-2}_K\|{\nabla}({\bf I}_3-A_K^{\dagger} A_{\Kmac})\|_{L^\infty(\hat{T})}^2\int_{\hat{T}}|\bvhatmac|^2
    +h^{-2}_K\|{\bf I}_3-A_K^{\dagger}A_{\Kmac}\|_{L^\infty(\hat{T})}^2\int_{\hat{T}}|{\nabla}\bvhatmac|^2.
\end{align*}
Note that with Lemmas \ref{lem:ATDiff} and \ref{lem:AB_bound}, 
\begin{align*}
    \|{\bf I}_3-A_K^{\dagger}A_{\Kmac}\|_{L^\infty(\hat{T})}^2\les \|A_K^{\dagger}\|_{L^\infty(\hat{K})}^2\|A_K-A_{\Kmac}\|_{L^\infty(\hat{T})}^2
    \les h^{2k}_K,
    \end{align*}
    and
    \begin{align*}
    \|{\nabla}({\bf I}_3-A_K^{\dagger}A_{\Kmac})\|_{L^\infty(\hat{T})}^2
    &\les \|{\nabla}A_K^{\dagger}\|_{L^\infty(\hat{T})}^2\|A_K-A_{\Kmac}\|_{L^\infty(\hat{T})}^2+\|A_K^{\dagger}\|_{L^\infty(\hat{T})}^2\|{\nabla}(A_K-A_{\Kmac})\|_{L^\infty(\hat{T})}^2\\&\les h_K^{2k}.
\end{align*}
Also, with \eqref{eqn:A_est},
\begin{align*}
    \int_{\hat{T}}|{\nabla}\bvhatmac|^2
    &\les \|A_{\Kmac}^{\dagger}\|_{L^\infty(\hat{T})}^2\int_{\hat{T}}|\nabla
(A_{\Kmac}\bvhatmac)|^2+\|{\nabla}A_{\Kmac}^{\dagger}\|_{L^\infty(\hat{T})}^2\int_{\hat{T}}|A_{\Kmac}\bvhatmac|^2\\
&\les h^2_K\|\bvmac\|_{H^1(\Kmac)}^2. 
\end{align*}
With these estimates, we have
\begin{align*}
    L_3\les h^{2k}_K\|\bvmac\|_{H^1(\Kmac)}^2.
\end{align*}
Combining the estimates for $L_1$, $L_2$, and $L_3$, we obtain $J_1\les h^2_K\|\bvmac\|_{H^1(\Kmac)}^2$.

For $J_2$, we can use the product rule and Lemma \ref{lem:ATDiff} to find
\begin{align*}
    J_2&\les \|{\nabla}(A_K-A_{\Kmac})\|_{L^\infty(\hat{T})}^2\int_{\hat{T}}|\bvhatmac|^2+\|A_K-A_{\Kmac}\|_{L^\infty(\hat{T})}^2\int_{\hat{T}}|{\nabla}\bvhatmac|^2\\ &\les h^2_K\|\bvmac\|_{H^1(\Kmac)}^2.
\end{align*}
Combining the bounds of $J_1$ and $J_2$ yields
\begin{align*}
    |\bar{\bv}-\overbvmac|_{H^1(\bar{K})}^2\les h^2_K\|\bvmac\|_{H^1(\Kmac)}^2.
\end{align*}

For the next claim, define $\tilde{q}:=q\circ \bPsict\circ \bPsimac^{-1}$ so that $q\circ \bPsict|_{\overline{T}}=\tilde{q}\circ \bPsimac|_{\overline{T}}$ for all $\overline{T}\in\overline{\Gamma}_h$. To satisfy the mean zero constraint, let $\overline{c}:= \frac{1}{|\Gammamac_h|}\int_{\Gammamac_h}\tilde{q}$, $q^{\rm m}=\tilde{q}-\overline{c}$, and note $q^{\rm m}\in \Qmac$. Now, there holds with a change of variables and the mean zero property of $q$:
\begin{align*}
    \|q\circ \bPsict-q^{\rm m}\circ \bPsimac\|_{L^2(\overline{\Gamma}_h)}^2
    &=\|\overline{c}\|_{L^2(\overline{\Gamma}_h)}^2\\ &= \frac{|\overline{\Gamma}_h|}{|\Gammamac_h|^2}\left|\sum_{K^{\rm m}\in\calT_h^{\rm m,ct}}\int_{K^{\rm m}}q\circ \bPsict\circ \bPsimac^{-1}\right|^2\\ 
    &\lesssim \left|\sum_{K\in \calT_h^{\rm ct}}\int_{K}q(J_{F_{K^m}\circ F_K^{-1}}-1)\right|^2\\
    &\le \left|\sum_{K\in \calT_h^{\rm ct}} h_K^{1/2}\|q\|_{L^2(K)} \|(J_{F_{K^m}}\circ F_K^{-1})(J_{F_K^{-1}})-1\|_{L^\infty(K)}\right|^2,
\end{align*}
where $J_{F_{K^m}\circ F_K^{-1}}$ is the product of the nonzero singular values of $\nabla (F_{K^m}\circ F_K^{-1})$,
and we applied the chain rule in the last inequality.

Continuing, we apply the Cauchy-Schwarz inequality
to conclude
\begin{equation}\label{eqn:CSSq}
\begin{split}
    \|q\circ \bPsict-q^{\rm m}\circ \bPsimac\|_{L^2(\overline{\Gamma}_h)}^2
    &\lesssim 
    \left(\sum_{K\in \calT_h^{\rm ct}} h_K^2\right)\left(\sum_{K\in \calT_h^{\rm ct}} h_K^{-1}\|(J_{F_{K^m}}\circ F_K^{-1})(J_{F_K^{-1}})-1\|_{L^\infty(K)}^2\|q\|_{L^2(K)}^2  \right)\\
    &\lesssim 
    \sum_{K\in \calT_h^{\rm ct}} h_K^{-1}\|(J_{F_{K^m}}\circ F_K^{-1})(J_{F_K^{-1}})-1\|_{L^\infty(K)}^2\|q\|_{L^2(K)}^2.
\end{split}
\end{equation}

We then write
  \begin{equation}\label{eqn:JDiffers}
      |(J_{F_{K^m}}\circ F_K^{-1})(J_{F_K^{-1}})-1|
  =|(J_{F_{K^m}}-J_{F_K})\circ F_{K^{-1}}|\cdot |J_{F_K^{-1}}|
  \lesssim h^{-2}_K |(J_{F_{K^m}}-J_{F_K})\circ F_{K^{-1}}|. 
  \end{equation}
We then proceed as in the proof of Lemma \ref{lem:ATDiff}, but using the algebraic inequality
\[|\sqrt{a}-\sqrt{b}|\leq \frac{|a-b|}{2\min\{\sqrt{a},\sqrt{b}\}}\qquad a,b>0. 
\]
Letting $C_K=\nab F_{K}^\intercal \nab F_K$ and $C_{K^m}=\nab F_{K^m}^\intercal \nab F_{K^m}$, there holds with \eqref{eqn:BDiff} and Lemma \ref{lem:PsiMac}
\begin{align*}|(J_{F_{K^m}}-J_{F_K})\circ F_{K^{-1}}|&\les \frac{|\det C_K-\det C_{K^m}|}{\min\{J_{F_{K}},J_{F_{K^m}}\}}\\&\les \frac{(|C_K|+|C_{K^m}|)|C_K-C_{K_m}|}{h^2_K} \\ &\les h_K^{k+2}.  \end{align*}
Applying this estimate and \eqref{eqn:JDiffers} to \eqref{eqn:CSSq} yields
 \[\|q\circ \bPsict-q^{\rm m}\circ \bPsimac\|_{L^2(\overline{\Gamma}_h)}\les h^{k-1/2}\|q\|_{L^2(\Gamma_h)} \le h \|q\|_{L^2(\Gamma_h)},
 \]
 where we used $k\ge 2$ in the last equality.
\end{proof}

\end{document}